\documentclass[letterpaper]{article} 
\usepackage{aaai24}  
\usepackage{times}  
\usepackage{helvet}  
\usepackage{courier}  
\usepackage[hyphens]{url}  
\usepackage{graphicx} 
\usepackage[sort]{natbib}  
\usepackage{caption} 
\usepackage{algorithm}
\usepackage{algorithmic}

\usepackage{amsmath}
\usepackage{amsthm}
\usepackage{booktabs}
\usepackage{mathtools}
\usepackage{amssymb}

\DeclarePairedDelimiter\ceil{\lceil}{\rceil}

\newcommand\norm[1]{\left\lVert#1\right\rVert}

\def\FastMix{\texttt{FastMix}}

\theoremstyle{plain}
\newtheorem{theorem}{Theorem}[section]

\newtheorem{lemma}[theorem]{Lemma}
\newtheorem{corollary}[theorem]{Corollary}
\theoremstyle{definition}
\newtheorem{definition}[theorem]{Definition}

\theoremstyle{remark}
\newtheorem{remark}[theorem]{Remark}

\usepackage{amsmath}\allowdisplaybreaks
\usepackage{amsfonts,bm}
\usepackage{amssymb}

\def\eqref#1{equation~(\ref{#1})}

\def\ceil#1{\left\lceil #1 \right\rceil}

\def\1{\bf{1}}

\def\vq{{\bf{q}}}

\def\vs{{\bf{s}}}

\def\vv{{\bf{v}}}
\def\vw{{\bf{w}}}
\def\vx{{\bf{x}}}
\def\vy{{\bf{y}}}

\def\fO{{\mathcal{O}}}

\def\BE{{\mathbb{E}}}

\def\BR{{\mathbb{R}}}

\def\mA {{\bf A}}

\DeclareMathOperator*{\argmin}{arg\,min}

\usepackage{amsthm}
\theoremstyle{plain}

\makeatletter
\def\Ddots{\mathinner{\mkern1mu\raise\p@
\vbox{\kern7\p@\hbox{.}}\mkern2mu
\raise4\p@\hbox{.}\mkern2mu\raise7\p@\hbox{.}\mkern1mu}}
\makeatother

\makeatletter
\newcommand*{\rom}[1]{\expandafter\@slowromancap\romannumeral #1@}
\makeatother

\usepackage{newfloat}
\usepackage{listings}
\DeclareCaptionStyle{ruled}{labelfont=normalfont,labelsep=colon,strut=off} 
\floatstyle{ruled}
\newfloat{listing}{tb}{lst}{}
\floatname{listing}{Listing}

\def\BE{{\mathbb{E}}}
\def\BR{{\mathbb{R}}}
\def\fO{{\mathcal{O}}}
\def\vq{{\bf{q}}}
\def\vx{{\bf{x}}}
\def\vs{{\bf{s}}}
\def\vv{{\bf{v}}}
\def\vw{{\bf{w}}}
\def\vy{{\bf{y}}}

\title{Decentralized Sum-of-Nonconvex Optimization}
\author{
    Zhuanghua Liu\textsuperscript{\rm 1, \rm2},
    Bryan Kian Hsiang Low\textsuperscript{\rm 1}
}
\affiliations{
    \textsuperscript{\rm 1}Department of Computer Science, National University of Singapore\\
    \textsuperscript{\rm 2}CNRS@CREATE LTD, 1 Create Way, \#08-01 CREATE Tower, Singapore 138602\\
    liuzhuanghua9@gmail.com, lowkh@comp.nus.edu.sg
}

\usepackage{bibentry}

\begin{document}

\maketitle

\begin{abstract}
We consider the optimization problem of minimizing the sum-of-nonconvex function, i.e., a convex function that is the average of nonconvex components. 
The existing stochastic algorithms for such a problem only focus on a single machine and the centralized scenario.
In this paper, we study the sum-of-nonconvex optimization in the decentralized setting. 
We present a new theoretical analysis of the PMGT-SVRG algorithm for this problem and prove the linear convergence of their approach.
However, the convergence rate of the PMGT-SVRG algorithm has a linear dependency on the condition number, which is undesirable for the ill-conditioned problem. 
To remedy this issue, we propose an accelerated stochastic decentralized first-order algorithm by incorporating the techniques of acceleration, gradient tracking, and multi-consensus mixing into the SVRG algorithm.
The convergence rate of the proposed method has a square-root dependency on the condition number.
The numerical experiments validate the theoretical guarantee of our proposed algorithms on both synthetic and real-world datasets.
\end{abstract}

\section{Introduction}

The exponential growth of data in the past decades has sparked substantial interest in developing algorithms distributed over multiple agents. A common scenario is that each agent within some network topology owns a disjoint subset of data, and they collaborate to tackle a global optimization objective. The network topology in which each agent resides can be classified into two categories: client-server  vs.~decentralized settings. 
For the former setting, a central parameter server communicates with all the workers and aggregates the information collected from them \cite{li2014scaling}. When there is a large volume of data on each agent, the central server becomes the bottleneck in the whole network. 
For the latter setting, each agent only communicates with its direct neighbors to exchange their information and finish the global task \cite{lian2017can}. 

This paper focuses on stochastic optimization for minimizing the sum-of-nonconvex objective function in the decentralized setting.
We formulate our problem as 
a convex optimization problem collaboratively solved by $m$ agents in the network. 
Consider the following composite optimization objective function:
\begin{equation}\label{obj}
\min_{x\in \BR^d} F(x) \coloneqq  f(x) + \psi(x),\quad f(x)  \coloneqq  \frac{1}{m}\sum_{i=1}^m f_i(x)
\end{equation}
where $f(x)$ is convex and smooth, and $\psi(\cdot)$ is convex but possibly non-smooth (e.g., $\ell_1$-regularization term). 
We suppose there are $m$ agents and the $i$-th agent stores the local objective function $f_i(x)$ which can be written as a finite-sum form:
\begin{equation*}
f_i(x) \coloneqq \frac{1}{n}\sum_{j=1}^n f_{i, j}(x)
\end{equation*}
where $n$ is the number of components and each $f_{i, j}(x)$ is smooth but possibly non-convex.
The sum-of-nonconvex optimization is common in real-world applications, including 1-PCA \cite{saad2011numerical}, $k$-PCA \cite{allen2016lazysvd}, online eigenvector problem \cite{allen2017follow}, and general nonconvex optimization \cite{agarwal2016finding,carmon2018accelerated}.

Existing decentralized first-order optimization algorithms suffer several limitations in solving Problem (\ref{obj}): 
(i) The decentralized deterministic algorithms, such as EXTRA \cite{shi2015extra}, Exact-Diffusion \cite{yuan2018exact}, P2D2 \cite{alghunaim2019linearly}, and SONATA \cite{sun2019convergence}, need access to the full gradient at each round. 
The computational cost of each iteration is prohibitively expensive on massive datasets.
(ii) While existing decentralized stochastic methods achieve a cheaper per-iteration cost by sampling a minibatch of samples, the theoretical analysis of these approaches is not specialized in the sum-of-nonconvex optimization problem. 
One class of methods \cite{xin2020variance,shi2015extra,ye2020pmgt} assumes that all component functions $f_{i,j}(\cdot)$ are convex such that the global objective function $f(\cdot)$ is convex\footnote{\citet{ye2020pmgt} have claimed that each component function can be possibly nonconvex. However, Assumption~$1$ in their work requires each component function to be both $L$-smooth and convex, otherwise Eq. (3) cannot be satisfied.}. 
Convergence analysis of these works does not apply to Problem (\ref{obj}) due to the mismatch of problem assumptions.
The other class of methods \cite{xin2022fast,li2022destress,luo2022optimal} assumes that component functions $f_{i,j}(\cdot)$ are nonconvex and the global objective function $f(\cdot)$ is also possibly nonconvex. 
Consequently, the rate achieved by these methods is not optimal for Problem~(\ref{obj}).

In this paper, we intend to design communication- and computation-efficient optimization algorithms for the sum-of-nonconvex problem in the decentralized setting. 
We start by presenting a new theoretical analysis of the PMGT-SVRG proposed by \citet{ye2020pmgt} for objective function (\ref{obj}). 
It achieves the stochastic first-order oracle (SFO) complexity of $\fO((n + \sqrt{n} \kappa) \log (1/ \epsilon))$ and communication complexity of $\tilde{O}((\sqrt{n} + \kappa) \xi \log (1/\epsilon))$ where $\kappa$ is the condition number and $\xi$ is some constant depending on the underlying network structure. 
Notice that both the computational and communication complexities of PMGT-SVRG have a linear dependency on the condition number $\kappa$, which can be exceptionally expensive when the objective function is ill-conditioned.

To remedy this issue, we propose an accelerated stochastic variance-reduced proximal-gradient optimization method called PMGT-KatyushaX for Problem~(\ref{obj}) to improve the dependency of complexities on the condition number.
Specifically, the vanilla KatyushaX algorithm proposed by \citet{allen2018katyusha} achieves the SFO complexity with a square-root dependency on the condition number on a single machine.
To extend the KatyushaX to the decentralized setting, we incorporate the powerful ideas of acceleration \cite{allen2014linear}, gradient tracking~\cite{di2016next,qu2017harnessing}, and multi-consensus mixing~\cite{liu2011accelerated} into the SVRG algorithm. 
The resulting PMGT-KatyushaX achieves the stochastic first-order oracle (SFO) complexity of $\mathcal{O}\big((n + n^{\frac{3}{4}} \sqrt{\kappa}\,) \xi \log ({1}/{\epsilon})\big)$ and the communication complexity of $\tilde{\mathcal{O}}\big( (\sqrt{n} + n^{\frac{1}{4}}\sqrt{\kappa} ) \xi \log (1/\epsilon)\big)$.
It is worth noting that the SFO complexity of our proposed algorithm matches the best-known result~\cite{allen2018katyusha} for a single machine.


Numerical experiments on several synthetic and real-world datasets demonstrate significant improvement of our proposed PMGT-KatyushaX over existing baseline methods.

\paragraph{Paper Organization} 
A review of related literature on decentralized stochastic first-order methods and stochastic sum-of-nonconvex optimization is presented in Section \ref{sec:related_works}.
In Section \ref{sec:notations}, we introduce the notations and problem setting of decentralized sum-of-nonconvex optimization. 
We present the theoretical result of PMGT-SVRG on Problem~(\ref{obj}) in Section \ref{sec:svrg_res}.
We formally present our proposed algorithm PMGT-KatyushaX with the main theorem in Section \ref{sec:methodology}. 
A proof sketch is provided for the main theorem in Section \ref{sec:sketch_proof}.
Numerical results are presented in Section \ref{sec:experiment}. 
Finally, we conclude this paper with a summary of our results in Section~\ref{sec:conclusion}.

\section{Related Work}\label{sec:related_works}
In this section, we review related literature on decentralized stochastic first-order algorithms.
In addition, we summarize existing works of stochastic optimization for the sum-of-nonconvex problem on a single machine.
\subsection{Decentralized Stochastic First-Order Methods}

We review related work about decentralized stochastic first-order methods for objective functions when each local function $f_i(\cdot)$ has the finite-sum structure. 
These methods can be divided into two categories based on the convexity of the component function.

\paragraph{$f_{i,j}(\cdot)$ is convex} The first decentralized variance-reduced method called DSA was proposed by \citet{mokhtari2016dsa}, and it is a combination of EXTRA \cite{shi2015extra} and SAGA \cite{defazio2014saga}. 
DBSA/ADFS \cite{shen2018towards,hendrikx2019accelerated,hendrikx2021optimal} attempted to accelerate DSA with proximal mapping and variance reduction.
Although several works \cite{xin2020variance,hendrikx2020dual} have proposed proximal mapping-free algorithms, the computation and communication complexities of these methods are worse than DBSA and ADFS.
\citet{ye2020pmgt,li2020communication} proposed decentralized stochastic algorithms that achieve a linear convergence rate by incorporating variance reduction, gradient tracking, and multi-consensus mixing.

\paragraph{$f_{i,j}(\cdot)$ is nonconvex} All existing decentralized stochastic methods assume the global objective function $f(\cdot)$ is possibly nonconvex if the component function $f_{i,j}(\cdot)$ is nonconvex. 
Although the analysis of these approaches can be applied to our setting, the resulting convergence rate may not be optimal for the problem studied in this paper. 
\citet{sun2020improving} provided the first decentralized stochastic algorithm, D-GET, combining variance reduction and gradient tracking.
\citet{xin2022fast,li2022destress} further proposed algorithms with improved complexity bound. 
Recently, DEAREST \cite{luo2022optimal} is the first decentralized stochastic algorithm that achieves both optimal computation and communication complexity.
Due to the assumption that the global objective is nonconvex, all these approaches can obtain at most sublinear convergence rates. 

A comparison between our work and related works is summarized in Table~\ref{tbl:comp}.

\subsection{Stochastic Sum-of-Nonconvex Optimization}

Stochastic optimization on the sum-of-nonconvex optimization problem is a commonly used technique for analyzing offline Principle Component Analysis (PCA) problems. 
\citet{garber2016robust} reduced 1-PCA subproblems to the sum-of-nonconvex problem, and they leveraged the conventional accelerated stochastic optimization scheme to accelerate the convergence. 
For the $k$-PCA problem, \citet{allen2016lazysvd} reduced the $k$-PCA problem to the sum-of-nonconvex problem, and they apply the accelerated stochastic technique to improve the convergence of $k$-PCA problem. 
\citet{allen2018katyusha} further improved the convergence by accelerating the stochastic optimization of the sum-of-nonconvex problem with the linear coupling technique \cite{allen2014linear}.

\begin{table*}
	\begin{center}
     \small
		\begin{tabular}{c|c|c|c}
			\hline
			Methods  & Problem & Stochastic Gradient Calls & Communication Complexity\\
            \hline \begin{tabular}{@{}c@{}}\texttt{ GT-SVRG } \\ \cite{xin2020variance} \end{tabular}  & \begin{tabular}{@{}c@{}}Sum-of-Convex \\ Non-Composite\end{tabular} & $\mathcal{O}\left(\big(n+\frac{\kappa^2\log\kappa}{(1-\lambda_2(W))^2}\big)\log\frac{1}{\epsilon}\right)$ & $\mathcal{O}\left(\big(n+\frac{\kappa^2\log\kappa}{(1-\lambda_2(W))^2}\big)\log\frac{1}{\epsilon}\right)$\\
            
            \hline \begin{tabular}{@{}c@{}}\texttt{ GT-SAGA } \\ \cite{xin2020variance} \end{tabular}  & \begin{tabular}{@{}c@{}}Sum-of-Convex \\ Non-Composite\end{tabular}& $\mathcal{O}\left(\big(n+\frac{\kappa^2}{(1-\lambda_2(W))^2}\big)\log\frac{1}{\epsilon}\right)$ & $\mathcal{O}\left(\big(n+\frac{\kappa^2}{(1-\lambda_2(W))^2}\big)\log\frac{1}{\epsilon}\right)$ \\
            
            \hline \begin{tabular}{@{}c@{}}\texttt{ PG-EXTRA } \\ \cite{shi2015proximal} \end{tabular}  & \begin{tabular}{@{}c@{}}Sum-of-Convex \\ Composite\end{tabular} & $\mathcal{O}\left(\frac{n\kappa}{1-\lambda_2(W)}\log\frac{1}{\epsilon}\right)$ & $\mathcal{O}\left(\frac{n\kappa}{1-\lambda_2(W)}\log\frac{1}{\epsilon}\right)$ \\
            
            \hline \texttt{ NIDS } \cite{li2019decentralized}& \begin{tabular}{@{}c@{}}Sum-of-Convex \\ Composite\end{tabular} & $\mathcal{O}\left(n\big(\kappa + \frac{1}{1-\lambda_2(W)}\big)\log\frac{1}{\epsilon}\right)$ & $\mathcal{O}\left(\big(\kappa + \frac{1}{1-\lambda_2(W)}\big)\log\frac{1}{\epsilon}\right)$\\
            
            \hline \begin{tabular}{@{}c@{}}\texttt{ PMGT-SVRG } \\ \cite{ye2020pmgt} \end{tabular}  & \begin{tabular}{@{}c@{}}Sum-of-Convex \\ Composite\end{tabular}& $\mathcal{O}\left((n+\kappa)\log\frac{1}{\epsilon}\right)$ & $\tilde{\mathcal{O}}\left(\frac{n + \kappa }{\sqrt{1-\lambda_2(W)}}\log\frac{1}{\epsilon}\right)$\\
			
			 \hline \begin{tabular}{@{}c@{}}\texttt{ PMGT-SAGA } \\ \cite{ye2020pmgt} \end{tabular} & \begin{tabular}{@{}c@{}}Sum-of-Convex \\ Composite\end{tabular}& $\mathcal{O}\left((n+\kappa)\log\frac{1}{\epsilon}\right)$ & $\tilde{\mathcal{O}}\left(\frac{n+\kappa }{\sqrt{1-\lambda_2(W)}}\log\frac{1}{\epsilon}\right)$\\
			
			 \hline \textbf{\texttt{ PMGT-SVRG } (This paper)}  & \begin{tabular}{@{}c@{}}Sum-of-Nonconvex \\ Composite\end{tabular}& $\mathcal{O}\left((n+ \sqrt{n} \kappa)\log\frac{1}{\epsilon}\right)$ & $\tilde{\mathcal{O}}\left( \frac{ \sqrt{n} + {\kappa}}{\sqrt{1 - \lambda_2(W)}} \log \frac{1}{\epsilon} \right)$\\
			
			\hline \textbf{\texttt{ PMGT-KatyushaX } (This paper)} & \begin{tabular}{@{}c@{}}Sum-of-Nonconvex \\ Composite\end{tabular}& $\mathcal{O}\left((n + n^{\frac{3}{4}} \sqrt{\kappa} ) \log \frac{1}{\epsilon}\right)$ & $ \tilde{\mathcal{O}}\left(\frac{ \sqrt{n} + \sqrt{\kappa} n^{\frac{1}{4}}}{\sqrt{1 - \lambda_2(W)}} \log \frac{1}{\epsilon} \right)$\\
			\hline
		\end{tabular}
  \caption{We compare the proposed algorithms with related work on Problem (\ref{obj}) when the global objective function is strongly convex. 
 We use notation $\tilde{\mathcal{O}}$ to hide the logarithm factor in the complexity. 
	Note that the condition number in this table does not consider the difference in smoothness parameters. We present our results by distinguishing $L$, $\ell_1$, and $\ell_2$ in Sections~\ref{sec:svrg_res} and~\ref{strong_convex_section}. 
	}
	\label{tbl:comp}	
	\end{center}
\end{table*}

\section{Preliminaries}\label{sec:notations}

In this section, we formalize our problem setting.

\subsection{Notations}

We denote $\|\cdot\|$ as the Euclidean norm for vectors and Frobenius norm for matrices, and we denote $\|\cdot\|_2$ as the operator norm for matrix. 
We use lowercase non-bold letter $x \in \mathcal{R}^d$  as a random variable of dimension $d$ and lowercase bold letter
\begin{align*}
    \mathbf{x} = \left( x_1, \ldots, x_m \right)^{\top}\in \mathcal{R}^{m \times d}
\end{align*}
as the aggregated variable collected from $m$ machines. We denote all one vector of dimension $m$ by $\mathbf{1} \in \mathcal{R}^{m}$.
For simplicity, we write $\mathbf{1} x \in \mathcal{R}^{m \times d}$ as the Kronecker product between the all one vector $\mathbf{1}$ and some vector $x$.
We use $\bar{x}$ as the average of the aggregated variable $\vx$ such that $\bar{x} = {m}^{-1} \mathbf{1}^\top \mathbf{x}$.
For the non-smooth function $\psi(\cdot)$, we define~$\Psi(\mathbf{x}) = {m}^{-1} \sum_{i=1}^m \psi(x_i)$ for the aggregated variable $\vx\in\mathcal{R}^{m \times d}$.
We also define the proximal operator for vector~$x$ as 
\begin{equation*}
\textstyle\rm prox_{\eta, \psi}(x) = \argmin_{z \in \mathcal{R}^d} \left( \psi(z) + \norm{z - x}^2 /(2\eta)\right)
\end{equation*}
and the proximal operator for aggregated variable $\mathbf{x}$ as
\begin{equation*}
\textstyle\rm prox_{m\eta, \Psi}(\mathbf{x})\! = \!\argmin_{\mathbf{z} \in \mathcal{R}^{m\! \times\! d}}\!\left(\! \Psi(\mathbf{z})\! + \! \norm{\mathbf{z}\! -\! \vx}^2\!/(2m \eta)\! \right).    
\end{equation*}
We use $x^*$ to represent the optimal solution for $F(\cdot)$ as
\begin{equation*}
   \textstyle x^* = \argmin_{x\in\mathcal{R}^d} F(x)\ .
\end{equation*}

\subsection{Problem Formulation}

We summarize some of the basic properties of convex and smooth functions below.
\begin{definition}
For a function $f: \mathcal{R}^d \to \mathcal{R}$, there exist some constants $L, \ell_1, \ell_2 > 0$ and $\sigma \geq 0$ such that
\begin{enumerate}
\item $f$ is $\sigma$-strongly convex. That is, for any $x, y \in \mathcal{R}^d$, 
\begin{equation*}
   f(x) - f(y) - \langle \nabla f(y), x - y \rangle \geq \frac{\sigma}{2} \|x-y\|^2 ;
\end{equation*}
\item $f$ is $L$-Lipschitz smooth. That is, for any $x, y \in \mathcal{R}^d$,
\begin{equation*}
    \| \nabla f(x) - \nabla f(y)\| \leq L\|x-y\|;
\end{equation*}
\item $f$ is $(\ell_1, \ell_2)$-smooth. That is, for any $x,y \in \mathcal{R}^d$,
\begin{equation*}
   \hspace{-4mm}
    -\frac{\ell_2}{2} \|x - y\|^2 \leq f(x) - f(y) - \langle \nabla f(y), x - y \rangle \leq \frac{\ell_1}{2} \|x - y\|^2 .
\end{equation*}
\end{enumerate}
\end{definition}
Recall that the global objective function (\ref{obj}) can be decomposed into $m n$ nonconvex functions $f_{i, j} (\cdot)$. 
The $i$th agent is given access to a disjoint subset of $n$ functions $f_{i,j}(\cdot)$, for $\forall j \in [n]$. 
We assume the function $f(\cdot)$ is convex and $L$-smooth, each $f_{i,j}(\cdot)$ is $(\ell_1, \ell_2)$-smooth with $\ell_2 \geq \ell_1$ and $\psi(\cdot)$ is a proper convex function. 
We further assume $f$ is $\sigma_f$-strongly convex and $\psi$ is $\sigma_{\psi}$-strongly convex with $\sigma_f\geq 0$ and $\sigma_{\psi} \geq 0$. 
We define~$\sigma = \sigma_f + \sigma_{\psi}$ such that $\sigma > 0$. 

We focus on decentralized optimization on a network in which each agent only communicates with its neighbors. 
The topology of the network is characterized by the gossip matrix $W$. 
We let $W_{i,j} \geq 0$ if nodes $i$ and $j$ are direct neighbour in $\mathcal{G}$; and $W_{i,j} = 0$ if nodes $i$ and $j$ are not connected. 
Furthermore, we assume $W$ is a doubly stochastic matrix, and it satisfies the following properties:
\begin{definition}
    Let $W$ be a doubly stochastic matrix. Then, 
(a) $W$ is symmetric,
(b) $0 \preccurlyeq W \preccurlyeq I$ and $W \mathbf{1} = \mathbf{1}$, and
(c) $\rm null(I - W) = \text{span}(\mathbf{1})$.
\label{doubly_stochastic}
\end{definition}

\section{Convergence Analysis of PMGT-SVRG}\label{sec:svrg_res}
In this section, we show the convergence rate of the PMGT-SVRG~\cite{ye2020pmgt} for the objective function (\ref{obj}).
PMGT-SVRG achieves a linear convergence rate on the sum-of-convex problem by integrating gradient tracking and multi-consensus mixing techniques into the SVRG algorithm.
We remove the assumption that each function $f_{i,j}(\cdot)$ is convex such that the inequality
\begin{equation*}
    \frac{1}{2 L} \hspace{-0.5mm} \!\norm{ \!\nabla f_{i,j} (x) \! - \! \nabla f_{i,j}(y)}^2 \hspace{-0.5mm}
    \leq\hspace{-0.5mm} f_{i,j} (x) \! -  \!f_{i,j}(y)  \!-  \!\langle  \nabla  \! f_{i,j} \!(y), \! x \! - \! y  \rangle 
\end{equation*} 
for any $x, y \in \BR^d$ in Assumption $1$ of~\cite{ye2020pmgt} no longer holds.
The following result shows that PMGT-SVRG can still achieve a linear convergence rate on the sum-of-nonconvex problem:
\begin{theorem}
Assume function $F(\cdot)$ defined in (\ref{obj}) is $\sigma$-strongly convex, $f(\cdot)$ is $L$-smooth, and each component $f_{i,j}$ is $(\ell_1, \ell_2)$-smooth.
Additionally, we assume that the underlying network matrix $W$ is doubly stochastic so it satisfies the properties in Definition \ref{doubly_stochastic}. 
Under appropriate hyperparameter setting, to obtain $F(\bar{y}^{k+1}) - F(x^*) \leq \epsilon$, the algorithm PMGT-SVRG requires at most 
\begin{equation*}
    \mathcal{O}\left( \left(n + \frac{L + (\ell_1 \ell_2)^\frac{1}{2}}{\sigma}n^{\frac{1}{2}}\right) \log \frac{F(\bar{y}^0) - F(x^*)}{\epsilon}\right)
\end{equation*}
SFO calls and 
\begin{equation*}
     \tilde{\mathcal{O}}\left( \frac{ \left(\sqrt{n} + (L + \sqrt{\ell_1 \ell_2})/{\sigma}  \right)  }{({1 - \lambda_2(W)})^{1/2}} \log \frac{F(\bar{y}^0) - F(x^*)}{\epsilon} \right)
\end{equation*}
rounds of communication.
\label{theorem_strong-SVRG}
\end{theorem}
\begin{remark}
Compared with the theoretical results of the PMGT-SVRG on the decentralized sum-of-convex problem in Table \ref{tbl:comp}, the SFO complexity introduces an additional dependency on $\sqrt{n}$. 
It can be inferred that when the condition number $\kappa \coloneqq (L + \sqrt{\ell_1 \ell_2}) / \sigma$ is larger than $\sqrt{n}$, the SFO complexity of PMGT-SVRG on the sum-of-nonconvex problem is worse than the sum-of-convex problem. 
Interestingly, the communication complexity of the PMGT-SVRG achieved by our analysis is better than that by \citet{ye2020pmgt} even though our objective function is harder.
The improvement comes from the introduction of the minibatch in Algorithm~3 in the appendix. 
\end{remark}

\section{PMGT-KatyushaX}\label{sec:methodology}

In this section, we propose the main idea behind PMGT-KatyushaX and present the convergence theorem of this algorithm.  
\subsection{The Algorithm}
We present the main intuition of the PMGT-KatyushaX algorithm.
The core design of the algorithm is to apply the acceleration scheme \cite{allen2014linear} on the stochastic variance reduced gradient (SVRG) \cite{johnson2013accelerating} method.
Furthermore, we blend the powerful ideas of gradient tracking and multi-consensus mixing into the accelerated algorithm to develop the decentralized variant of the algorithm.

The backbone of our algorithm is the SVRG which adopts an outer-inner loop structure to reduce the inherent variance of stochastic gradients.
Specifically, we construct a full gradient snapshot $\nabla f_i(w_i^0)$ for each agent $i \in [m]$ at each epoch.
At each iteration of the inner loop, a variance-reduced unbiased gradient estimator is updated as:
\begin{equation*}
    \textstyle v_i^t = \nabla f_i(w_i^0) + {b}^{-1} \sum_{j_i \in B_i^t} \left(\nabla f_{i, j_i}^t (w_i^t) - \nabla f_{i, j_i}^t (w_i^0)\right)
\end{equation*}
where $B_i^t$ is a minibatch of $b$ indices sampled uniformly from $\{1, \dots, n\}$. 
Recall that the global objective function~(\ref{obj}) contains a convex, non-smooth function $\psi(\cdot)$, we apply the proximal mapping after executing one step of gradient descent:
\begin{equation*}
    w_i^{t+1} = \rm prox_{\eta, \psi} (w_i^t - \eta v_i^t)
\end{equation*}
where $\eta$ is the learning rate of the SVRG algorithm.
The SVRG algorithm is known to achieve a linear rate of convergence for the strongly convex objective function.

\paragraph{Acceleration} If we naively extend the SVRG algorithm to the decentralized setting, the convergence rate of the resulting algorithm has a linear dependency on the condition number as shown in Theorem \ref{theorem_strong-SVRG}. 
To obtain an improved dependency on the condition number, we employ the acceleration technique introduced by \citet{allen2014linear}.
The acceleration is achieved by the linear coupling of the gradient descent step and mirror descent step.
In particular, denote $x_i^k \coloneqq w_i^0$ and $y_i^k \coloneqq w_i^T$ as the first and the last iterate of the SVRG inner loop at Epoch $k$, then 
we apply one step of mirror descent as follows
\begin{equation*}
    q_i^k =  \!\argmin_{q \in \BR^d}\Big\{  \!\frac{1}{2} \! \norm{q \! -  \!q_i^{k-1}}^2  \! +  \! \big\langle  \!\frac{x_i^k  \!-  \!y_i^k}{2 \tau}, q \big\rangle \! +  \!\frac{\tau}{4} \norm{q - y_i^k}^2 \! \Big\},
\end{equation*}
which can be simplified as
\begin{equation*}
    q_i^k = \frac{q_i^{k-1} + \frac{\tau y_i^k}{2} - \frac{x_i^k - y_i^k}{2 \tau}}{1 + \frac{\tau}{2}}\ .
\end{equation*}
The iterate $x_i^{k+1}$ can be updated as a linear combination of $y_i^k$ and $q_i^k$:
\begin{equation*}
    x_i^{k+1} = \tau q_i^k + (1 - \tau) y_i^k\ .
\end{equation*}

On top of the accelerated SVRG method, we also apply gradient tracking and multi-consensus mixing to extend the above algorithm to the decentralized setting.

\paragraph{Gradient Tracking} Recall that our goal is to find the minima $x^*$ of the global objective function ($\ref{obj}$). 
However, gradients collected from local neighbors have large variances due to the dissimilarity between distinct local objective functions~$f_i(\cdot)$.
To alleviate this issue, we adopt the gradient tracking technique that introduces a new variable $\vs$ to track the difference between local gradient estimators:
\begin{equation*}
    \vs^{t+1} = \rm Mix(\vs^{t} + \vv^{t+1} - \vv^{t} )
\end{equation*}
where $\rm Mix(\cdot)$ is some mixing protocol, and $\vv^t$ is the aggregated variance-reduced gradient estimator of $\{v_i^t\}_{i=1}^m$ at the $t$-th iteration of the inner loop. 
The intuition behind the technique is that while the variance of the local gradient estimators can be arbitrarily large in general, the variance of the differences between local gradient estimators will be small when the local variable approaches the global minima $x^*$.
The technique also applies to the full gradient constructed in the outer loop.

\paragraph{Multi-Consensus Mixing} Under the decentralized learning scenario, each agent wishes to obtain the global average of their variables through communication with local neighbors.
The agent can get a more faithful estimate of the global average through multiple rounds of local communication.
Additionally, the mixing technique can be accelerated to achieve even faster consensus.
The complete procedure of the Multi-Consensus Mixing can be found in Algo.~\ref{fastmix}.
We apply the Mixing technique to both state variables and gradient tracking variables in both the inner loop and outer loop.

After fusing these techniques with the SVRG algorithm, the complete PMGT-KatyushaX for the decentralized setting is presented in Algo.~\ref{strongconvex_algo_2}. 

\begin{algorithm}[t]
\caption{PMGT-KatyushaX}
\label{strongconvex_algo_2}
\textbf{Input}: $x_i^0$, $y_i^0$, $y_i^{-1}$,  $q_i^{0}$ and $S_i^0$\\
\textbf{Parameter}: {FastMix parameter $M$, functions $\{f_i\}_{i=1}^m$ and $\psi$, initial point $x^0$, mini-batch size $b$, learning rate
$\eta \geq 0$, momentum parameter $\tau \in (0, 1]$, number $K$ of epochs}\\
\textbf{Output}: $x_i^K$, $y_i^K$, $q_i^{K}$
\begin{algorithmic}[1]
  \STATE  Initialize $x_i^0 = x^0$, $y_i^0 = x^0$, $y_i^{-1} = x^0$, $q_i^{0} = x^0$ and $\hat{s}_i^0 = \nabla F (x_0)$ for each $i$ in parallel \\
\FOR{$k = 0, \ldots, K-1$}
    \STATE $x_i^{k+1} = \texttt{FastMix}(\tau \mathbf{q}^k + (1 - \tau) \mathbf{y}^k, M)_i$ 
    \STATE $\mathbf{\nu}_i^{k+1} = \hat{s}_i^{k} + \nabla f_i (x_i^{k + 1}) - \nabla f_i (x_i^{k})$ 
    \STATE$\hat{s}_i^{k+1} = \texttt{FastMix}(\mathbf{\nu}^{k+1}, M)_i$ 
    \STATE$s_i^{-1} = v_i^{-1} = \hat{s}_i^{k+1}$ 
    \STATE$\mu_i = \hat{s}_i^{k+1}$, \quad $w_i^0 = x_i^{k+1}$ 
    \STATE $t_0 = \ceil{n/b}$ and sample $T \sim \texttt{Geom}(1/t_0)$
    
    \FOR{$t = 0, \ldots, T$} 
        \STATE Let $B_i^t$ be a batch of $b$ indices sampled uniformly from $[n]$ with replacement 
        \STATE $v_i^t = \mu_i + {b}^{-1}\sum_{j_i \in B_i^t}(\nabla f_{i,j_i}^t(w_i^t) - \nabla f_{i, j_i}^t(w_i^0))$ 
        \STATE $s_i^{t} = \texttt{FastMix}(\mathbf{s}^{t-1} + \mathbf{v}^{t} - \mathbf{v}^{t-1}, M)_i$ 
        \STATE $w_i^{t+1} = \texttt{FastMix}({\rm prox}_{m \eta , \Psi}(w^t - \eta s^{t}), M)_i$
    \ENDFOR
    \STATE $y_i^{k+1} = w_i^{T+1}$ 
    \STATE $q_i^{k+1} = \texttt{FastMix}(\frac{\mathbf{q}^k+\frac{\tau \mathbf{y}^{k+1}}{2}\!-\!\frac{\mathbf{x}^{k+1}\!-\!\mathbf{y}^{k+1}}{2 \tau} }{1+\frac{\tau}{2}}, M)_i$
\ENDFOR
\end{algorithmic}
\end{algorithm}

\begin{algorithm}[!b]
\textbf{Initilaize}: ${\mathbf x^{-1}} = {\mathbf x}^{0}$ and $\eta = 1 / (1+\sqrt{1-\lambda_2^2(W)})$ \\
\textbf{Output}: {$W_K$}
\begin{algorithmic}[1]
    \FOR{$k=0, \ldots, M$}
\STATE
\begin{equation}
\mathbf{x}^{k+1} = (1+\eta)\mathbf{x}^k W - \eta \mathbf{x}^{k-1}    
\end{equation}
\ENDFOR
\end{algorithmic}
\caption{FastMix($\mathbf{x}^0$, $M$, $W$)}
\label{fastmix}
\end{algorithm}

\subsection{Main Theorem}\label{strong_convex_section}

In this subsection, we characterize the linear convergence rate of PMGT-KatyushaX on the decentralized sum-of-nonconvex problems. 

\begin{theorem}\label{theorem_strong_convex}
Assume function $F(\cdot)$ defined in (\ref{obj}) is $\sigma$-strongly convex, $f(\cdot)$ is $L$-smooth, and each component $f_{i,j}(\cdot)$ is $(\ell_1, \ell_2)$-smooth.
We also assume that the underlying network matrix $W$ is doubly stochastic so it satisfies the properties in Definition \ref{doubly_stochastic}. 
Under appropriate parameter settings, 
the outputs of Algo.~\ref{strongconvex_algo_2} has the following properties:
\begin{equation*}
\begin{array}{l}
\displaystyle\mathbb{E}\left[F(\bar{y}^{K}) - F(x^*) \right] \leq \frac{3}{\left( 1 + {\tau}/{2} \right)^{K}} \left( F(\bar{y}^0) - F(x^*) \right), \\
\displaystyle \mathbb{E}\left[ \norm{\bar{x}^K - x^*}^2 \right] \leq  \frac{128L}{\sigma\left( 1 + {\tau}/{2} \right)^{K}}\norm{\bar{x}^0 - x^*}^2\ ,
 \end{array}
\end{equation*}
and
\begin{equation*}
\mathbb{E}\left[ \norm{\bar{q}^K - x^*}^2 \right] \leq \frac{1750L}{\sigma \tau^4\left( 1 + {\tau}/{2} \right)^{K}}\norm{\bar{x}^0 - x^*}^2. 
\end{equation*}
\end{theorem}
\begin{remark}
To obtain the $\epsilon$-approximate solution of the Problem (\ref{obj}), i.e., $F(\bar{y}^K) - F(x^*) \leq \epsilon$, the outer loop of Algorithm PMGT-KatyushaX has to be executed for at least 
\begin{equation*}
    K = \fO\left( \!\big(1 + \frac{\sqrt{L b}}{\sqrt{n \sigma}} + \frac{(\ell_1 \ell_2)^{\frac{1}{4}}}{\sqrt{\sigma} n^{\frac{1}{4}} }\big) \log \left( \frac{F(\bar{y}^0) - F(x^*)}{\epsilon}\right) \!\right)
\end{equation*}
times. Recall that at each epoch, the algorithm makes one call of full gradient oracle and $T b$ SFO calls where $\BE[T] = n/b$. 
Consequently, our algorithm makes an expectation of $2n$ SFO calls. 
Additionally, as can be seen in Algo.~\ref{strongconvex_algo_2}, the multi-consensus mixing takes $M$ rounds of communication when called. 
We can deduce that
\begin{equation*}
    M = \fO\left(\log \left(\frac{L}{ n \sigma} + \frac{\sqrt{\ell_1 \ell_2}}{n \sigma}\right)\right) .
\end{equation*}
To reach $\epsilon$-accuracy, Algo.~\ref{strongconvex_algo_2} has to make $\fO(K n)$ calls to the first-order oracle and  $\fO(K n M / b)$ rounds of communication. 
\end{remark}

We can bound the computation complexity and communication complexity by setting batch size $b=\sqrt{n}$ and the following corollary can be obtained:

\begin{corollary}\label{corollary_strong}
Under the setting of Theorem \ref{theorem_strong_convex}, to obtain an~$\epsilon$-approximate solution $\bar{y}^K$, i.e., $F(\bar{y}^K) - F(x^*) \leq \epsilon$, PMGT-KatyushaX requires at most 
\begin{equation*}
    \mathcal{O}\left( \left(n + \frac{L^{\frac{1}{2}} + (\ell_1 \ell_2)^\frac{1}{4}}{\sqrt{\sigma}}n^{\frac{3}{4}}\right) \log \frac{F(\bar{y}^0) - F(x^*)}{\epsilon}\right)
\end{equation*}
SFO calls and 
\begin{equation*}
     \tilde{\mathcal{O}}\left( \frac{ \left(\sqrt{n} + \frac{L^{\frac{1}{2}} + (\ell_1 \ell_2)^\frac{1}{4}}{\sqrt{\sigma}}n^{\frac{1}{4}}  \right)  }{\sqrt{1 - \lambda_2(W)}} \log \frac{F(\bar{y}^0 - F(x^*))}{\epsilon} \right)
\end{equation*}
rounds of communication.
\end{corollary}

 \section{Proof Sketch of Theorem \ref{theorem_strong_convex}}\label{sec:sketch_proof}
In this section, we provide a sketch of the proof for Theorem \ref{theorem_strong_convex}.
Since Algo.~\ref{strongconvex_algo_2} has a double loop structure, we prove the theorem in two stages.
For the first stage, we present the analysis for one epoch of decentralized SVRG adapted for the sum-of-nonconvex problem; 
For the second stage, our goal is to show that the consensus error and the convergence error of the outer loop decay linearly at the same rate.

We first define the vector of consensus errors in the inner loop as $z^t = (\norm{\vw^t - {\bf 1} \Bar{w}^t}^2, \norm{\vs^t - {\bf 1} \Bar{s}^t}^2)^{\top}$. 
We can build a system of linear inequalities with $z^t$ such that the spectral norm of the coefficient matrix is less than 1 if the hyperparameters are chosen appropriately.
We start by reviewing the following lemma which is essential for the analysis of the consensus error:
\begin{lemma}[\citet{liu2011accelerated}]
Let $\vw_0, \vw_M \in \mathcal{R}^d$ be the input and output of the Algorithm \ref{fastmix} and $\bar{w} = \frac{1}{m} \mathbf{1}^T \vw_0$, then it satisfies that
\begin{equation*}
    \|\vw_M - \mathbf{1} \bar{w}\| \leq \rho \|\vw_0 - \mathbf{1} \bar{w}\|, 
\end{equation*}
where $\rho = \sqrt{14} \Big(1- \big(1 - \frac{1}{\sqrt{2}}\big)\sqrt{1 - \lambda_2 (L)}\Big)^M$.
\label{fastmix_lemma}
\end{lemma}

After constructing the linear system of inequalities for $z^{t}$,
we obtain the following lemma that bounds the consensus error accumulated in the inner loop:
\begin{lemma}
    Given $T \sim \rm Geom (p)$, the expected consensus error at the end of the inner loop satisfies
     \begin{equation*}
    \begin{array}{rl}
           \mathbb{E}_T\left[\|{z}^{T} \|\right] \leq\hspace{-2.4mm} &\displaystyle \|{z}^0\| +  \frac{1-p}{p}{16\rho^2\eta^2\ell_1\ell_2}\|\vw^0 - \mathbf{1}\bar{w}^0\|^2 \\
    &\displaystyle + \frac{2-p}{p}{16\rho^2\eta^2 \ell_1\ell_2m}\ \mathbb{E}_T\left[\|\bar{w}^{T}- \bar{w}^0\|^2\right].
    \end{array}
    \end{equation*}
\end{lemma}
Using the convexity and smoothness assumption, we can obtain the following result.
\begin{lemma}
For any iteration $t$ in the inner loop, the average variable is defined as $\bar{w}^{t} = {m}^{-1}\mathbf{1}^{\top}\vw^{t}$ where $\vw^{t+1} = \rm FastMix(\mathbf{prox}_{m \eta, \Psi}(\vw^t - \eta \vs^t), M)$. For any $u \in \BR^d$, 
\begin{equation*}
\hspace{-1.7mm}
\begin{array}{l}
     \mathbb{E}\left[F(\bar{w}^{t+1}) - F(u)\right]\\
     \displaystyle\leq  \mathbb{E} \left[  \frac{\eta}{1 - \eta L}\norm{ \bar{s}^t- \nabla f(\bar{w}^t)}^2 + \frac{2-\eta \sigma_f}{4\eta}\|u - \bar{w}^t\|^2 \right.\\
      \qquad\ \ \displaystyle -  \frac{(1 \!+ \!\sigma_{\psi} \eta)\|u  \!- \! \bar{w}^{t+1}\|^2}{2\eta} +  \frac{1 \!- \!\eta L  \!+ \! 2\eta}{2(1 \!- \!\eta L)m}\|\mathbf{1}\bar{s}^t  \! -  \! \mathbf{s}^t\|^2 \\
     \qquad\ \ \displaystyle \left.  + \left(\frac{\ell_1\ell_2}{\sigma_f m}+ \frac{1+\eta}{2m\eta}\right)\|\vw^t - \mathbf{1}\bar{w}^t \|^2 \right]. 
\end{array}    
\end{equation*}
\end{lemma}
Moreover, one can show the following result by the $(\ell_1, \ell_2)$-smoothness of $f_{i,j}$:
\begin{lemma} Denote the variable $\bar{s}^t = {m}^{-1}\sum_{i=1}^m \mathbf{s}_i^t$ with $\mathbf{s}_i^{t} = FastMix(\mathbf{s}^{t-1} + \mathbf{v}^{t} - \mathbf{v}^{t-1}, M)_i$. Then,
\begin{equation*}
\begin{array}{l}
    \displaystyle\mathbb{E}\left[\|\bar{s}^t - \nabla f(\bar{w}^t)\|^2\right] \\
    \displaystyle\leq \left( \frac{6 \ell_1\ell_2}{m b} + \frac{2 \ell_1\ell_2}{m} \right) \|  \vw^t - \mathbf{1}\bar{w}^t\|^2 \\
    \displaystyle \quad \ + \frac{6\ell_1\ell_2}{b} \|  \bar{w}^t - \bar{w}^0\|^2 +\frac{6\ell_1\ell_2}{m b}\| \vw^0 - \mathbf{1}\bar{w}^0  \|^2. 
\end{array}        
\end{equation*}
\end{lemma}
Combining the above three lemmas, we can prove the following main result for the inner loop:
\begin{lemma}
    If we choose the hyperparameters for Algorithm \ref{strongconvex_algo_2} such that $\eta \!\leq \! \min\left({1}/({2L}), \sqrt{{b}/{(\ell_1 \ell_2 t_0)}}/8\right)$, and $\rho \leq  \min\left((8 b m \eta J  t_0)^{-1/2}, (18(1 + {b}))^{-1/2}\right)$ where $J \coloneqq J(\eta, b)$ is some positive constant. Then, for any $u \in \BR^d$,
\begin{equation*}
\begin{array}{l}
    \mathbb{E}\left[F(\bar{w}^{T+1}) - F(u)\right] \\ 
    \displaystyle \leq  \mathbb{E} \left[-\frac{\|\bar{w}^{T+1} - \bar{w}^0\|^2}{4\eta t_0} + \frac{\langle \bar{w}^0 - u, \bar{w}^0 - \bar{w}^{T+1}\rangle}{t_0\eta} \right. \\
    \displaystyle \qquad\ \ - \frac{\sigma_f + \sigma_{\psi}}{8}\| u - \bar{w}^{T+1}\|^2 +  J \eta^2 \| \mathbf{s}^0 - \mathbf{1}\bar{s}^0  \|^2 \\
    \displaystyle  \qquad\  \left.  + \left( J + \frac{1}{4 \eta m  }\right)\| \vw^0 - \mathbf{1}\bar{w}^0  \|^2  \right].
\end{array}
\end{equation*}
\end{lemma}
For the theoretical analysis of the outer loop, we define the following vector of consensus errors 
\begin{equation*}
(\norm{\vx^t  \!- \! {\bf 1} \bar{x}^t}^2 \!,  \norm{\vy^t  \!-  \!{\bf 1} \bar{y}^t}^2\!,  \norm{\vq^t  \!-  \!{\bf 1} \bar{q}^t}^2\!,  \eta^2\norm{\bf \hat{s}^t  \!-  \!{\bf 1} \bar{\hat{s}}^t}^2)^{\top}.  
\end{equation*}
We also construct a system of linear inequalities of the above quantity. 
Using a similar proof technique as the inner loop to bound the consensus error of the outer loop, we can prove Theorem~\ref{theorem_strong_convex} by induction. 
The complete theoretical analysis of PMGT-KatyushaX is presented in the Appendix.

\section{Numerical Experiments}\label{sec:experiment}
To demonstrate the efficiency of PMGT-KatyushaX, we evaluate the proposed method on the sub-problem of solving PCA by the shift-and-invert method.
The corresponding sum-of-nonconvex optimizing problem has the form of
\begin{equation}\label{exp_obj}
\small\begin{split}
\min_{x\in\mathcal{R}^d} F(x) = \frac{1}{2} x^{\top}\left(\left(\lambda_1 + \frac{\lambda_1 - \lambda_2}{r}\right) I - A\right) x + b^{\top}x
\end{split}
\end{equation}
where $A = ({m n})^{-1}\sum_{i=1}^m \sum_{j=1}^n a_{i,j} a_{i,j}^{\top}\in\BR^{d\times d}$, $\lambda_1$ and $\lambda_2$ are the largest and the second largest eigenvalues of the matrix $A$; $r$ is a hyperparameter that controls the ratio for the eigengap.

\begin{figure}[!tb]
    \centering
    \begin{tabular}{cc}
     \includegraphics[scale=0.233]{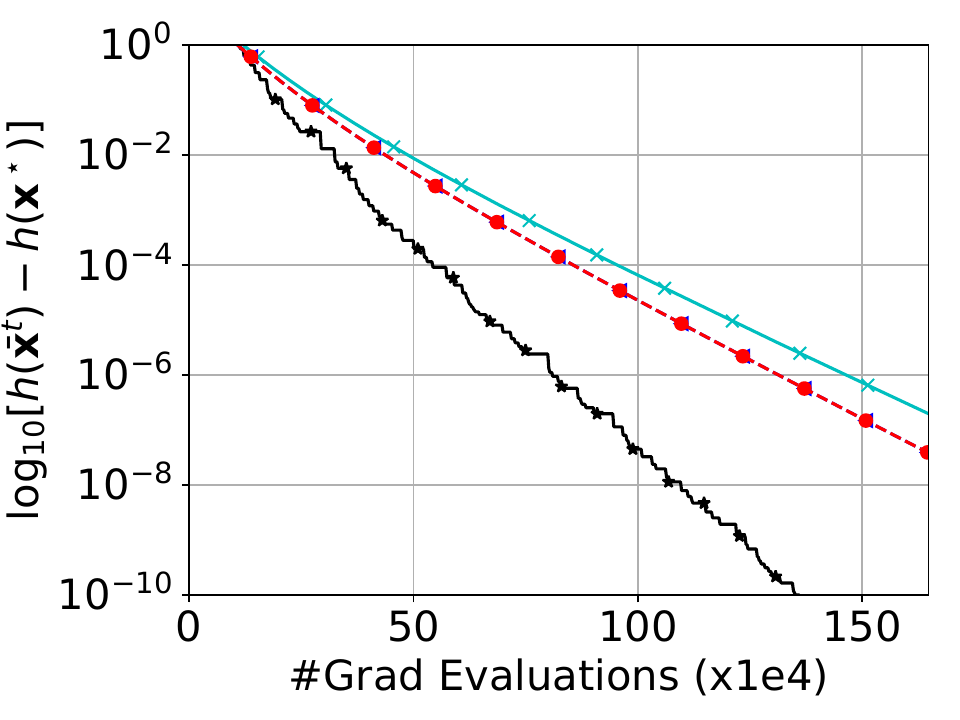}
     & \includegraphics[scale=0.233]{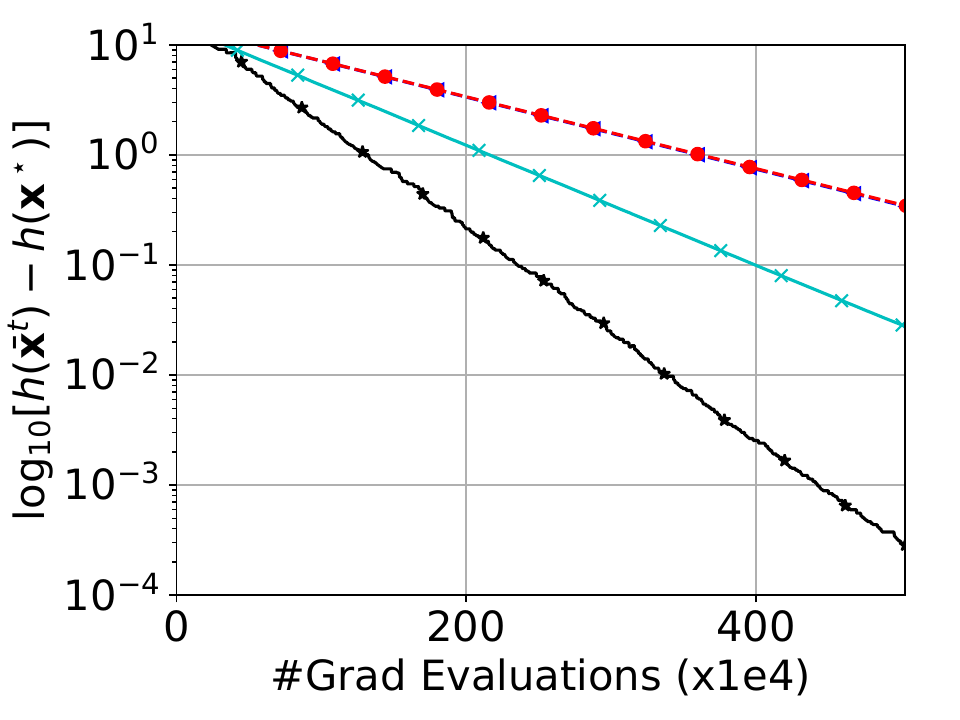}  \\
      \includegraphics[scale=0.233]{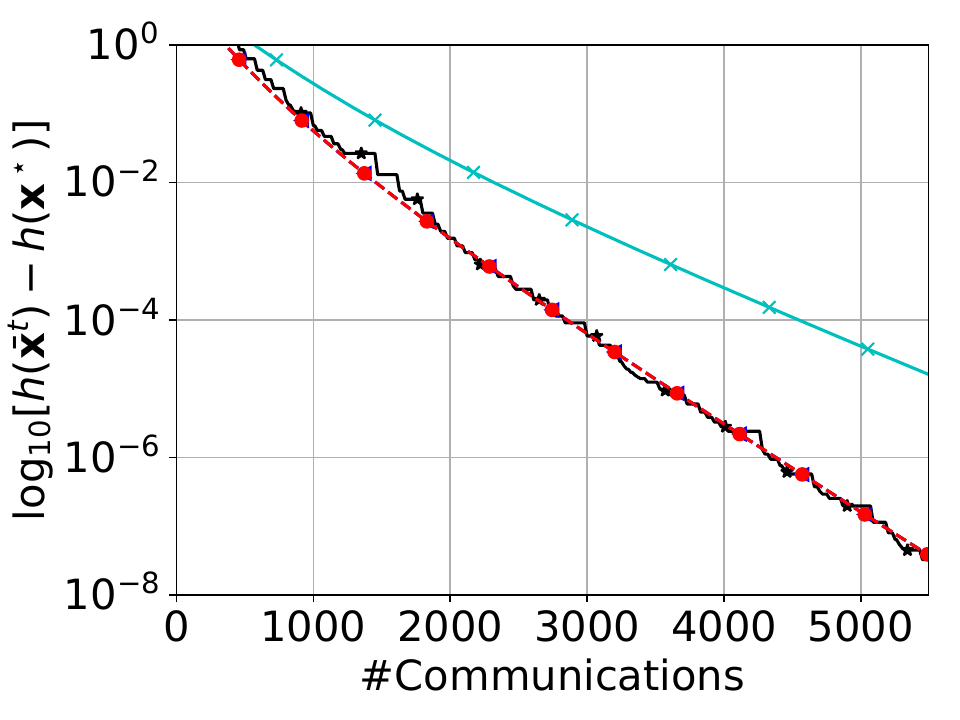}
     & \includegraphics[scale=0.233]{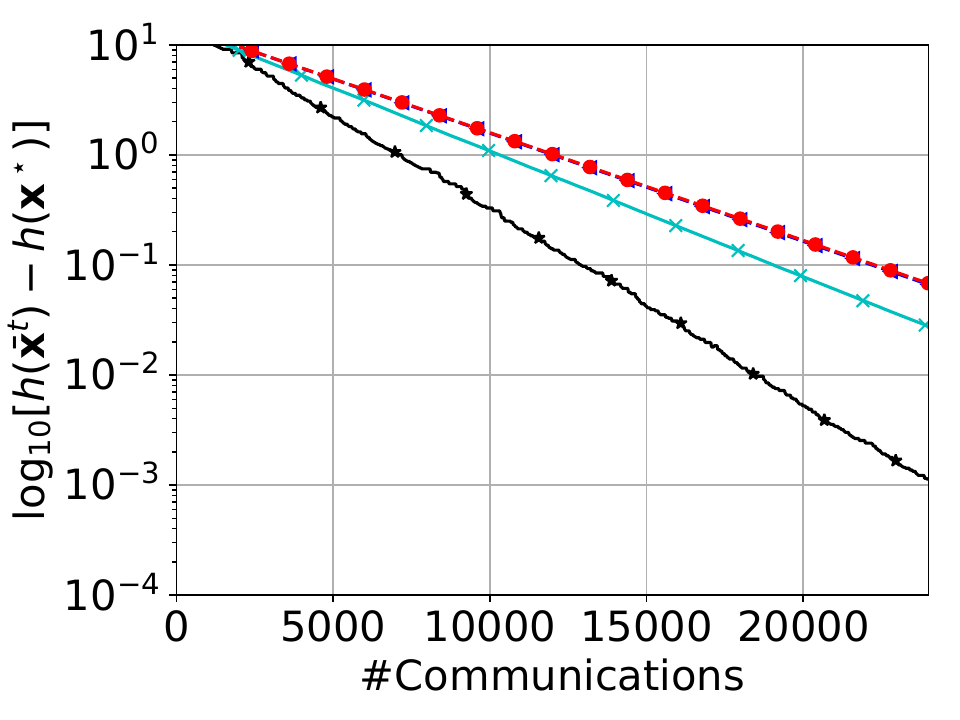}  \\
      \includegraphics[scale=0.233]{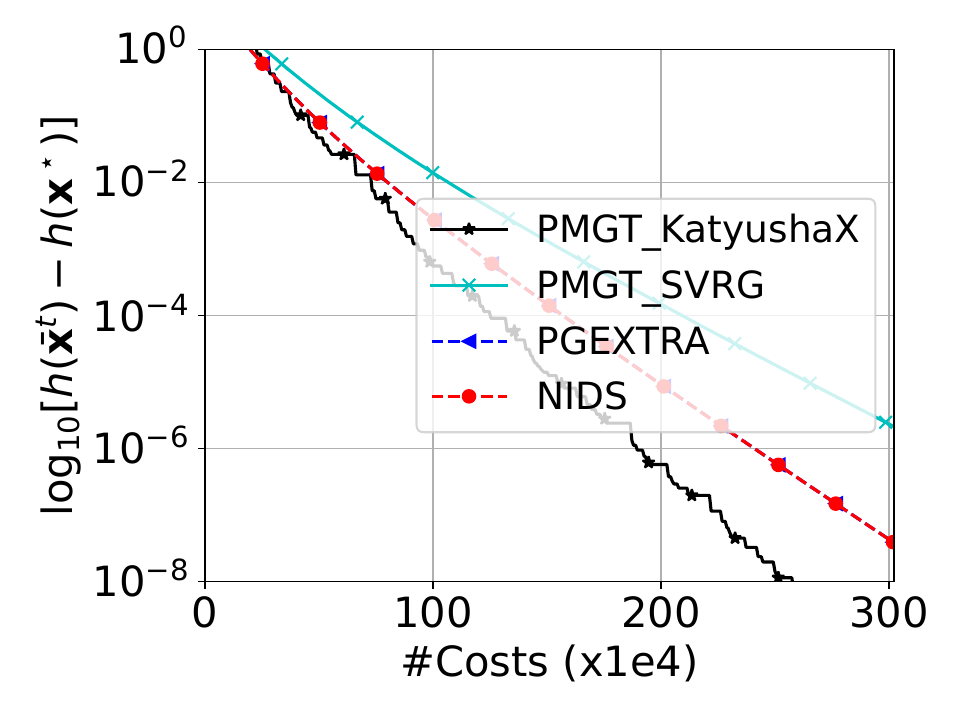}
     & \includegraphics[scale=0.233]{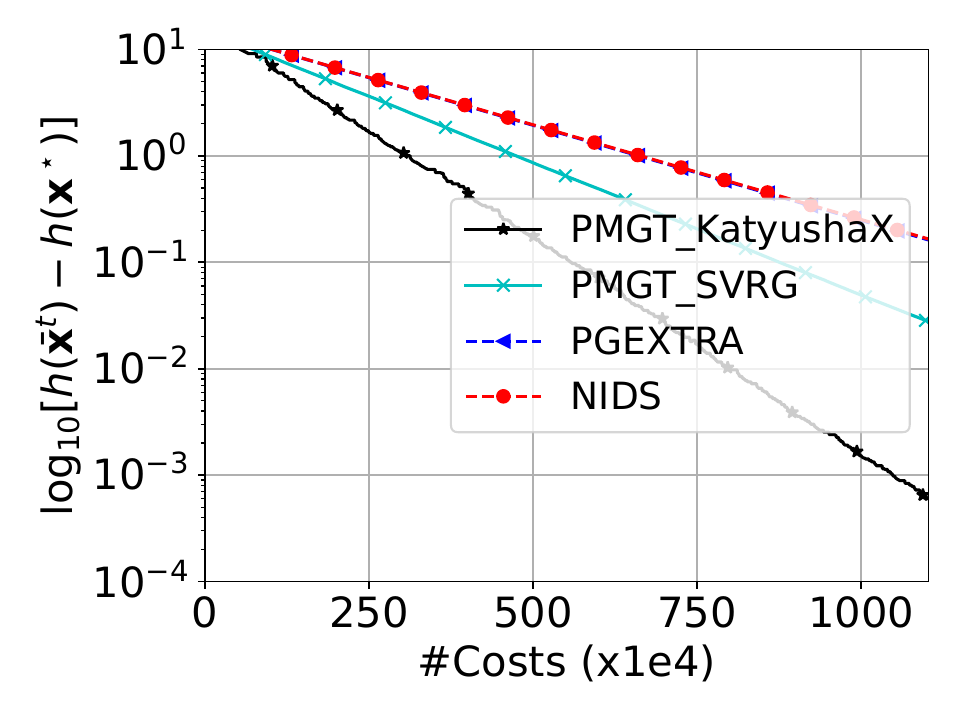}  \\
\end{tabular}
    \caption{Performance comparison between PGEXTRA, NIDS, PMGT-SVRG, and PMGT-KatyushaX on the synthetic dataset. 
The left column represents results with the ratio $r=2$ and the right column represents results with the ratio $r=300$ defined in Problem~(\ref{exp_obj}).
The plot of PGEXTRA and NIDS are overlapped as their performance are close to each other.} 
\label{synthetic_result}
\end{figure}

\begin{figure}[!tb]
    \centering
    \begin{tabular}{cc}
     \includegraphics[scale=0.233]{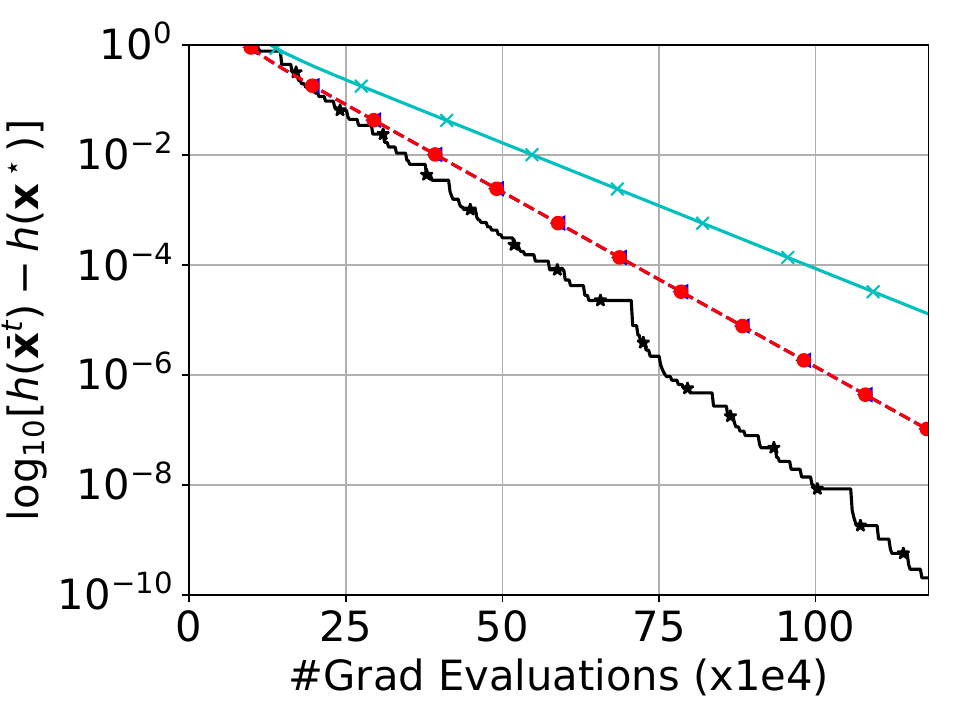}
     & \includegraphics[scale=0.233]{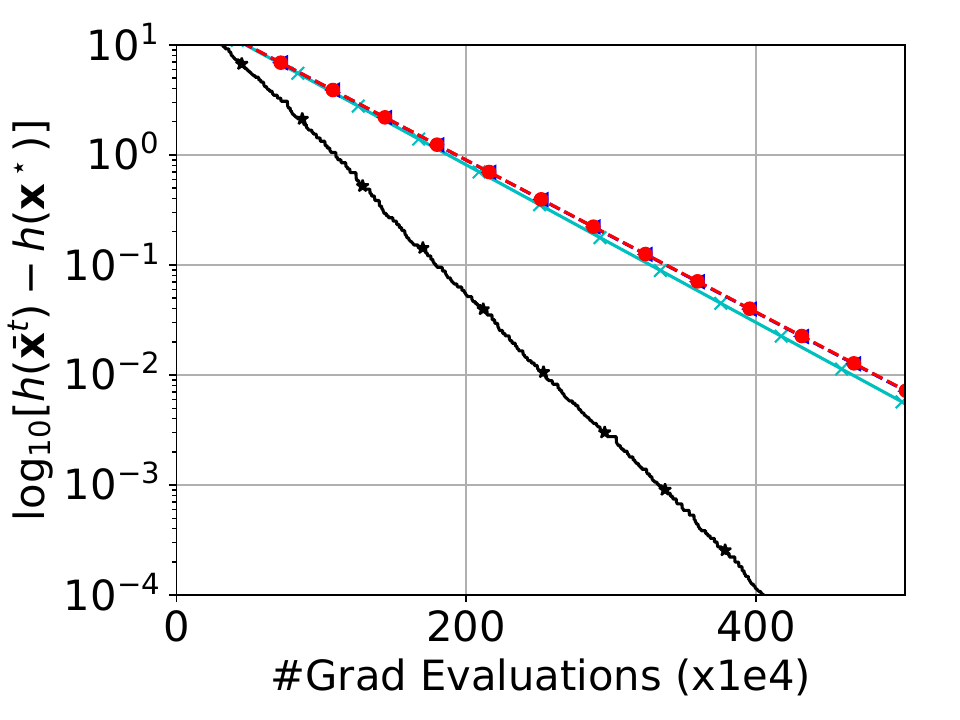}  \\
      \includegraphics[scale=0.233]{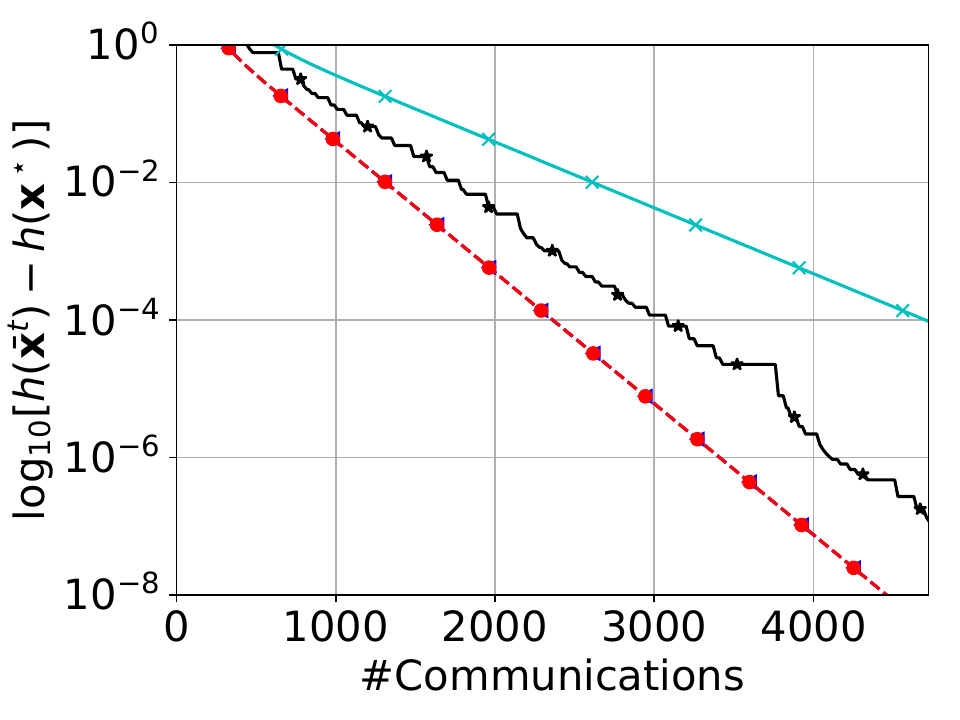}
     & \includegraphics[scale=0.233]{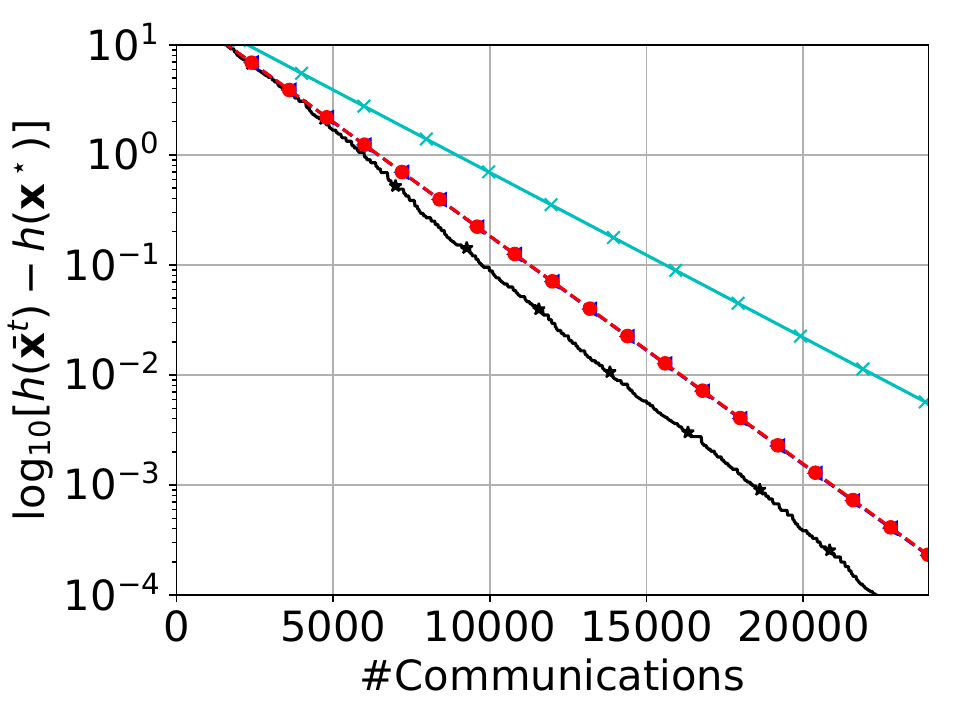}  \\
      \includegraphics[scale=0.233]{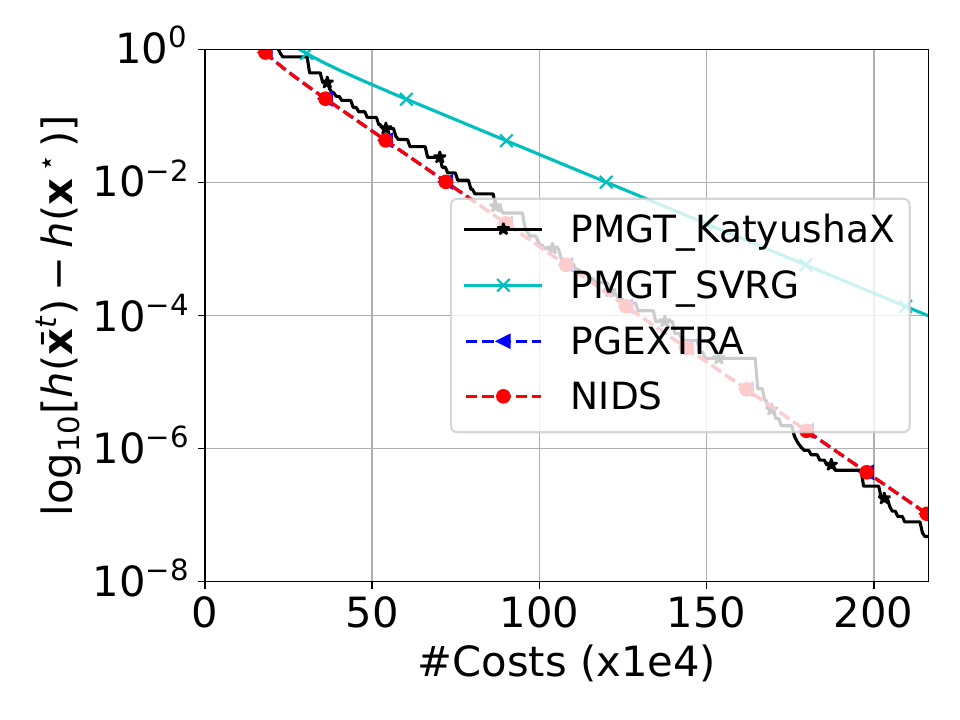}
     & \includegraphics[scale=0.233]{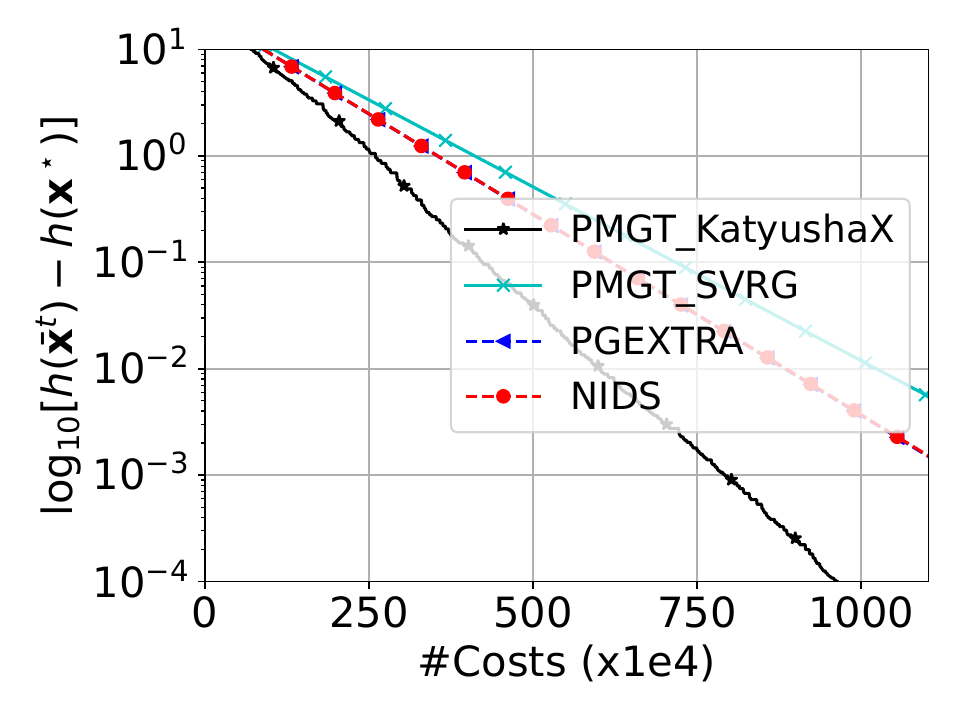}  \\
\end{tabular}
    \caption{Performance comparison between PGEXTRA, NIDS, PMGT-SVRG, and PMGT-KatyushaX on the Covtype dataset. The left column represents results with the ratio $r=2$ and the right column represents results with the ratio $r=300$ defined in~(\ref{exp_obj}). The plot of PGEXTRA and NIDS are overlapped as their performance are close to each other.} 
\label{covtype_result}
\end{figure}

We conduct our experiments on both synthetic and real-world datasets. 
For the synthetic dataset, we generate a Bernoulli matrix of size $60,000 \times 50$ with entries in $\{ \pm 1\}$; 
For the real-world dataset, we use the Covtype downloaded from the LIBSVM website\footnote{\url{https://www.csie.ntu.edu.tw/~cjlin/libsvmtools/datasets}}.
For the gossip matrix $W$ underlying the decentralized network, we generate a matrix where each agent is randomly connected to two neighbors. 
The second largest eigenvalue of the resulting matrix is $\lambda_2(W) = 0.97$. 

\paragraph{Baselines} We compare the empirical performance of our proposed method with several baselines including PMGT-SVRG (\cite{ye2020pmgt}), PGEXTRA \cite{shi2015proximal}, and NIDS \cite{li2019decentralized}. 
Although theoretical results of PGEXTRA and NIDS are developed for the sum-of-convex objective functions, they still show convergence behavior on the sum-of-nonconvex functions empirically.

\paragraph{Experiment Specifications}  For all of the experiments, the y-axis represents the suboptimality of function value, i.e., $F(\bar{y}) - F(x^*)$ where $F(x^*)$ is taken as the lowest value achieved among all the baselines.
 The left and the middle figure represents the suboptimality vs. gradient evaluations and communication rounds, respectively. 
 The right figure shows the suboptimality vs. the cost which is the weighted average between gradient evaluation and communication rounds.

\paragraph{Comparison Results on the Synthetic Dataset}  Figure~\ref{synthetic_result} reports the performance comparison between our PMGT-KatyushaX and baselines when the ratio $r=2$ or $r=300$.
PMGT-KatyushaX outperforms other baselines for both settings, even when the ratio is small. We point out that PGEXTRA and NIDS have similar performance so their curves are overlapped in the performance comparison. 
For each setting, PMGT-KatyushaX makes the least number of gradient evaluations compared with other baselines. 
To our surprise, it even requires fewer communication rounds to converge than  NIDS although the communication bound of PMGT-KatyushaX depends on the number of component functions~$\sqrt{n}$. 
As a result, PMGT-KatyushaX achieves the best cost among all baselines.

\paragraph{Comparison Results on the Real-world Dataset}  Figure~\ref{covtype_result} reports the performance comparison between PMGT-KatyushaX and other baselines on the Covtype dataset when the ratio $r=2$ or $r=300$. 
When $r$ is small, PMGT-KatyushaX has fewer gradient evaluations than NIDS/PGEXTRA, but PMGT-KatyushaX requires more communication rounds than NIDS/PGEXTRA. 
When $r$ is large, PMGT-KatyushaX outperforms other baselines in both the number of gradient evaluations and communication rounds.    

\section{Conclusions and Future Work}\label{sec:conclusion}
This paper presented a new theoretical analysis of PMGT-SVRG to the decentralized sum-of-nonconvex problem, and it has a linear dependency on the condition number.
To achieve a better dependency, we proposed the first accelerated stochastic first-order algorithm for the decentralized sum-of-nonconvex problem and showed it enjoys a linear convergence rate with a square-root dependency on the condition number. 
The empirical evidence validates the advantages of our algorithm on both synthetic and real-world datasets.  

Although the computational complexity of the proposed algorithm has matched the best-known algorithm for the centralized scenario, the communication upper bound still looks unsatisfying due to the inclusion of the factor $\sqrt{n}$. 
It is interesting to study how to design a more communication-efficient algorithm for decentralized sum-of-nonconvex optimization.

\section*{Acknowledgments}
This research/project is supported by the National Research Foundation, Singapore under its AI Singapore Programme (AISG Award No: AISG2-PhD-2023-08-043T-J).
This research is part of the programme DesCartes and is supported by the National Research Foundation, Prime Minister’s Office, Singapore under its Campus for Research Excellence and Technological Enterprise (CREATE) programme.

\bibliography{aaai24}
\appendix\onecolumn
The appendix is organized as below.
In Section \ref{appendix:svrg_epoch}, we provide theoretical analysis for one epoch of the SVRG algorithm, which is the foundation of both PMGT-SVRG and PMGT-KatyushaX algorithms.
Section \ref{appendix:pmgt_katyusha_analysis} presents the concrete theoretical analysis of the PMGT-KatyushaX algorithm on the sum-of-nonconvex problem. 
Section \ref{appendix:svrg_analysis} provides an alternative proof of the centralized SVRG algorithm.
Section \ref{appendix:pmgt_svrg_analysis} extends the convergence results for the single machine setting to the decentralized setting.
Finally, the extension of the theoretical analysis of the proposed algorithms to the case that the global objective function is convex is presented in Section \ref{appendix:general_convex_ext}.

\section{Theoretical Analysis for One Epoch of SVRG}\label{appendix:svrg_epoch}
In this section, we analyze the convergence for one epoch of SVRG, which is the building stone for the analysis of both PMGT-SVRG and PMGT-KatyushaX algorithms.

\subsection{Established Results}
We introduce several established results from existing literature.
\begin{lemma}[\citet{allen2018katyusha}]
Given any sequence $D_0,D_1, \ldots$ of reals, if $N \sim \texttt{Geom}(p)$, then
\begin{equation}
\mathbb{E}_N[D_N - D_{N+1}] = \frac{p}{1-p} \mathbb{E}[D_0 - D_N]   
\end{equation}
and
\begin{equation}
\mathbb{E}_N[D_N] = (1-p)\mathbb{E}[D_{N+1}] + p D_0.
\end{equation}
\end{lemma}
\begin{lemma}[\citet{allen2018katyusha}]\label{smooth_bound}
If function $f_{i,j}$ is $(\ell_1, \ell_2)$-smooth, then it holds
\begin{equation}
    \|\nabla f_{i,j}(x) - \nabla f_{i,j}(y)\|^2 \leq (\ell_1 - \ell_2) \langle \nabla f_{i,j}(x) - \nabla f_{i,j}(y), x- y\rangle + \ell_1 \ell_2 \|x - y\|^2.
\end{equation}
\end{lemma}
\begin{lemma}[\citet{allen2018katyusha}]
If $\tilde{\psi}(\cdot)$ is a proper convex and $\sigma$-strongly convex function and $\hat{w} = \mathbf{prox}_{\hat{\eta}, \tilde{\psi}}(\tilde{w} - \hat{\eta} s^t) = \argmin_{w\in \mathcal{R}^d}{\{\frac{1}{2\hat{\eta}} \|w- \tilde{w}\|^2 + \langle s^t, w\rangle + \tilde{\psi}(w)\}}$, then for $\forall u \in \mathcal{R}^d$, we have
\begin{equation}
    \begin{split}
        \langle s^t, \tilde{w} - u \rangle + \tilde{\psi}(\hat{w}) - \tilde{\psi}(u) &\leq \langle s^t, \tilde{w} - \hat{w}\rangle + \frac{\|u - \tilde{w}\|^2}{2\hat{\eta}} - \frac{(1+ \sigma \hat{\eta})\|u - \hat{w}\|^2}{2\hat{\eta}} - \frac{\|\tilde{w} - \hat{w}\|^2}{2 \hat{\eta}} \\
        & \leq \frac{\hat{\eta}}{2}\|s^t\|^2 + \frac{\|u - \tilde{w}\|^2}{2 \hat{\eta}} - \frac{(1+ \sigma \hat{\eta})\|u - \hat{w}\|^2}{2 \eta}.
    \end{split}
\end{equation}
\label{mirror_desc}
\end{lemma}

\begin{lemma}[\citet{ye2020pmgt}]\label{decentral_grad}
Suppose each function $f_i$ is $(\ell_1, \ell_2)$-smooth for $\ell_2 \geq \ell_1 >0$ and the update of $\vs^t$ follows
\begin{align*}
\mathbf{s}^{t} = \FastMix(\mathbf{s}^{t-1} + \mathbf{v}^{t} - \mathbf{v}^{t-1}, M),
\end{align*}
then it holds that $\bar{s}^{t} = \bar{v}^t$, $\mathbb{E}[\bar{s}^{t}] = \frac{1}{m}\sum_{i=1}^m \nabla f_i(w_i^t)$ and
\begin{equation}
    \|\nabla f(\bar{w}^t) - \mathbb{E}(\bar{s}^{t})\| \leq \frac{\sqrt{\ell_1\ell_2}}{\sqrt{m}} \|\vw^t - \mathbf{1}\bar{w}^t\|.
\end{equation}
\end{lemma}
\begin{lemma}[\citet{ye2020pmgt}]\label{prox_lemma}
Let $\mathbf{prox}_{m \eta, \Psi}(\vw)$ and $\mathbf{prox}_{\eta, \psi}(w)$ be defined as 
\begin{align*}
\bf{prox}_{m \eta, \Psi}(\vw) = \argmin_{\mathbf{z} \in \mathcal{R}^{m \times d}} (\Psi(\mathbf{z}) + \frac{1}{2 m \eta}\|\mathbf{z} - \vw\|^2),
\end{align*}
and
\begin{align*}
\bf{prox}_{\eta, \psi}(w) = \argmin_{z\in \mathcal{R}^d}(\psi(z) + \frac{1}{2\eta}\|z - w\|^2)   
\end{align*}
respectively. Then we have the following relationship:
\begin{equation}
    \mathbf{prox}_{m \eta, \Psi}^{(i)}(\vw) =  \mathbf{prox}_{\eta, \psi}(\vw_i)
\end{equation}
and
\begin{equation}
    \norm{\mathbf{prox}_{m \eta, \Psi}(\frac{1}{m}\mathbf{1}\mathbf{1}^{\top}\vw) - \frac{1}{m}\mathbf{1}\mathbf{1}^{\top} \mathbf{prox}_{m \eta, \Psi}(\vw)}^2 \leq 
     \norm{\vw - \frac{1}{m}\mathbf{1}\mathbf{1}^\top\vw}^2.
\end{equation}
\end{lemma}
\subsection{Supporting Lemmas}
In this subsection, we are aimed to obtain the theoretical results for one epoch of the SVRG algorithm.
First, we show the relationships of the consensus error between successive iterations:
\begin{lemma}
Let $\vw^{t}$, $\vs^{t}$ be the $t$-th iterate of local weight variable and gradient tracking variable in the inner loop of Algorithm \ref{strongconvex_algo_2}, then we have
\begin{equation}
    \|\mathbf{w^{t+1}} - \mathbf{1}\bar{w}^{t+1}\|^2  \leq 8\rho^2  \|\vw^t - \mathbf{1}\bar{w}^t\|^2 + 8\rho^2\eta^2\|\mathbf{s}^{t} - \mathbf{1}\bar{s}^{t}\|^2
\end{equation}
and
\begin{equation}
    \|\mathbf{s^{t+1}} - \mathbf{1}\bar{s}^{t+1}\|^2 \leq 2\rho^2 \|\mathbf{s}^t - \mathbf{1} \bar{s}^t\|^2 + 2\rho^2 \| \mathbf{v}^{t+1} - \mathbf{v}^{t}\|^2.
    \label{eq:grad_track_inner}
\end{equation}
\end{lemma}

\begin{proof}
We can bound the difference between local variable $\mathbf{w}^{t+1}$ with the average of local variables $\bar{w}^{t+1}$ with:
\begin{align*}
    \|\mathbf{w}^{t+1} - \mathbf{1}\bar{w}^{t+1}\|^2
     \leq &\rho^2 \norm{ \mathbf{prox}_{m\eta, \Psi} (\vw^t - \eta \mathbf{s}^{t}) - \frac{1}{m} \mathbf{1} \mathbf{1}^{\top} \mathbf{prox}_{m\eta, \Psi} (\vw^t - \eta \mathbf{s}^{t})}^2 \\
     \leq & 2\rho^2 \norm{ \mathbf{prox}_{m\eta, \Psi} (\vw^t - \eta \mathbf{s}^{t}) -  \mathbf{prox}_{m\eta, \Psi} (\frac{1}{m} \mathbf{1} \mathbf{1}^{\top}\vw^t - \frac{1}{m} \mathbf{1} \mathbf{1}^{\top}\eta \mathbf{s}^{t})}^2  \\
    & + 2\rho^2 \norm{\mathbf{prox}_{m\eta, \Psi} (\frac{1}{m} \mathbf{1} \mathbf{1}^{\top}\vw^t - \frac{1}{m} \mathbf{1} \mathbf{1}^{\top}\eta \mathbf{s}^{t})  - \frac{1}{m} \mathbf{1} \mathbf{1}^{\top} \mathbf{prox}_{m\eta, \Psi} (\vw^t - \eta \mathbf{s}^{t})}^2 \\
     \leq  & 2\rho^2 \norm{\vw^t - \eta \mathbf{s}^{t} - \mathbf{1}\bar{w}^t + \eta \mathbf{1} \bar{s}^{t}}^2 + 2\rho^2\norm{\vw^t - \eta \mathbf{s}^{t} - \frac{1}{m} \mathbf{1}\mathbf{1}^{\top}(\vw^t - \eta \mathbf{s}^{t})}^2 \\
    \leq  & 8\rho^2  \|\vw^t - \mathbf{1}\bar{w}^t\|^2 + 8\rho^2\eta^2\|\mathbf{s}^{t} - \mathbf{1}\bar{s}^{t}\|^2,
\end{align*}
where the first inequality uses Lemma \ref{fastmix_lemma}; the second and the last inequality use $\|a + b\|^2 \leq 2 \|a\|^2 + 2 \|b\|^2$; the third inequality uses Lemma \ref{prox_lemma}. Then we have
\begin{align*}
    \|\mathbf{s^{t+1}} - \mathbf{1}\bar{s}^{t+1}\|^2 &\leq \rho^2 \norm{ \mathbf{s}^t + \mathbf{v}^{t+1} - \mathbf{v}^{t} - \frac{1}{m} \mathbf{1} \mathbf{1}^{\top} (\mathbf{s}^t + \mathbf{v}^{t+1} - \mathbf{v}^{t})}^2 \\
    & \leq 2\rho^2 \|\mathbf{s}^t - \mathbf{1} \bar{s}^t\|^2 + 2\rho^2 \| \mathbf{v}^{t+1} - \mathbf{v}^{t}\|^2,
\end{align*}
where the last inequality follows from $\norm{\mathbf{x} - \frac{1}{m} \mathbf{1} \mathbf{1}^{\top} \mathbf{x}} \leq \|\mathbf{x}\|$ for $\forall \mathbf{x} \in \mathcal{R}^{m\times d}$.
\end{proof}

We show the quantity $\norm{\vv^{t+1} - \vv^{t}}^2$ that appears in the right-hand side of inequality \ref{eq:grad_track_inner} can be expressed as:
\begin{lemma}
Let $\mathbf{v}_i^t = \mu_i + \frac{1}{b}\sum_{j_i\in S_i^t}(\nabla f_{i, j_i}(w_i^t) - \nabla f_{i,j_i} (w_i^0))$ where $S_i^t$ is a set of uniformly sampled indices of size of $b$, then for $\forall u\in \mathcal{R}^d$,  we have
\begin{align*}
    \mathbb{E}\left[\| \mathbf{v}^{t+1} - \mathbf{v}^t\|^2\right]  \leq & {8\ell_1\ell_2}\left((1 + 8\rho^2)\|\vw^t - \mathbf{1}\bar{w}^t\|^2 + \|\vw^0 - \mathbf{1}u\|^2\right) \\
    & + {64 \ell_1\ell_2\rho^2\eta^2}\|\mathbf{s}^t - \mathbf{1}\bar{s}^t\|^2 + {8\ell_1\ell_2m}\left(\|\bar{w}^t - u\|^2   + \|\bar{w}^{t+1} - u\|^2\right).
\end{align*}
\end{lemma}
\begin{proof}
Denote $S^t$ as the aggregated set of $S_i^t$ with $i\in[m]$, we claim that
\begin{align*}
    & \mathbb{E}_{B^t, B^{t+1}}\left[\| \mathbf{v}^{t+1} - \mathbf{v}^t\|^2\right] \\ \leq & \mathbb{E}_{B^t, B^{t+1}}\left[ \sum_{i=1}^m \|\mathbf{v}_i^{t+1} - \mathbf{v}_i^{t}\|^2\right] \\
    = &  \mathbb{E}_{B^t, B^{t+1}} \sum_{i=1}^m\left[\left\|\mu_i + \frac{1}{b}\sum_{j_i\in S_i^{t+1}}\left(\nabla f_{i, j_i}(\vw_i^{t+1}) - \nabla f_{i,j_i} (\vw_i^0)\right) \right. \right.
      \left. \left. - \left(\mu_i + \frac{1}{b}\sum_{j_i\in S_i^t}\left(\nabla f_{i, j_i}(\vw_i^t) - \nabla f_{i,j_i} (\vw_i^0)\right)\right)\right\|^2\right] \\
    \leq & 2\mathbb{E}_{B^t, B^{t+1}} \sum_{i=1}^m\left[\norm{\frac{1}{b}\sum_{j_i\in S_i^{t+1}}(\nabla f_{i, j_i}(\vw_i^{t+1}) - \nabla f_{i,j_i} (\vw_i^0))}^2 \right.  \left. + \norm{\frac{1}{b}\sum_{j_i\in S_i^t}\left(\nabla f_{i, j_i}(\vw_i^t) - \nabla f_{i,j_i} (\vw_i^0)\right)}^2 \right] \\
    \leq & 2\mathbb{E}_{B^t, B^{t+1}} \sum_{i=1}^m\left[\frac{1}{b}\sum_{j_i\in S_i^{t+1}}\norm{\nabla f_{i, j_i}(\vw_i^{t+1}) - \nabla f_{i,j_i} (\vw_i^0)}^2 \right.  \left. + \frac{1}{b}\sum_{j_i\in S_i^t}\norm{\nabla f_{i, j_i}(\vw_i^t) - \nabla f_{i,j_i} (\vw_i^0)}^2 \right] \\
     = & \frac{2}{n} \sum_{i=1}^m \sum_{j=1}^n \left(\norm{\nabla f_{i, j}(\vw_i^{t+1}) - \nabla f_{i,j} (\vw_i^0)}^2 + \norm{\nabla f_{i, j}(\vw_i^{t}) - \nabla f_{i,j} (\vw_i^0)}^2 \right),
\end{align*}
where second inequality is due to the cancellation of the $\mu_i$ term and $\|a + b\|^2 \leq 2 \|a\|^2 + 2 \|b\|^2$ for $\forall a, b \in \mathcal{R}^d$; the third inequality uses the bound $\norm{\sum_{i=1}^b a_i}^2 \leq b\sum_{i=1}^b \norm{a_i}^2$ for any sequence $a_i \in \mathcal{R}^d$; the last equality follows that $S^{t}$ and $S^{t+1}$ are samples from uniform distribution.

For any $u \in \mathcal{R}^d$, we have
\begin{align*}
    \norm{\nabla f_{i, j}(\vw_i^{t}) - \nabla f_{i,j} (\vw_i^0)}^2 \leq & 2\|\nabla f_{i, j}(\vw_i^{t}) - \nabla f_{i, j}(u)\|^2 + 2\|\nabla f_{i, j}(u)- \nabla f_{i,j} (\vw_i^0)\|^2 \\
    \leq & 4\|\nabla f_{i, j}(\vw_i^{t}) - \nabla f_{i, j}(\bar{w}^{t})\|^2 + 4\|\nabla f_{i, j}(\bar{w}^{t}) - \nabla f_{i, j}(u)\|^2 \\
    & + 2 \|\nabla f_{i, j}(u)- \nabla f_{i,j} (\vw_i^0)\|^2 
\end{align*}
and
\begin{align*}
    \|\nabla f_{i, j}(\vw_i^{t+1}) - \nabla f_{i,j} (\vw_i^0)\|^2 \leq & 2\|\nabla f_{i, j}(\vw_i^{t+1}) - \nabla f_{i, j}(u)\|^2 + 2\|\nabla f_{i, j}(u)- \nabla f_{i,j} (\vw_i^0)\|^2 \\
    \leq & 4\|\nabla f_{i, j}(\vw_i^{t+1}) - \nabla f_{i, j}(\bar{w}^{t})\|^2 + 4\|\nabla f_{i, j}(\bar{w}^{t+1}) - \nabla f_{i, j}(u)\|^2 \\
    & + 2 \|\nabla f_{i, j}(u)- \nabla f_{i,j} (\vw_i^0)\|^2.
\end{align*}
Therefore, we achieve
\begin{align*}
     & \mathbb{E}_{S^t, S^{t+1}}\left[\norm{ \mathbf{v}^{t+1} - \mathbf{v}^t}^2\right] \\
     \leq &\frac{8}{n} \sum_{i=1}^m \sum_{j=1}^n \left(\|\nabla f_{i, j}(\vw_i^{t}) - \nabla f_{i, j}(\bar{w}^{t})\|^2 + \|\nabla f_{i, j}(\bar{w}^{t}) - \nabla f_{i, j}(u)\|^2 \right.\\
    & +  \|\nabla f_{i, j}(u)- \nabla f_{i,j} (\vw_i^0)\|^2 + \|\nabla f_{i, j}(\vw_i^{t+1}) - \nabla f_{i, j}(\bar{w}^{t+1})\|^2 \\
    & \left. + \|\nabla f_{i, j}(\bar{w}^{t+1}) - \nabla f_{i, j}(u)\|^2 \right) \\
    \leq & \frac{8\ell_1\ell_2}{n} \sum_{i=1}^m \sum_{j=1}^n \left(\|\vw_i^{t} - \bar{w}^{t}\|^2 + \|\bar{w}^{t} - u\|^2 +  \|u- \vw_i^0\|^2 \right. \\
    &  \left. + \|\vw_i^{t+1} - \bar{w}^{t+1}\|^2 + \norm{ \bar{w}^{t+1} - u}^2 \right) \\ 
    = & {8\ell_1\ell_2}\left(\|\vw^t - \mathbf{1}\bar{w}^t\|^2 + \|\vw^0 - \mathbf{1}u\|^2  + \|\vw^{t+1} - \mathbf{1}\bar{w}^t\|^2\right) \\
    &+  8\ell_1\ell_2m\left(\|\bar{w}^t - u\|^2 + \|\bar{w}^{t+1} - u\|^2\right) \\
    \leq & {8\ell_1\ell_2}\left(\|\vw^t - \mathbf{1}\bar{w}^t\|^2 + \|\vw^0 - \mathbf{1}u\|^2 + (8\rho^2\|\vw^t - \mathbf{1}\bar{w}^t\|^2 \right. \\
    & \left. + 8\rho^2\eta^2\|\mathbf{s}^t - \mathbf{1}\bar{s}^t\|^2)\right) + {8\ell_1\ell_2m}(\|\bar{w}^t - u\|^2  + \|\bar{w}^{t+1} - u\|^2 ) \\
    \leq & {8\ell_1\ell_2}\left((1 + 8\rho^2)\|\vw^t - \mathbf{1}\bar{w}^t\|^2 + \|\vw^0 - \mathbf{1}u\|^2\right) \\
    & + {64 \ell_1\ell_2\rho^2\eta^2}\|\mathbf{s}^t - \mathbf{1}\bar{s}^t\|^2 + {8\ell_1\ell_2m}\left(\|\bar{w}^t - u\|^2   + \|\bar{w}^{t+1} - u\|^2\right).
\end{align*}
\end{proof}
Combining the above two lemmas, we arrive at the following linear system of inequalities for $\|\vw^{t+1} - \mathbf{1}\bar{w}^{t+1}\|^2$ and
$\eta^2\|\mathbf{s}^{t+1} - \mathbf{1}\bar{s}^{t+1}\|^2$:
\begin{align}
\begin{split}
    & \begin{pmatrix}
           \|\vw^{t+1} - \mathbf{1}\bar{w}^{t+1}\|^2\\
           \eta^2\|\mathbf{s}^{t+1} - \mathbf{1}\bar{s}^{t+1}\|^2\\
    \end{pmatrix} \\
    \leq & 2\rho^2
    \begin{pmatrix}
    4 & 4\\
    {8\ell_1\ell_2\eta^2}(1+8\rho^2)& 1+{64 \ell_1 \ell_2\rho^2\eta^2}\\
    \end{pmatrix}
        \begin{pmatrix}
           \|\vw^{t} - \mathbf{1}\bar{w}^{t}\|^2\\
           \eta^2\|\mathbf{s}^{t} - \mathbf{1}\bar{s}^{t}\|^2\\
    \end{pmatrix} \\
    &+  \begin{pmatrix}
       0 \\
       {16\rho^2\eta^2\ell_1\ell_2}\|\vw^0 - \mathbf{1}u\|^2 + {16\ell_1\ell_2m\rho^2\eta^2}\|\bar{w}^t- u\|^2 + {16\ell_1\ell_2m\rho^2\eta^2}\|\bar{w}^{t+1}- u\|^2
    \end{pmatrix}.
    \end{split}
\end{align}
To simplify the notation, we abbreviate $\begin{pmatrix}
           \|\vw^{t+1} - \mathbf{1}\bar{w}^{t+1}\|^2\\
           \eta^2\|\mathbf{s}^{t+1} - \mathbf{1}\bar{s}^{t+1}\|^2\\
    \end{pmatrix}$ as $z^{t+1}$, the above inequality can be simplified as:     
    \begin{equation}
    \begin{split}
        &z^{t+1} \leq \mathbf{A}z^t + b^t,
    \end{split}
    \end{equation}    
where the matrix $\mathbf{A}$ and the vector $b^t$ are defined as
        \begin{equation}
    \begin{split}
     \mathbf{A} &= 2\rho^2 \begin{pmatrix}
    4 & 4\\
    {8\ell_1\ell_2\eta^2}(1+8\rho^2)& 1+{64 \ell_1 \ell_2\rho^2\eta^2}\\
    \end{pmatrix} 
    \end{split}     
    \label{A_def}
    \end{equation}
and
    \begin{equation}
        \begin{split}
     \mathbf{b}^t&=\begin{pmatrix}
       0 \\
       {16\rho^2\eta^2\ell_1\ell_2}\|\vw^0 - \mathbf{1}u\|^2 + {16\ell_1\ell_2m\rho^2\eta^2}\|\bar{w}^t- u\|^2 + {16\ell_1\ell_2m\rho^2\eta^2}\|\bar{w}^{t+1}- u\|^2
    \end{pmatrix}.
    \end{split}
    \label{b_def}
\end{equation}
Based on the construction of the linear system of inequalities, we have:
\begin{lemma}
Let $z^t = \begin{pmatrix}
           \|\vw^{t+1} - \mathbf{1}\bar{w}^{t+1}\|^2\\
           \eta^2\|\mathbf{s}^{t+1} - \mathbf{1}\bar{s}^{t+1}\|^2\\
    \end{pmatrix}$, and it enjoys linear system of inequalities such that $z^{t+1} \leq A z^{t} + b^t$ where $A$ is defined in \eqref{A_def} and $b^t$ is defined by \eqref{b_def}. Given $T \sim \rm Geom(p)$, it follows that:
    \begin{equation}
    \begin{split}
     \mathbb{E}_T\left[\|{z}^{T} \|\right] \leq \|{z}^0\| +  \frac{1-p}{p}{16\rho^2\eta^2\ell_1\ell_2}\|\vw^0 - \mathbf{1}\bar{w}^0\|^2   + \frac{2-p}{p}{16\rho^2\eta^2 \ell_1\ell_2m}\mathbb{E}_T\left[\|\bar{w}^{T}- \bar{w}^0\|^2\right].
    \end{split}
    \end{equation}
    \label{inner_lemma_1}
\end{lemma}
\begin{proof}
We have ${z}^{t+1} = \mathbf{A}{z}^t + {b}^t = \mathbf{A}^{t+1}{z}^0 + \sum_{i=0}^t \mA^{t-i}b^i$. Note that $\mathbf{A}^i$ is the $i$-th power of the matrix $\mathbf{A}$, $b^i$ is expressed in the same way as $b^t$ except $\bar{w}^t$ and $\bar{w}^{t+1}$ is replaced with $\bar{w}^i$ and $\bar{w}^{i+1}$, then
\begin{equation}
    \|{z}^{t+1}\| \leq \|\mathbf{A}\|_2^{t+1}\|{z}^0\| + \sum_{i=0}^t \|\mathbf{A}\|_2^{t-i}\|b^i\|.
\label{z_split}
\end{equation}
Also by the definition of $\mathbf{b}_i$ and triangle inequality of norm, we have
\begin{equation}
\begin{split}
    \|\mathbf{b}^i\|=& \norm{
       {16\rho^2\eta^2\ell_1\ell_2}\|\vw^0 - \mathbf{1}u\|^2 + {16\ell_1\ell_2m\rho^2\eta^2}\|\bar{w}^i- u\|^2 + {16\ell_1\ell_2m\rho^2\eta^2}\|\bar{w}^{i+1}- u\|^2
}\\ 
    \leq  & 
       {16\rho^2\eta^2\ell_1\ell_2}\|\vw^0 - \mathbf{1}u\|^2
     + 
        {16\ell_1\ell_2m\rho^2\eta^2}\|\bar{w}^i- u\|^2
     +  
       {16\ell_1\ell_2m\rho^2\eta^2}\|\bar{w}^{i+1}- u\|^2. 
    \end{split}
\label{b_split}
\end{equation}
By combining \ref{z_split} and \ref{b_split}, we claim that 
\begin{equation}
\begin{split}
    \|{z}^{t+1}\| \leq &\|\mathbf{A}\|_2^{t+1}\|{z}^0\| + \sum_{i=0}^t 
       {16\rho^2\eta^2\ell_1\ell_2}\|\mathbf{A}\|_2^i\|\vw^0 - \mathbf{1}u\|^2  \\
   & +  \sum_{i=1}^t (\|\mathbf{A}\|_2^{t-i} + \|\mathbf{A}\|_2^{t-i+1})
        {16\ell_1\ell_2m\rho^2\eta^2}\|\bar{w}^i- u\|^2 \\
   & +         {16\ell_1\ell_2m\rho^2\eta^2}\|\mathbf{A}\|_2^{t}\|\bar{w}^0- u\|^2    +
    {16\ell_1\ell_2m\rho^2\eta^2}\|\bar{w}^{t+1}- u\|^2.
    \end{split}
\end{equation}
Let $T\sim \texttt{Geom}(p)$ follows from the geometric distribution where $p\in [0, 1]$, it follows that 
\begin{equation}
    \begin{split}
        \mathbb{E}_T[\|{z}^{T} \|] 
        \leq & \sum_{t\geq0}(1-p)^t p(\|\mathbf{A}\|_2^{t}\|{z}^0\| + \frac{1-\|\mathbf{A}\|_2^t}{1-\|\mathbf{A}\|_2}
       {16\rho^2\eta^2\ell_1\ell_2}\|\vw^0 - \mathbf{1}u\|^2 
    \\
    &+  \sum_{i=1}^{t-1} (\|\mathbf{A}\|_2^{t-i} + \|\mathbf{A}\|_2^{t-i-1})  {16\ell_1\ell_2m\rho^2\eta^2}\|\bar{w}^i- u\|^2
  \\
    &+ \|\mathbf{A}\|_2^{t-1}
        {16\ell_1\ell_2m\rho^2\eta^2}\|\bar{w}^0- u\|^2
    +    {16\ell_1\ell_2m\rho^2\eta^2}\|\bar{w}^{t}- u\|^2
    ).
    \end{split}
\end{equation}
We will show that by carefully choosing $\rho$. We can bound $0 < \|\mathbf{A}\|_2 < 1$ and we replace $u$ with $\bar{w}^0$. By elementary mathematics, it follows that
\begin{equation}
\begin{split}\label{inner_comm}
       & \mathbb{E}_T\left[\|{z}^{T} \|\right]  \\
       \leq & \frac{p}{1-(1-p)\|\mathbf{A}\|_2}\|{z}^0\| +  \frac{1-p}{1-(1-p)\|\mathbf{A}\|_2}{16\rho^2\eta^2\ell_1\ell_2}\|\vw^0 - \mathbf{1}\bar{w}^0\|^2\\
         & +  \sum_{i>0}\sum_{t\geq i} \left((1 - p)^t p \|\mathbf{A}\|_2^{t-i} + (1 - p)^{t+1} p \|\mathbf{A}\|_2^{t-i}\right){16\ell_1\ell_2m\rho^2\eta^2}\|\bar{w}^{i}- \bar{w}^0\|^2\\
        \leq & \frac{p}{1-(1-p)\|\mathbf{A}\|_2}\|{z}^0\| +  \frac{1-p}{1-(1-p)\|\mathbf{A}\|_2}{16\rho^2\eta^2\ell_1\ell_2}\|\vw^0 - \mathbf{1}\bar{w}^0\|^2\\
         & + \sum_{i>0} \frac{(2-p)(1-p)^i p}{1-(1-p)\|\mathbf{A}\|_2}{16\ell_1\ell_2m\rho^2\eta^2}\|\bar{w}^{i}- \bar{w}^0\|^2\\
    = & \mathbb{E}_T\left[\frac{p}{1-(1-p)\|\mathbf{A}\|_2}\|{z}^0\| +  \frac{1-p}{1-(1-p)\|\mathbf{A}\|_2}{16\rho^2\eta^2\ell_1\ell_2}\|\vw^0 - \mathbf{1}\bar{w}^0\|^2 \right.\\
    & \left. + \frac{2-p}{1-(1-p)\|\mathbf{A}\|_2}{16\ell_1\ell_2m\rho^2\eta^2}\|\bar{w}^{T}- \bar{w}^0\|^2\right] .
\end{split}
\end{equation}
\end{proof}

One can obtain the following key result by using the definition of convexity and smoothness of function $f(\cdot)$:

\begin{lemma}
Let $\bar{w}^{t+1} = \frac{1}{m}\mathbf{1}^{\top}\vw^{t+1}$, where $\vw^{t+1} = \FastMix(\mathbf{prox}_{m \eta , \Psi}(w^t - \eta s^t), M)$, it satisfies that
\begin{equation}
\begin{split}
      \mathbb{E}\left[F(\bar{w}^{t+1}) - F(u)\right] &\leq \mathbb{E} \left[  \frac{\eta}{1 - \eta L}\norm{ \bar{s}^t- \nabla f(\bar{w}^t)}^2 + \frac{2-\eta \sigma_f}{4\eta}\|u - \bar{w}^t\|^2  -  \frac{(1+\sigma_{\psi} \eta)\|u - \bar{w}^{t+1}\|^2}{2\eta} \right.\\
    & \quad \left. +  \frac{1-\eta L + 2\eta}{2(1-\eta L)m}\|\mathbf{1}\bar{s}^t - \mathbf{s}^t\|^2 + \left(\frac{\ell_1\ell_2}{\sigma_f m}+ \frac{1+\eta}{2m\eta}\right)\|\vw^t - \mathbf{1}\bar{w}^t \|^2 \right]. 
\end{split}
\label{cur_iterate_inner_bound}
\end{equation}
\label{lemma:inner_lemma_2}
\end{lemma}
\begin{proof}
Denote $\Tilde{\vw}^{t+1} = \mathbf{prox}_{m \eta , \Psi}(w^t - \eta s^t)$, so $\vw^{t+1}= \FastMix(\Tilde{\vw}^{t+1}, M)$ then we have
\begin{align*}
&    \mathbb{E}\left[F(\bar{w}^{t+1}) - F(u)\right] \\
=& \mathbb{E}\left[f(\bar{w}^{t+1}) - f(u) + \psi(\bar{w}^{t+1}) - \psi(u)\right] \\
\leq & \mathbb{E}\left[ f(\bar{w}^{t}) + \langle \nabla f(\bar{w}^t), \bar{w}^{t+1} - \bar{w}^t \rangle + \frac{L}{2}\|\bar{w}^{t+1} - \bar{w}^t\|^2 - f(u) \right. \left. + \psi(\bar{w}^{t+1}) - \psi(u)\right] \\
\leq & \mathbb{E}\left[\langle \nabla f(\bar{w}^t), \bar{w}^t - u \rangle + \psi(\bar{w}^{t+1}) - \psi(u)  - \frac{\sigma_f}{2}\| \bar{w}^t - u\|^2  \right. \left. + \langle \nabla f(\bar{w}^t), \bar{w}^{t+1} - \bar{w}^t \rangle + \frac{L}{2}\|\bar{w}^{t+1} - \bar{w}^t\|^2\right] \\
\leq &\mathbb{E}\left[\frac{1}{m}\sum_{i=1}^m \left(\langle \nabla f(\bar{w}^t), \vw_i^t - u \rangle + \psi(\Tilde{\vw}_i^{t+1}) - \psi(u)\right)  - \frac{\sigma_f}{2}\| \bar{w}^t - u\|^2 \right. \left. + \langle \nabla f(\bar{w}^t), \bar{w}^{t+1} - \bar{w}^t \rangle + \frac{L}{2}\|\bar{w}^{t+1} - \bar{w}^t\|^2 \right], 
\end{align*}
where the first inequality is due to the $L$-smoothness of $f(\cdot)$, the second inequality is due to the strong convexity of $f(\cdot)$. The last inequality holds because $\psi(\cdot)$ is a convex function. 

Then we have
\begin{align*}
     \mathbb{E}&\left[F(\bar{w}^{t+1}) - F(u)\right]  \\
     = \mathbb{E}&\left[\frac{1}{m}\sum_{i=1}^m \left(\langle \mathbf{s}_i^t, \vw_i^t - u \rangle + \langle \nabla f(\bar{w}^t) - \mathbf{s}_i^t, \vw_i^t - u\rangle + \psi(\Tilde{\vw}_i^{t+1}) - \psi(u)\right)  - \frac{\sigma_f}{2}\| \bar{w}^t - u\|^2 \right.\\
    & \quad \left. + \langle \nabla f(\bar{w}^t), \bar{w}^{t+1} - \bar{w}^t \rangle + \frac{L}{2}\|\bar{w}^{t+1} - \bar{w}^t\|^2 \right] \\
     \leq \mathbb{E}&\left[ \langle \nabla f(\bar{w}^t) - \bar{s}^t, \bar{w}^t - u \rangle + \frac{1}{m}\sum_{i=1}^m\langle \bar{s}^t - \mathbf{s}_i^t, \vw_i^t - u\rangle \right. \\
    & \quad  + \frac{1}{m} \sum_{i=1}^m \left(\langle \mathbf{s}_i^t, \vw_i^t - \Tilde{\vw}_i^{t+1} \rangle  + \frac{\|u - \vw_i^t\|^2}{2\eta} - \frac{(1+\sigma_{\psi} \eta)\|u - \Tilde{\vw}_i^{t+1}\|^2}{2\eta}   - \frac{\|\vw_i^t - \Tilde{\vw}_i^{t+1}\|^2}{2\eta}\right)\\
    &  \quad \left. - \frac{\sigma_f}{2}\| \bar{w}^t - u\|^2 + \langle \nabla f(\bar{w}^t), \bar{w}^{t+1} - \bar{w}^t \rangle  + \frac{L}{2}\|\bar{w}^{t+1} - \bar{w}^t\|^2 \right] \\
     \leq \mathbb{E}&\left[ \langle \nabla f(\bar{w}^t) - \bar{s}^t, \bar{w}^t - u \rangle + \frac{1}{m}\sum_{i=1}^m\langle \bar{s}^t - \mathbf{s}_i^t, \vw_i^t - \bar{w}^t\rangle \right. \\
    & \quad + \frac{1}{m} \sum_{i=1}^m \left(\langle \mathbf{s}_i^t - \nabla f(\bar{w}^t) , \vw_i^t - \Tilde{\vw}_i^{t+1} \rangle + \frac{\|u - \vw_i^t\|^2}{2\eta} - \frac{(1+\sigma_{\psi} \eta)\|u - \Tilde{\vw}_i^{t+1}\|^2}{2\eta}  - \frac{\|\vw_i^t - \Tilde{\vw}_i^{t+1}\|^2}{2\eta}\right)\\
    & \left. \quad - \frac{\sigma_f}{2}\| \bar{w}^t - u\|^2 + \frac{L}{2m} \sum_{i=1}^m\|\vw_i^t - \Tilde{\vw}_i^{t+1} \|^2 \right],  
\end{align*}
where the first inequality is deduced from Lemma \ref{mirror_desc}. It can be easily verified that $\frac{1}{m}\sum_{i=1}^m \langle \bar{s}^t - \mathbf{s}_i^t, \vw_i^t - u\rangle = \frac{1}{m}\sum_{i=1}^m \langle \bar{s}^t - \mathbf{s}_i^t, \vw_i^t - \bar{w}^t\rangle$ and $\|\bar{w}^{t+1} - \bar{w}^t\|^2 = \|\frac{1}{m}\sum_{i=1}^m (\vw_i^t - \Tilde{\vw}_i^{t+1})\|^2 \leq \frac{1}{m}\sum_{i=1}^m\|\vw_i^t - \Tilde{\vw}_i^{t+1}\|^2$, so the second inequality holds. Recall that $B^t$ is the batch of samples used at iteration $t$, we have
\begin{align*}
    \mathbb{E} & \left[  F(\bar{w}^{t+1}) - F(u)\right]  \\
    \leq \mathbb{E} & \left[  \langle \nabla f(\bar{w}^t) - \mathbb{E}_{B^t}[\bar{s}^t], \bar{w}^t - u\rangle  +  \frac{\|\mathbf{1}\bar{s}^t - \mathbf{s}^t\|^2 + \|\vw^t - \mathbf{1}\bar{w}^t \|^2}{2m} \right.\\
    &  + \frac{1}{m} \sum_{i=1}^m \left(\langle \mathbf{s}_i^t - \nabla f(\bar{w}^t) , \vw_i^t - \Tilde{\vw}_i^{t+1} \rangle - \frac{(1-\eta L)\|\vw_i^t - \Tilde{\vw}_i^{t+1}\|^2}{2\eta}  + \frac{\|u - \vw_i^t\|^2}{2\eta}\right)   \\ 
     &  \left. - \frac{(1+\sigma_{\psi} \eta)\|u - \bar{w}^{t+1}\|^2}{2\eta}  - \frac{\sigma_f}{2}\| \bar{w}^t - u\|^2 \right] \\
     \leq \mathbb{E} & \left[  \langle \nabla f(\bar{w}^t) - \mathbb{E}_{B^t}[\bar{s}^t], \bar{w}^t - u\rangle +  \frac{\|\mathbf{1}\bar{s}^t - \mathbf{s}^t\|^2 + \|\vw^t - \mathbf{1}\bar{w}^t \|^2}{2m} \right. \\
    &  \left. + \frac{\eta}{2(1 - \eta L)m} \|\mathbf{s}^t - \mathbf{1}\nabla f(\bar{w}^t)\|^2  + \frac{1}{m} \sum_{i=1}^m \frac{\|u - \vw_i^t\|^2}{2\eta} -  \frac{(1+\sigma_{\psi} \eta)\|u - \bar{w}^{t+1}\|^2}{2\eta} - \frac{\sigma_f}{2}\| \bar{w}^t - u\|^2 \right]  \\
      \leq \mathbb{E} & \left[  \norm{ \nabla f(\bar{w}^t) - \frac{1}{m} \sum_{i=1}^m \nabla f_i(\vw_i^t)}\norm{ \bar{w}^t - u} +   \frac{\|\mathbf{1}\bar{s}^t - \mathbf{s}^t\|^2 + \|\vw^t - \mathbf{1}\bar{w}^t \|^2}{2m} \right. \\
      & + \frac{\eta}{2(1 - \eta L)m} (2\|\mathbf{s}^t - \mathbf{1}\bar{s}^t\|^2 + 2m\| \bar{s}^t- \nabla f(\bar{w}^t)\|^2)  + \frac{1}{m} \sum_{i=1}^m \frac{\|u - \bar{w}^t\|^2 +  \|\bar{w}^t -\vw_i^t \|^2}{2\eta}  \\
    & \left. -  \frac{(1+\sigma_{\psi} \eta)\|u - \bar{w}^{t+1}\|^2}{2\eta} - \frac{\sigma_f}{2}\| \bar{w}^t - u\|^2 \right] \\
     = \mathbb{E} & \left[  \norm{ \nabla f(\bar{w}^t) - \frac{1}{m} \sum_{i=1}^m \nabla f_i(\vw_i^t)}\| \bar{w}^t - u\| +   \frac{\|\mathbf{1}\bar{s}^t - \mathbf{s}^t\|^2 + \|\vw^t - \mathbf{1}\bar{w}^t \|^2}{2m} \right. \\
    &  + \frac{\eta}{(1 - \eta L)m} \|\mathbf{s}^t - \mathbf{1}\bar{s}^t\|^2 + \frac{\eta}{1 - \eta L}\| \bar{s}^t- \nabla f(\bar{w}^t)\|^2
    + \frac{(1 - \sigma_f \eta)\|u - \bar{w}^t\|^2 }{2\eta}  \\
    & \left. -  \frac{(1+\sigma_{\psi} \eta)\|u - \bar{w}^{t+1}\|^2}{2\eta} +   \frac{\|\mathbf{1}\bar{w}^t -\vw^t \|^2}{2m\eta} \right] \\
     = \mathbb{E} & \left[ \left(\norm{ \nabla f(\bar{w}^t) - \frac{1}{m} \sum_{i=1}^m \nabla f_i(\vw_i^t)}\norm{ \bar{w}^t - u}  - \frac{\sigma_f}{4} \|u - \bar{w}^t\|^2\right) + \frac{2-\eta \sigma_f}{4\eta} \|u - \bar{w}^t\|^2 \right.\\
    &  + \frac{\eta}{1 - \eta L}\| \bar{s}^t- \nabla f(\bar{w}^t)\|^2 -  \frac{(1+\sigma_{\psi} \eta)\|u - \bar{w}^{t+1}\|^2}{2\eta}\\
    &  \left. +  \frac{1-\eta L + 2\eta}{2(1-\eta L)m}\|\mathbf{1}\bar{s}^t - \mathbf{s}^t\|^2 + \frac{1+ \eta}{2m\eta}\|\vw^t - \mathbf{1}\bar{w}^t \|^2 \right] \\
    \leq \mathbb{E} & \left[\frac{1}{\sigma_f}\norm{ \nabla f(\bar{w}^t) - \frac{1}{m} \sum_{i=1}^m \nabla f_i(\vw_i^t)}^2 + \frac{2-\eta \sigma_f}{4\eta}\norm{u - \bar{w}^t}^2 \right. \\
    &  + \frac{\eta}{1 - \eta L}\| \bar{s}^t- \nabla f(\bar{w}^t)\|^2 -  \frac{(1+\sigma_{\psi} \eta)\|u - \bar{w}^{t+1}\|^2}{2\eta}\\
    &  \left. +  \frac{1-\eta L + 2\eta}{2(1-\eta L)m}\|\mathbf{1}\bar{s}^t - \mathbf{s}^t\|^2 + \frac{1+ \eta}{2m\eta}\|\vw^t - \mathbf{1}\bar{w}^t \|^2 \right].  
\end{align*}
Applying Young's inequality and rearranging term, we reach the first inequality. The second and the last inequalities also follow from Young's inequality. Using $\sum_{i=1}^m \frac{\|\bar{w}^t - u\|^2 + \|\bar{w}^t - \vw^i_t\|^2 - \|u - \vw^i_t\|^2}{m} = \sum_{i=1}^m \frac{2\langle \bar{w}^t - u, \bar{w}^t - \vw^i_t \rangle}{m} = 0$, we have the third inequality. Using Lemma \ref{smooth_bound}
, we immediately have
\begin{align*}
    & \quad \mathbb{E}\left[\norm{ \nabla f(\bar{w}^t) - \frac{1}{m} \sum_{i=1}^m \nabla f_i(\vw_i^t)}^2 \right]  \leq  \mathbb{E}\left[\frac{1}{m} \sum_{i=1}^m \|\nabla f_i(\bar{w}^t) - \nabla f_i(\vw_i^t)\|^2 \right] \\
    & \leq \mathbb{E}\left[\frac{1}{m} \sum_{i=1}^m \norm{\frac{1}{n}\sum_{j=1}^n \left(\nabla f_{i,j}(\bar{w}^t) - \nabla f_{i,j}(\vw_i^t)\right)}^2\right] \\
    & \leq \mathbb{E}\left[\frac{\ell_1\ell_2}{m} \sum_{i=1}^m \|\bar{w}^t - \vw_i^t\|^2 \right]  = \mathbb{E}\left[\frac{\ell_1\ell_2}{m}  \|\mathbf{1}\bar{w}^t - \vw^t\|^2 \right]. 
\end{align*}

After combining the two inequalities above, we have
\begin{align*}
     \mathbb{E} &\left[F(\bar{w}^{t+1}) - F(u)\right]  \\
     \leq \mathbb{E} & \left[\frac{\ell_1\ell_2}{\sigma_f m} \|\mathbf{1}\bar{w}^t - \vw^t\|^2 + \frac{2-\eta \sigma_f}{4\eta}\|u - \bar{w}^t\|^2 \right. \\
    & + \frac{\eta}{1 - \eta L}\| \bar{s}^t- \nabla f(\bar{w}^t)\|^2 -  \frac{(1+\sigma_{\psi} \eta)\|u - \bar{w}^{t+1}\|^2}{2\eta} \\
    &  \left.  +  \frac{1-\eta L + 2\eta}{2(1-\eta L)m}\|\mathbf{1}\bar{s}^t - \mathbf{s}^t\|^2 + \frac{1+ \eta}{2m\eta}\|\vw^t - \mathbf{1}\bar{w}^t \|^2 \right] \\
     = \mathbb{E} & \left[  \frac{\eta}{1 - \eta L}\| \bar{s}^t- \nabla f(\bar{w}^t)\|^2 + \frac{2-\eta \sigma_f}{4\eta}\|u - \bar{w}^t\|^2  -  \frac{(1+\sigma_{\psi} \eta)\|u - \bar{w}^{t+1}\|^2}{2\eta} \right.\\
    &  \left. +  \frac{1-\eta L + 2\eta}{2(1-\eta L)m}\|\mathbf{1}\bar{s}^t - \mathbf{s}^t\|^2 + \left(\frac{\ell_1\ell_2}{\sigma_f m}+ \frac{1+\eta}{2m\eta}\right)\|\vw^t - \mathbf{1}\bar{w}^t \|^2 \right].
\end{align*}
\end{proof}
One can show that the expected value of the quantity $\norm{\bar{s}^t - \nabla f(\bar{w}^t)}^2$ on the right-hand side of inequality (\ref{cur_iterate_inner_bound}) can be upper bounded with the $(\ell_1, \ell_2)$-smoothness of functions $f_{i,j}(\cdot)$:
\begin{lemma} If $\bar{s}^t = \frac{1}{m}\sum_{i=1}^m \mathbf{s}_i^t$ with $\mathbf{s}_i^{t} = FastMix(\mathbf{s}^{t-1} + \mathbf{v}^{t} - \mathbf{v}^{t-1}, M)_i$, we have
\begin{equation}
    \mathbb{E}\left[\|\bar{s}^t - \nabla f(\bar{w}^t)\|^2\right] \leq \left( \frac{6 \ell_1\ell_2}{m b} + \frac{2 \ell_1\ell_2}{m} \right) \|  \vw^t - \mathbf{1}\bar{w}^t\|^2 + \frac{6\ell_1\ell_2}{b} \|  \bar{w}^t - \bar{w}^0\|^2 +\frac{6\ell_1\ell_2}{m b}\| \vw^0 - \mathbf{1}\bar{w}^0  \|^2. 
\end{equation}
\label{variance_lemma}
\end{lemma}
\begin{proof}
Recall that $\bar{s}^t = \bar{v}^t$ by Lemma \ref{decentral_grad}. Then we have
\begin{align*}
    & \mathbb{E}\left[\|\bar{s}^t - \nabla f(\bar{w}^t)\|^2\right] \\
    = & \mathbb{E} \left[\| \bar{v}^t - \nabla f(\bar{w}^t) \|^2 \right] \\
    = & \mathbb{E}\left[\norm{ \bar{\mu} + \frac{1}{m b} \sum_{i=1}^m \sum_{j_i \in B_i^t} \left(\nabla f_{i, j_i}^t(w_i^t) - \nabla f_{i, j_i}^t(w_i^0)\right) - \nabla f(\bar{w}^t)}^2 \right] \\
    = & \mathbb{E}\left[\norm{ \frac{1}{m}\sum_{i=1}^m \nabla f_i(w_i^0) + \frac{1}{m b} \sum_{i=1}^m \sum_{j_i \in B_i^t} \left(\nabla f_{i, j_i}^t(w_i^t) - \nabla f_{i, j_i}^t(w_i^0)\right) - \nabla f(\bar{w}^t)}^2\right] \\
     = & \mathbb{E}\left[ \left\| \frac{1}{m}\sum_{i=1}^m \nabla f_i(w_i^0) + \frac{1}{m b} \sum_{i=1}^m \sum_{j_i \in B_i^t} \left(\nabla f_{i, j_i}^t(w_i^t) - \nabla f_{i, j_i}^t(w_i^0)\right)  -  \frac{1}{m}\sum_{i=1}^m \nabla f_i(w_i^t) + \frac{1}{m}\sum_{i=1}^m \nabla f_i(w_i^t)  - \nabla f(\bar{w}^t)\right\|^2\right] \\
    \leq & \mathbb{E}\left[2 \left\| \frac{1}{m}\sum_{i=1}^m \nabla f_i(w_i^0) + \frac{1}{m b} \sum_{i=1}^m \sum_{j_i \in B_i^t} (\nabla f_{i, j_i}^t(w_i^t) - \nabla f_{i, j_i}^t(w_i^0))  -  \frac{1}{m}\sum_{i=1}^m \nabla f_i(w_i^t)\right\|^2 + 2\left\|\frac{1}{m}\sum_{i=1}^m \nabla f_i(w_i^t)  - \nabla f(\bar{w}^t)\right\|^2\right].   
\end{align*}
We can bound the first variance term in the above inequality using the classical variance reduction argument, i.e.,
\begin{align*}
      & \mathbb{E}\left[\|\bar{s}^t - \nabla f(\bar{w}^t)\|^2\right] \\
     \leq & \mathbb{E}\left[2\norm{ \frac{1}{m b} \sum_{i=1}^m \sum_{j_i \in B_i^t} \left(\nabla f_{i, j_i}^t(w_i^t) - \nabla f_{i, j_i}^t(w_i^0)\right) - \frac{1}{m}\sum_{i=1}^m \left( \nabla f_i(w_i^t) -\nabla f_i(w_i^0)\right) }^2 
       + \frac{2}{m} \sum_{i=1}^m\norm{ \nabla f_i(w_i^t)  - \nabla f_i(\bar{w}^t)}^2\right] \\
     \leq & \mathbb{E}\left[\frac{2}{m } \sum_{i=1}^m \norm{\frac{1}{b} \sum_{j_i \in B_i^t} (\nabla f_{i, j_i}^t(w_i^t) - \nabla f_{i, j_i}^t(w_i^0)) - \left( \nabla f_i(w_i^t) -\nabla f_i(w_i^0)\right)}^2 
       + \frac{2\ell_1\ell_2}{m} \sum_{i=1}^m\| w_i^t  - \bar{w}^t\|^2\right] \\
       = & \mathbb{E}\left[\frac{2}{m b} \sum_{i=1}^m \norm{  \left(\nabla f_{i, j_i}^t(w_i^t) - \nabla f_{i, j_i}^t(w_i^0)\right) - \left( \nabla f_i(w_i^t) -\nabla f_i(w_i^0)\right)}^2 
       + \frac{2\ell_1\ell_2}{m} \sum_{i=1}^m\| w_i^t  - \bar{w}^t\|^2\right] \\
       \leq  & \mathbb{E}\left[\frac{2}{m b} \sum_{i=1}^m \|  (\nabla f_{i, j_i}^t(w_i^t) - \nabla f_{i, j_i}^t(w_i^0)) \|^2 
       + \frac{2\ell_1\ell_2}{m} \sum_{i=1}^m\| w_i^t  - \bar{w}^t\|^2\right] \\
      \leq & \frac{2\ell_1\ell_2}{m b} \sum_{i=1}^m \|  w_i^t - w_i^0\|^2 
       + \frac{2\ell_1\ell_2}{m} \| \vw^t  - \mathbf{1}\bar{w}^t\|^2 \\
        \leq & \frac{2\ell_1\ell_2}{m b} \sum_{i=1}^m \left[3\|  w_i^t - \bar{w}^t\|^2 + 3\|  \bar{w}^t - \bar{w}^0\|^2 +3\| \bar{w}^0 - w_i^0\|^2 \right] 
       + \frac{2\ell_1\ell_2}{m} \| \vw^t  - \mathbf{1}\bar{w}^t\|^2 \\
        = & \left( \frac{6 \ell_1\ell_2}{m b} + \frac{2 \ell_1\ell_2}{m} \right) \|  \vw^t - \mathbf{1}\bar{w}^t\|^2 + \frac{6\ell_1\ell_2}{b} \|  \bar{w}^t - \bar{w}^0\|^2 +\frac{6\ell_1\ell_2}{m b}\| \vw^0 - \mathbf{1}\bar{w}^0  \|^2.
\end{align*}
\end{proof}
Finally, we can show the main theoretical results in the inner loop by combining the results of Lemma \ref{inner_lemma_1}, Lemma \ref{lemma:inner_lemma_2}, and Lemma \ref{variance_lemma}.
\begin{lemma}
If the gradient estimator $\mathbf{s}_i^t$ satisfies  Lemma \ref{variance_lemma} for every $t$, then as long as $\eta \leq \min(\frac{1}{2L}, \frac{1}{8}\sqrt{\frac{b}{\ell_1 \ell_2 t_0}})$ and $\rho \leq  \min\left(\sqrt{\frac{1}{8 b m \eta J  t_0}}, \sqrt{\frac{1}{18(1 + {b})}}\right)$ where $J \coloneqq \frac{12 \eta \ell_1\ell_2}{m b} + \frac{4 \eta \ell_1\ell_2}{m} + \frac{\ell_1\ell_2}{\sigma_f m}+ \frac{1+\eta}{2m\eta} + \frac{1-\eta L + 2\eta}{2(1-\eta L)m\eta^2}$, let $T$ be the last iterate of the inner loop in Algorithm \ref{strongconvex_algo_2}, then it satisfies 
\begin{align*}
    \mathbb{E}\left[F(\bar{w}^{T+1}) - F(u)\right] \leq  \mathbb{E} & \left[-\frac{\|\bar{w}^{T+1} - \bar{w}^0\|^2}{4\eta t_0} + \frac{\langle \bar{w}^0 - u, \bar{w}^0 - \bar{w}^{T+1}\rangle}{t_0\eta} - \frac{\sigma_f + \sigma_{\psi}}{8}\| u - \bar{w}^{T+1}\|^2 \right.\\
    &  \left.  + \left( J + \frac{1}{4 \eta m  }\right)\| \vw^0 - \mathbf{1}\bar{w}^0  \|^2 +  J \eta^2 \| \mathbf{s}^0 - \mathbf{1}\bar{s}^0  \|^2 \right].
\end{align*}
\label{lemma_inner}
\end{lemma}
\begin{proof}
We bound the spectral norm of $\mathbf{A}$ by
\begin{align*}
    & \quad \|\mathbf{A}\|_2 \leq \|\mathbf{A}\|_F \\
          &  = 2\rho^2 \sqrt{ 4^2 +  4^2 +
    [{8\ell_1\ell_2\eta^2}(1+8\rho^2)]^2 + (1+{64 \ell_1 \ell_2\rho^2\eta^2})^2} \\
    & \leq  2\rho^2 (4 + 4 + {8\ell_1\ell_2\eta^2}(1+8\rho^2) + (1+{64 \ell_1 \ell_2\rho^2\eta^2}) ) \\
    & \leq  2\rho^2 (9 + 40 \ell_1\ell_2\eta^2 + 32 \ell_1 \ell_2 \eta^2 ) = 18 \rho^2 (1 + 8 \ell_1 \ell_2 \eta^2) \\
    & \leq 18 \rho^2 \left(1 +  \frac{b}{8t_0}\right) \\
    & \leq 1.
\end{align*}
For any $t$, it holds that
\begin{align*}
    \mathbb{E} &\left[F(\bar{w}^{t+1}) - F(u)\right]  \\
    \leq \mathbb{E}  & \left[  2\eta\| \bar{s}^t- \nabla f(\bar{w}^t)\|^2 + \frac{2-\eta \sigma_f}{4\eta}\|u - \bar{w}^t\|^2  -  \frac{(1+\sigma_{\psi} \eta)\|u - \bar{w}^{t+1}\|^2}{2\eta} \right. \\
     &\left.  +  \frac{1-\eta L + 2\eta}{2(1-\eta L)m}\|\mathbf{1}\bar{s}^t - \mathbf{s}^t\|^2 + \left(\frac{\ell_1\ell_2}{\sigma_f m}+ \frac{1+\eta}{2m\eta}\right)\|\vw^t - \mathbf{1}\bar{w}^t \|^2 \right] \\
     \leq \mathbb{E} &\left[  2\eta \left(\left( \frac{6 \ell_1\ell_2}{m b} + \frac{2 \ell_1\ell_2}{m} \right) \|  \vw^t - \mathbf{1}\bar{w}^t\|^2 + \frac{6\ell_1\ell_2}{b} \|  \bar{w}^t - \bar{w}^0\|^2 +\frac{6\ell_1\ell_2}{m b}\| \vw^0 - \mathbf{1}\bar{w}^0  \|^2\right) \right. \\
    &  + \frac{2-\eta \sigma_f}{4\eta}\|u - \bar{w}^t\|^2  -  \frac{(1+\sigma_{\psi} \eta)\|u - \bar{w}^{t+1}\|^2}{2\eta}\\
    &  \left. +  \frac{1-\eta L + 2\eta}{2(1-\eta L)m}\|\mathbf{1}\bar{s}^t - \mathbf{s}^t\|^2 + \left(\frac{\ell_1\ell_2}{\sigma_f m}+ \frac{1+\eta}{2m\eta}\right)\|\vw^t - \mathbf{1}\bar{w}^t \|^2\right] \\
     \leq \mathbb{E} &\left[\frac{12\eta \ell_1 \ell_2}{b} \|\bar{w}^t - \bar{w}^0\|^2 + \frac{2-\eta \sigma_f}{4\eta}\|u - \bar{w}^t\|^2  -  \frac{(1+\sigma_{\psi} \eta)\|u - \bar{w}^{t+1}\|^2}{2\eta} \right. \\
    &  + \frac{12 \eta \ell_1\ell_2}{m b}\| \vw^0 - \mathbf{1}\bar{w}^0  \|^2 +  \frac{1-\eta L + 2\eta}{2(1-\eta L)m}\|\mathbf{1}\bar{s}^t - \mathbf{s}^t\|^2 \\
    &  \left. + \left(\frac{12 \eta \ell_1\ell_2}{m b} + \frac{4 \eta \ell_1\ell_2}{m} + \frac{\ell_1\ell_2}{\sigma_f m}+ \frac{1+\eta}{2m\eta} \right)\|\vw^t - \mathbf{1}\bar{w}^t \|^2\right].
\end{align*}
Taking $t=T$ and we choose $p = \frac{1}{t_0}$, so $T\sim \texttt{Geom}(\frac{1}{t_0})$ follows from the geometric distribution, one has
\begin{align*}
\small
    \mathbb{E} &\left[F(\bar{w}^{T+1}) - F(u)\right] \\
  \leq \mathbb{E} &\left[\frac{12\eta \ell_1 \ell_2 (t_0 - 1)}{ b t_0} \|\bar{w}^{T+1} - \bar{w}^0\|^2 + \frac{(2-\eta \sigma_f) (t_0 - 1)}{4\eta t_0}\|u - \bar{w}^{T+1}\|^2 + \frac{2-\eta \sigma_f}{4 \eta t_0}\|u - \bar{w}^{0}\|^2 \right.\\
    &  -  \frac{(1+\sigma_{\psi} \eta)\|u - \bar{w}^{T+1}\|^2}{2\eta}
     + \frac{12 \eta \ell_1\ell_2}{m b}\| \vw^0 - \mathbf{1}\bar{w}^0  \|^2 +  \frac{1-\eta L + 2\eta}{2(1-\eta L)m}\|\mathbf{1}\bar{s}^T - \mathbf{s}^T\|^2 \\
    &  \left. + \left(\frac{12 \eta \ell_1\ell_2}{m b} + \frac{4 \eta \ell_1\ell_2}{m} + \frac{\ell_1\ell_2}{\sigma_f m}+ \frac{1+\eta}{2m\eta}\right)\|\vw^T - \mathbf{1}\bar{w}^T \|^2 \right] \\
    \leq \mathbb{E} &\left[\frac{12\eta \ell_1 \ell_2 (t_0 - 1)}{b t_0} \|\bar{w}^{T+1} - \bar{w}^0\|^2 - \frac{1}{2 \eta t_0} \|u - \bar{w}^{T+1}\|^2 - \frac{\sigma_f}{8}\|u - \bar{w}^{T+1}\|^2 + \frac{1}{ 2 \eta t_0}\|u-\bar{w}^0\|^2 \right. \\
    &  - \frac{\sigma_{\psi}}{2}\|u - \bar{w}^{T+1}\|^2 + \frac{12 \eta \ell_1\ell_2}{m b}\| \vw^0 - \mathbf{1}\bar{w}^0  \|^2  \\
     &  + \left(\frac{12 \eta \ell_1\ell_2}{m b} + \frac{4 \eta \ell_1\ell_2}{m} + \frac{\ell_1\ell_2}{\sigma_f m}+ \frac{1+\eta}{2m\eta} + \frac{1-\eta L + 2\eta}{2(1-\eta L)m\eta^2}\right) \times \\
    &  \left. \left(\frac{p}{1-(1-p)\|A\|_2}\|\mathbf{z}^0\| +  \frac{1-p}{1-(1-p)\|A\|_2}{16\rho^2\eta^2\ell_1\ell_2}\|\vw^0 - \mathbf{1}\bar{w}^0\|^2 + \frac{(2-p)}{1-(1-p)\|\mathbf{A}\|_2}{16\ell_1\ell_2m\rho^2\eta^2}\|\bar{w}^{T}- \bar{w}^0\|^2\right) \right],
\end{align*}
where the last inequality is due to Eq. (\ref{inner_comm}). Recall the definition of $J$, we have
\begin{align*}
    \mathbb{E}&\left[F(\bar{w}^{T+1}) - F(u)\right] \\
    \leq \mathbb{E} &\left[\left(\frac{12\eta \ell_1 \ell_2 (t_0 - 1)}{b t_0} + \frac{16\ell_1\ell_2m\rho^2\eta^2J(2-p)(1-p)}{1-(1-p)\|\mathbf{A}\|_2} \right)\|\bar{w}^{T+1} - \bar{w}^0\|^2 + \frac{\|u - \bar{w}^0\|^2 - \|u - \bar{w}^{T+1}\|^2}{2 \eta t_0} \right. \\
    &  - \frac{\sigma_f + \sigma_{\psi}}{8}\| u - \bar{w}^{T+1}\|^2  + \frac{pJ}{1-(1-p)\|A\|_2}(\|\vw^0 - \mathbf{1}\bar{w}^0\|^2 + \eta^2 \|\mathbf{s}^0 - \mathbf{1}\bar{s}^0\|^2) \\
    &  \left. + \left(\frac{12 \eta \ell_1\ell_2}{m b}  + \frac{16\rho^2\eta^2\ell_1\ell_2J(1-p)}{ 1-(1-p)\|A\|_2}\right) \| \vw^0 - \mathbf{1}\bar{w}^0  \|^2 \right]\\
    \leq \mathbb{E} &\left[\left(\frac{12\eta \ell_1 \ell_2 }{b } + \frac{b m\rho^2J(2-p)}{4 t_0 (1-(1-p)\|\mathbf{A}\|_2) } \right)\|\bar{w}^{T+1} - \bar{w}^0\|^2 + \frac{\|u - \bar{w}^0\|^2 - \|u - \bar{w}^{T+1}\|^2}{2 \eta t_0} \right. \\
    &  - \frac{\sigma_f + \sigma_{\psi}}{8}\| u - \bar{w}^{T+1}\|^2  + \frac{p J}{1-(1-p)\|A\|_2}(\|\vw^0 - \mathbf{1}\bar{w}^0\|^2 + \eta^2 \|\mathbf{s}^0 - \mathbf{1}\bar{s}^0\|^2) \\
    &  \left. + \left(\frac{12 \eta \ell_1\ell_2}{m b}  + \frac{16\rho^2\eta^2\ell_1\ell_2J(1-p)}{ 1-(1-p)\|A\|_2}\right) \| \vw^0 - \mathbf{1}\bar{w}^0  \|^2\right]\\
    \leq \mathbb{E} &\left[\left(\frac{3 }{16 \eta t_0 } + \frac{1}{16 \eta t_0}\right)\|\bar{w}^{T+1} - \bar{w}^0\|^2 + \frac{\|u - \bar{w}^0\|^2 - \|u - \bar{w}^{T+1}\|^2}{2 \eta t_0} \right. \\
    &  - \frac{\sigma_f + \sigma_{\psi}}{8}\| u - \bar{w}^{T+1}\|^2  + J(\|\vw^0 - \mathbf{1}\bar{w}^0\|^2 + \eta^2 \|\mathbf{s}^0 - \mathbf{1}\bar{s}^0\|^2) \\
    &  \left. + \left(\frac{3}{16 m \eta t_0}  + {\frac{1-p}{16 \eta m t_0  }}\right) \| \vw^0 - \mathbf{1}\bar{w}^0  \|^2\right]\\
     \leq \mathbb{E}&\left[-\frac{\|\bar{w}^{T+1} - \bar{w}^0\|^2}{4\eta t_0} + \frac{\|u-\bar{w}^0\|^2 - \|u - \bar{w}^{T+1}\|^2 + \|\bar{w}^{T+1} - \bar{w}^0\|^2}{2t_0\eta} \right. \\
    &  \left. - \frac{\sigma_f + \sigma_{\psi}}{8}\| u - \bar{w}^{T+1}\|^2  + \left( J + \frac{1}{4 \eta m t_0 }\right)\| \vw^0 - \mathbf{1}\bar{w}^0  \|^2 +  J \eta^2 \| \mathbf{s}^0 - \mathbf{1}\bar{s}^0  \|^2 \right] \\
     = \mathbb{E}&\left[-\frac{\|\bar{w}^{T+1} - \bar{w}^0\|^2}{4\eta t_0} + \frac{\langle \bar{w}^0 - u, \bar{w}^0 - \bar{w}^{T+1}\rangle}{t_0\eta} \right. \\
    &  \left. - \frac{\sigma_f + \sigma_{\psi}}{8}\| u - \bar{w}^{T+1}\|^2  + \left( J + \frac{1}{4 \eta m  }\right)\| \vw^0 - \mathbf{1}\bar{w}^0  \|^2 +  J \eta^2 \| \mathbf{s}^0 - \mathbf{1}\bar{s}^0  \|^2 \right]. 
\end{align*}
\end{proof}

\section{Theoretical Analysis of PMGT-KatyushaX}\label{appendix:pmgt_katyusha_analysis}
In this section, we present the theoretical analysis of the PMGT-KatyushaX algorithm.

\subsection{Theoretical Results of Outer Loop in PMGT-KatyushaX}
To better understand the acceleration trick of our algorithm, we introduce random variable $q_i^k$ at iteration $k$ on agent $i$. In addition, we introduce a new hyperparameter $\alpha_{k+1}$ following the centralized KatyushaX counterpart \cite{allen2018katyusha}. We will show later that $\alpha_{k+1} = \alpha = \frac{t_0 \eta}{2 \tau}$, so we can get the PMGT-KatyushaX algorithm in the main text.   We aim to choose $\bar{q}^{k+1}$ with the following rule:
\begin{equation*}
    \bar{q}^{k+1} = \argmin_{q\in \mathcal{R}^d}\{ \frac{1}{2\alpha_{k+1}}\|q - \bar{q}^k\|^2 + \Bigl \langle \frac{x^{k+1} - y^{k+1}}{t_0 \eta}, q \Bigr \rangle + \frac{\sigma}{8}\|q - \bar{y}^{k+1}\|^2 \}.
\end{equation*}
So it equivalent to choose $\mathbf{q}_i^k$ with the following closed form:
\begin{equation}
    q_i^{k+1}= \FastMix\left(\frac{1}{\frac{1}{\alpha_{k+1}} + \frac{\sigma}{4}}\left(\frac{q_i^k}{\alpha_{k+1}} + \left(\frac{1}{t_0\eta} + \frac{\sigma}{4}\right)y_i^{k+1} - \frac{1}{t_0\eta}x_i^{k+1}\right), M\right)_i.
    \label{q_recur}
\end{equation}
We can bound the difference of iterates at successive iterations with the following lemma:
\begin{lemma} Let $\mathbf{x}^k$ be the $k$-th iterate of local weight variables in the outer loop of PMGT-KatyushaX, then we can bound the difference of $\mathbf{x}^{k+1}$ and $\mathbf{x}^{k}$ with:
\begin{equation}
\begin{split}
      & \|\mathbf{x}^{k + 1} - \mathbf{x}^{k}\|^2 \\
    \leq & 3\|\mathbf{x}^{k + 1} - \mathbf{1}\bar{x}^{k+1}\|^2 + 3 \| \mathbf{x}^{k} - \mathbf{1}\bar{x}^{k}\|^2 + 12m\tau_k^2 \|\bar{q}^k - u\|^2 + 
    \frac{24m(1-\tau_{k})^2}{\sigma}(F(\bar{y}^k) - F(u)) \\
    & + 12m\tau_{k-1}^2 \|\bar{q}^{k-1} - x^*\|^2      +  \frac{24m(1-\tau_{k-1})^2}{\sigma}\left(F(\bar{y}^{k-1}) - F(x^*)\right).
\end{split}
\end{equation}
\label{x_recur_lemma}
\end{lemma}
\begin{proof}
Recall $x^*$ is the minimizer of $F(\cdot)$, it follows that
\begin{align*}
   & \|\mathbf{x}^{k + 1} - \mathbf{x}^{k}\|^2 \\
   \leq & \|\mathbf{x}^{k + 1} - \mathbf{1}\bar{x}^{k+1} - \mathbf{x}^{k} + \mathbf{1}\bar{x}^{k} + \mathbf{1}\bar{x}^{k+1} - \mathbf{1}\bar{x}^{k} \|^2 \\
   \leq & 3\|\mathbf{x}^{k + 1} - \mathbf{1}\bar{x}^{k+1}\|^2 + 3 \| \mathbf{x}^{k} - \mathbf{1}\bar{x}^{k}\|^2 + 3m\|\bar{x}^{k+1} - \bar{x}^{k}\|^2 \\
   \leq & 3\|\mathbf{x}^{k + 1} - \mathbf{1}\bar{x}^{k+1}\|^2 + 3 \| \mathbf{x}^{k} - \mathbf{1}\bar{x}^{k}\|^2 + 6m\|\bar{x}^{k+1} - x^*\|^2 + 6m\|\bar{x}^{k} - x^*\|^2  .
\end{align*}
Since $0 \in \partial F(x^*)$, so $\frac{\sigma}{2}\|\bar{y}^k - x^*\|^2 \leq F(\bar{y}^k) - F(x^*) $, and we have
\begin{align*}
    &\|\bar{x}^{k+1} - x^*\|^2  \leq 2\tau_k^2 \|\bar{q}^k - x^*\|^2 + 2(1 - \tau_{k})^2\|\bar{y}^k - x^*\|^2 \\
    \leq & 2\tau_k^2 \|\bar{q}^k - x^*\|^2 + \frac{4(1 - \tau_{k})^2}{\sigma}\left(F(\bar{y}^k) - F(x^*)\right) .
\end{align*}
Combining the above two equations, we have 
\begin{align*}
    & \|\mathbf{x}^{k + 1} - \mathbf{x}^{k}\|^2 \\
    \leq & 3\|\mathbf{x}^{k + 1} - \mathbf{1}\bar{x}^{k+1}\|^2 + 3 \| \mathbf{x}^{k} - \mathbf{1}\bar{x}^{k}\|^2 + 6m\|\bar{x}^{k+1} - x^*\|^2 + 6m\|\bar{x}^{k} - x^*\|^2 \\
    \leq & 3\|\mathbf{x}^{k + 1} - \mathbf{1}\bar{x}^{k+1}\|^2 + 3 \| \mathbf{x}^{k} - \mathbf{1}\bar{x}^{k}\|^2 + 12m\tau_k^2 \|\bar{q}^k - x^*\|^2 + 
    \frac{24m(1-\tau_{k})^2}{\sigma}\left(F(\bar{y}^k) - F(x^*)\right) \\
    & + 12m\tau_{k-1}^2 \|\bar{q}^{k-1} - x^*\|^2      +  \frac{24m(1-\tau_{k-1})^2}{\sigma}\left(F(\bar{y}^{k-1}) - F(x^*)\right).
\end{align*}
\end{proof}

Similar to the analysis of the inner loop, we can upper bound the consensus error of state variables and gradient tracking variables with the following lemma:
\begin{lemma}
Let $\mathbf{x}^{k}$,  $\mathbf{q}^k$, $\mathbf{y}^k$ and $\hat{\mathbf{s}}^k$ be the $k$-th iterate of local weight variables in the outer loop of Algorithm \ref{strongconvex_algo_2}, then we have the following recursive inequalities:
\begin{equation}
    \|\mathbf{x^{k+1}} - \mathbf{1}\bar{x}^{k+1}\|^2 \leq  2 \rho^2 \tau_k^2 \|\mathbf{q}^k - \mathbf{1} \bar{q}^k\|^2 + 2\rho^2 (1 - \tau_{k})^2 \| \mathbf{y}^k - \mathbf{1} \bar{y}^k\|^2.
\end{equation}
The variable $\mathbf{q}^{k+1}$ has the following recursive form:
\begin{equation}
\small\begin{split}
    &\|\mathbf{q}^{k+1} - \mathbf{1} \bar{q}^{k+1}\|^2 \\
     \leq&  \frac{3 \rho^2}{\left(\frac{1}{\alpha_{k+1}} + \frac{\sigma }{4}\right)^2 } \left(\frac{1}{\alpha_{k+1}^2} + \frac{2 \rho^2 \tau_k^2}{t_0^2 \eta^2} + 2 \rho^2 \tau_k^2 \left(1 + 16\rho^2\eta^2\ell_1\ell_2 t_0  \right)\left(\frac{1}{t_0\eta} + \frac{\sigma}{4}\right)^2 \right)\|\mathbf{q}^k - \mathbf{1} \bar{q}^k\|^2 \\
   & + \frac{6 \rho^4 (1 - \tau_{k})^2}{\left(\frac{1}{\alpha_{k+1}} + \frac{\sigma }{4}\right)^2 }\left(\frac{1}{t_0^2 \eta^2} +  \left(1 + 16\rho^2\eta^2\ell_1\ell_2 t_0  \right)\left(\frac{1}{t_0\eta} + \frac{\sigma}{4}\right)^2 \right) \| \mathbf{y}^k - \mathbf{1} \bar{y}^k\|^2 +    \frac{6\rho^4 \eta^2 \left(\frac{1}{t_0\eta} + \frac{\sigma}{4}\right)^2}{\left(\frac{1}{\alpha_{k+1}} + \frac{\sigma}{4}\right)^2}\|\hat{\mathbf{s}}^{k} - \mathbf{1} \bar{\hat{s}}^{k}\|^2 \\
     & +  \frac{18 \rho^4 \eta^2 \ell_1 \ell_2 \left(\frac{1}{t_0\eta} + \frac{\sigma}{4}\right)^2}{\left(\frac{1}{\alpha_{k+1}} + \frac{\sigma}{4}\right)^2} \| \mathbf{x}^{k} - \mathbf{1}\bar{x}^{k}\|^2   +   \frac{72 \rho^4  \eta^2 \ell_1 \ell_2 m \tau_k^2 \left(\frac{1}{t_0\eta} + \frac{\sigma}{4}\right)^2}{\left(\frac{1}{\alpha_{k+1}} + \frac{\sigma}{4}\right)^2} \|\bar{q}^k - x^*\|^2 \\
     & + 
    \frac{144 \rho^4 \eta^2 \ell_1 \ell_2 m (1 - \tau_k)^2 \left(\frac{1}{t_0\eta} + \frac{\sigma}{4}\right)^2}{\sigma \left(\frac{1}{\alpha_{k+1}} + \frac{\sigma}{4}\right)^2}(F(\bar{y}^k) - F(x^*)) +   \frac{72 \rho^4  \eta^2 \ell_1 \ell_2 m \tau_{k-1}^2\left(\frac{1}{t_0\eta} + \frac{\sigma}{4}\right)^2}{\left(\frac{1}{\alpha_{k+1}} + \frac{\sigma}{4}\right)^2} \|\bar{q}^{k-1} - x^*\|^2 \\
     & + 
    \frac{144 \rho^4 \eta^2 \ell_1 \ell_2 m (1 - \tau_{k-1})^2 \left(\frac{1}{t_0\eta} + \frac{\sigma}{4}\right)^2}{\sigma \left(\frac{1}{\alpha_{k+1}} + \frac{\sigma}{4}\right)^2}(F(\bar{y}^{k-1}) - F(x^*)) + \frac{48 \rho^4 \ell_1 \ell_2 m\eta^4 t_0^2 (2 t_0 - 1)\left(\frac{1}{t_0\eta} + \frac{\sigma}{4}\right)^2}{\left(\frac{1}{\alpha_{k+1}} + \frac{\sigma}{4}\right)^2}\|\bar{\mathcal{G}}^{k+1}\|^2.
    \end{split}
\end{equation}
The variable $\mathbf{y}^{k+1}$ has the following recursive form:
\begin{equation}
    \begin{split}
        \|\mathbf{y}^{k+1} - \mathbf{1} \bar{y}^{k+1}\|^2 
           \leq &2 \rho^2 \tau_k^2 \left(1 + 16\rho^2\eta^2\ell_1\ell_2 t_0  \right) \|\mathbf{q}^k - \mathbf{1} \bar{q}^k\|^2  + 2\rho^2 (1 - \tau_{k})^2 \left(1 + 16\rho^2\eta^2\ell_1\ell_2 t_0  \right) \| \mathbf{y}^k - \mathbf{1} \bar{y}^k\|^2 \\
     & +   2 \rho^2 \eta^2 \|\hat{\mathbf{s}}^{k} - \mathbf{1} \bar{\hat{s}}^{k}\|^2 + {6\rho^2 \eta^2 \ell_1 \ell_2} \| \mathbf{x}^{k} - \mathbf{1}\bar{x}^{k}\|^2  \\
     & +  24\rho^2 \eta^2 \ell_1 \ell_2 m \tau_k^2 \|\bar{q}^k - x^*\|^2 + 
    \frac{48\rho^2 \eta^2 \ell_1 \ell_2 m (1 - \tau_k)^2}{\sigma}\left(F(\bar{y}^k) - F(x^*)\right) \\
    & + {24\rho^2 \eta^2 \ell_1 \ell_2 m \tau_{k-1}^2} \|\bar{q}^{k-1} - x^*\|^2      +  \frac{48\rho^2 \eta^2 \ell_1 \ell_2 m (1 - \tau_{k-1})^2}{\sigma}\left(F(\bar{y}^{k-1}) - F(x^*)\right) \\
    & + {16\ell_1\ell_2m\rho^2\eta^4 t_0^2 (2 t_0 - 1)}\|\bar{\mathcal{G}}^{k+1}\|^2.
    \end{split}
\end{equation}
The variable $\hat{\mathbf{s}}^{k+1}$ has the following recursive form:
\begin{equation}\small
\begin{split}
    \|\hat{\mathbf{s}}^{k+1} - \mathbf{1} \bar{\hat{s}}^{k+1}\|^2  \leq & 2\rho^2 \|\hat{\mathbf{s}}^{k} - \mathbf{1} \bar{\hat{s}}^{k}\|^2 +   {12\rho^4 \ell_1 \ell_2 \tau_k^2} \|\mathbf{q}^k - \mathbf{1} \bar{q}^k\|^2 + {12\rho^4 \ell_1 \ell_2 (1-\tau_k)^2} \| \mathbf{y}^k - \mathbf{1} \bar{y}^k\|^2 \\
     & + {6\rho^2 \ell_1 \ell_2} \| \mathbf{x}^{k} - \mathbf{1}\bar{x}^{k}\|^2   +  24\rho^2 \ell_1 \ell_2 m \tau_k^2 \|\bar{q}^k - x^*\|^2 + 
    \frac{48\rho^2 \ell_1 \ell_2 m (1 - \tau_k)^2}{\sigma}\left(F(\bar{y}^k) - F(x^*)\right) \\
    & + {24\rho^2 \ell_1 \ell_2 m \tau_{k-1}^2} \|\bar{q}^{k-1} - x^*\|^2      +  \frac{48\rho^2 \ell_1 \ell_2 m (1 - \tau_{k-1})^2}{\sigma}\left(F(\bar{y}^{k-1}) - F(x^*)\right).     
\end{split}
\end{equation}
\label{consensus_error_outer_loop}
\end{lemma}
\begin{proof}
Recall that $\mathbf{x^{k+1}} = \tau_{k} \mathbf{q}^k  + (1 - \tau_{k}) \mathbf{y}^k$, we have
\begin{align*}
    &  \|\mathbf{x^{k+1}} - \mathbf{1}\bar{x}^{k+1}\|^2 \\
    \leq &  \rho^2 \|\tau_{k} \mathbf{q}^k  + (1 - \tau_{k}) \mathbf{y}^k - \tau_{k} \mathbf{1} \bar{q}^k + (1 - \tau_{k}) \mathbf{1} \bar{y}^k\|^2 \\
    \leq & 2 \rho^2 \tau_k^2 \|\mathbf{q}^k - \mathbf{1} \bar{q}^k\|^2 + 2\rho^2 (1 - \tau_{k})^2 \| \mathbf{y}^k - \mathbf{1} \bar{y}^k\|^2.
\end{align*}
Due to Eq. (\ref{q_recur}), it follows that:
\begin{equation}
   \begin{split}
   &  \|\mathbf{q}^{k+1} - \mathbf{1} \bar{q}^{k+1}\|^2 \\
    \leq & \rho^2 \left\|\frac{1}{\frac{1}{\alpha_{k+1}} + \frac{\sigma}{4}}\left(\frac{\mathbf{q}^k}{\alpha_{k+1}} + \left(\frac{1}{t_0\eta} + \frac{\sigma}{4}\right)\mathbf{y}^{k+1} - \frac{1}{t_0\eta}\mathbf{x}^{k+1}\right) \right. \\
     & \left. - \frac{1}{\frac{1}{\alpha_{k+1}} + \frac{\sigma}{4}}\left(\frac{\mathbf{1}\bar{q}^k}{\alpha_{k+1}} + \left(\frac{1}{t_0\eta} + \frac{\sigma}{4}\right)\mathbf{1}\bar{y}^{k+1} - \frac{1}{t_0\eta}\mathbf{1}\bar{x}^{k+1}\right)\right\|^2 \\
    \leq & \frac{3 \rho^2}{\left(\frac{1}{\alpha_{k+1}} + \frac{\sigma}{4}\right)^2}\left(\frac{1}{\alpha_{k+1}^2}\|\mathbf{q}^k - \mathbf{1} \bar{q}^k\|^2 + \frac{1}{t_0^2 \eta^2}\|\mathbf{x}^{k+1} - \mathbf{1} \bar{x}^{k+1}\|^2   + \left(\frac{1}{t_0\eta} + \frac{\sigma}{4}\right)^2\|\mathbf{y}^{k+1} - \mathbf{1} \bar{y}^{k+1}\|^2\right).    
   \end{split}
   \label{q_recur_1}
\end{equation}
Recall that $\hat{\mathbf{s}}^{k+1} = \texttt{FastMix}(\mathbf{\nu}^{k+1}, M)$, it can be verified that
\begin{align*}
    &  \|\hat{\mathbf{s}}^{k+1} - \mathbf{1} \bar{\hat{s}}^{k+1}\|^2 \\
    \leq & \rho^2 \|\mathbf{\nu}^{k+1} - \mathbf{1} \bar{\nu}^{k+1}\|^2 \\
    \leq & 2\rho^2 \|\hat{\mathbf{s}}^{k} - \mathbf{1} \bar{\hat{s}}^{k}\|^2 +   2\rho^2 \sum_{i=1}^m\norm{\nabla f_i (x_i^{k + 1}) - \nabla f_i (x_i^{k})}^2 \\
     \leq & 2\rho^2 \|\hat{\mathbf{s}}^{k} - \mathbf{1} \bar{\hat{s}}^{k}\|^2 +   {2\rho^2 \ell_1 \ell_2} \|\mathbf{x}^{k + 1} - \mathbf{x}^{k}\|^2,
\end{align*}
where the second inequality follows from $\norm{\mathbf{x} - \frac{1}{m} \mathbf{1} \mathbf{1}^T \mathbf{x}} \leq \|\mathbf{x}\|$ for $\forall \mathbf{x} \in \mathcal{R}^{m\times d}$. Using Lemma \ref{x_recur_lemma}, we can have the following bound:
\begin{align*}
      & \|\hat{\mathbf{s}}^{k+1} - \mathbf{1} \bar{\hat{s}}^{k+1}\|^2 \\ 
      \leq & 2\rho^2 \|\hat{\mathbf{s}}^{k} - \mathbf{1} \bar{\hat{s}}^{k}\|^2 +   {6\rho^2 \ell_1 \ell_2} \|\mathbf{x}^{k + 1} - \mathbf{1}\bar{x}^{k+1}\|^2 + {6\rho^2 \ell_1 \ell_2} \| \mathbf{x}^{k} - \mathbf{1}\bar{x}^{k}\|^2  \\
     & +  24\rho^2 \ell_1 \ell_2 m \tau_k^2 \|\bar{q}^k - x^*\|^2 + 
    \frac{48\rho^2 \ell_1 \ell_2 m (1 - \tau_k)^2}{\sigma}\left(F(\bar{y}^k) - F(x^*)\right) \\
    & + {24\rho^2 \ell_1 \ell_2 m \tau_{k-1}^2} \|\bar{q}^{k-1} - x^*\|^2      +  \frac{48\rho^2 \ell_1 \ell_2 m (1 - \tau_{k-1})^2}{\sigma}\left(F(\bar{y}^{k-1}) - F(x^*)\right) \\
      \leq & 2\rho^2 \|\hat{\mathbf{s}}^{k} - \mathbf{1} \bar{\hat{s}}^{k}\|^2 +   {6\rho^2 \ell_1 \ell_2} \left(2 \rho^2 \tau_k^2 \|\mathbf{q}^k - \mathbf{1} \bar{q}^k\|^2 + 2\rho^2 (1 - \tau_{k})^2 \| \mathbf{y}^k - \mathbf{1} \bar{y}^k\|^2\right) \\
     & + {6\rho^2 \ell_1 \ell_2} \| \mathbf{x}^{k} - \mathbf{1}\bar{x}^{k}\|^2   +  24\rho^2 \ell_1 \ell_2 m \tau_k^2 \|\bar{q}^k - x^*\|^2 + 
    \frac{48\rho^2 \ell_1 \ell_2 m (1 - \tau_k)^2}{\sigma}\left(F(\bar{y}^k) - F(x^*)\right) \\
    & + {24\rho^2 \ell_1 \ell_2 m \tau_{k-1}^2} \|\bar{q}^{k-1} - x^*\|^2      +  \frac{48\rho^2 \ell_1 \ell_2 m (1 - \tau_{k-1})^2}{\sigma}\left(F(\bar{y}^{k-1}) - F(x^*)\right) \\
     = & 2\rho^2 \|\hat{\mathbf{s}}^{k} - \mathbf{1} \bar{\hat{s}}^{k}\|^2 +   {12\rho^4 \ell_1 \ell_2 \tau_k^2} \|\mathbf{q}^k - \mathbf{1} \bar{q}^k\|^2 + {12\rho^4 \ell_1 \ell_2 (1-\tau_k)^2} \| \mathbf{y}^k - \mathbf{1} \bar{y}^k\|^2 \\
     & + {6\rho^2 \ell_1 \ell_2} \| \mathbf{x}^{k} - \mathbf{1}\bar{x}^{k}\|^2   +  24\rho^2 \ell_1 \ell_2 m \tau_k^2 \|\bar{q}^k - x^*\|^2 + 
    \frac{48\rho^2 \ell_1 \ell_2 m (1 - \tau_k)^2}{\sigma}\left(F(\bar{y}^k) - F(x^*)\right) \\
    & + {24\rho^2 \ell_1 \ell_2 m \tau_{k-1}^2} \|\bar{q}^{k-1} - x^*\|^2      +  \frac{48\rho^2 \ell_1 \ell_2 m (1 - \tau_{k-1})^2}{\sigma}\left(F(\bar{y}^{k-1}) - F(x^*)\right) .
\end{align*}
Let  $\bar{\mathcal{G}}^{k + 1} = \frac{\bar{w}^0 - \bar{w}^{T+1} }{\eta t_0}$.
Noting that $\mathbf{y}^{k+1} = \mathbf{w}^{T+1}$ in Algorithm \ref{strongconvex_algo_2}, one has $\bar{\mathcal{G}}^{k + 1} = \frac{\bar{y}^{k+1} - \bar{x}^{k+1}}{\eta t_0}$.
By Eq. (\ref{inner_comm}), we have the following bound
\begin{align*}
    & \mathbb{E}_T\left[\|\mathbf{y}^{k+1} - \mathbf{1}\bar{y}^{k+1}\|^2\right] \\
     \leq& \mathbb{E}_T\left[\frac{p}{1-(1-p)\|A\|_2}(\|\mathbf{x}^{k+1} - \mathbf{1} \bar{x}^{k+1}\|^2 + \eta^2 \|\hat{\mathbf{s}}^{k+1} - \mathbf{1}\bar{\hat{s}}^{k+1}\|^2) \right.\\
    &  \left. +  \frac{1-p}{1-(1-p)\|A\|_2}{16\rho^2\eta^2\ell_1\ell_2}\|\mathbf{x}^{k+1} - \mathbf{1}\bar{x}^{k+1}\|^2 + \frac{(2-p)}{1-(1-p)\|\mathbf{A}\|_2}{16\ell_1\ell_2m\rho^2\eta^2}\|\bar{y}^{k+1}- \bar{x}^{k+1}\|^2\right] \\
      \leq &\left(1 + {16\rho^2\eta^2\ell_1\ell_2(t_0 -1)}\right)\|\mathbf{x}^{k+1} - \mathbf{1}\bar{x}^{k+1}\|^2 +  \eta^2 \|\hat{\mathbf{s}}^{k+1} - \mathbf{1}\bar{\hat{s}}^{k+1}\|^2 
      + {16\ell_1\ell_2m\rho^2\eta^4 t_0^2 (2 t_0 - 1)}\|\bar{\mathcal{G}}^{k+1}\|^2 \\
      \leq & \left(1 + {16\rho^2\eta^2\ell_1\ell_2(t_0 -1)}\right) \left(2 \rho^2 \tau_k^2 \|\mathbf{q}^k - \mathbf{1} \bar{q}^k\|^2 + 2\rho^2 (1 - \tau_{k})^2 \| \mathbf{y}^k - \mathbf{1} \bar{y}^k\|^2 \right) \\
     & +  \eta^2 \left( 2\rho^2 \|\hat{\mathbf{s}}^{k} - \mathbf{1} \bar{\hat{s}}^{k}\|^2 +   {12\rho^4 \ell_1 \ell_2 \tau_k^2} \|\mathbf{q}^k - \mathbf{1} \bar{q}^k\|^2 + {12\rho^4 \ell_1 \ell_2 (1-\tau_k)^2} \| \mathbf{y}^k - \mathbf{1} \bar{y}^k\|^2 \right.\\
     & + {6\rho^2 \ell_1 \ell_2} \| \mathbf{x}^{k} - \mathbf{1}\bar{x}^{k}\|^2   +  24\rho^2 \ell_1 \ell_2 m \tau_k^2 \|\bar{q}^k - x^*\|^2 + 
    \frac{48\rho^2 \ell_1 \ell_2 m (1 - \tau_k)^2}{\sigma}(F(\bar{y}^k) - F(x^*)) \\
    & \left. + {24\rho^2 \ell_1 \ell_2 m \tau_{k-1}^2} \|\bar{q}^{k-1} - x^*\|^2      +  \frac{48\rho^2 \ell_1 \ell_2 m (1 - \tau_{k-1})^2}{\sigma}\left(F(\bar{y}^{k-1}) - F(x^*)\right) \right) \\
     & + {16\ell_1\ell_2m\rho^2\eta^4 t_0^2 (2 t_0 - 1)}\|\bar{\mathcal{G}}^{k+1}\|^2 \\
      \leq & 2 \rho^2 \tau_k^2 \left(1 + 16\rho^2\eta^2\ell_1\ell_2 t_0 \right) \|\mathbf{q}^k - \mathbf{1} \bar{q}^k\|^2 \\
     & + 2\rho^2 (1 - \tau_{k})^2 \left(1 + 16\rho^2\eta^2\ell_1\ell_2 t_0  \right) \| \mathbf{y}^k - \mathbf{1} \bar{y}^k\|^2 +   2 \rho^2 \eta^2 \|\hat{\mathbf{s}}^{k} - \mathbf{1} \bar{\hat{s}}^{k}\|^2 \\
     & + {6\rho^2 \eta^2 \ell_1 \ell_2} \| \mathbf{x}^{k} - \mathbf{1}\bar{x}^{k}\|^2   +  24\rho^2 \eta^2 \ell_1 \ell_2 m \tau_k^2 \|\bar{q}^k - x^*\|^2 + 
    \frac{48\rho^2 \eta^2 \ell_1 \ell_2 m (1 - \tau_k)^2}{\sigma}\left(F(\bar{y}^k) - F(x^*)\right) \\
    & + {24\rho^2 \eta^2 \ell_1 \ell_2 m \tau_{k-1}^2} \|\bar{q}^{k-1} - x^*\|^2      +  \frac{48\rho^2 \eta^2 \ell_1 \ell_2 m (1 - \tau_{k-1})^2}{\sigma}\left(F(\bar{y}^{k-1}) - F(x^*)\right) \\
    &+ {16\ell_1\ell_2m\rho^2\eta^4 t_0^2 (2 t_0 - 1)}\|\bar{\mathcal{G}}^{k+1}\|^2 .
\end{align*}
By combining Eq. (\ref{q_recur_1}) and the  above inequality, one has 
\begin{align*}
     &\|\mathbf{q}^{k+1} - \mathbf{1} \bar{q}^{k+1}\|^2 \\
    \leq & \frac{3 \rho^2}{\left(1 + \frac{\sigma \alpha_{k+1}}{4}\right)^2}\|\mathbf{q}^k - \mathbf{1} \bar{q}^k\|^2 + \frac{3 \rho^2}{\left(\frac{1}{\alpha_{k+1}} + \frac{\sigma}{4}\right)^2 t_0^2 \eta^2}\|\mathbf{x}^{k+1} - \mathbf{1} \bar{x}^{k+1}\|^2  + \frac{3 \rho^2 \left(\frac{1}{t_0\eta} + \frac{\sigma}{4}\right)^2}{\left(\frac{1}{\alpha_{k+1}} + \frac{\sigma}{4}\right)^2}\|\mathbf{y}^{k+1} - \mathbf{1} \bar{y}^{k+1}\|^2 \\
    \leq & \frac{3 \rho^2}{\left(\frac{1}{\alpha_{k+1}} + \frac{\sigma}{4}\right)^2 \alpha_{k+1}^2}\|\mathbf{q}^k - \mathbf{1} \bar{q}^k\|^2 + \frac{3 \rho^2}{\left(\frac{1}{\alpha_{k+1}} + \frac{\sigma}{4}\right)^2 t_0^2 \eta^2}\left(2 \rho^2 \tau_k^2 \|\mathbf{q}^k - \mathbf{1} \bar{q}^k\|^2 \right.    \left. + 2\rho^2 (1 - \tau_{k})^2 \| \mathbf{y}^k - \mathbf{1} \bar{y}^k\|^2\right) \\
    & + \frac{3 \rho^2 \left(\frac{1}{t_0\eta} + \frac{\sigma}{4}\right)^2}{\left(\frac{1}{\alpha_{k+1}} + \frac{\sigma}{4}\right)^2} \left[2 \rho^2 \tau_k^2 \left(1 + 16\rho^2\eta^2\ell_1\ell_2 t_0  \right) \|\mathbf{q}^k - \mathbf{1} \bar{q}^k\|^2 \right.\\
     & + 2\rho^2 (1 - \tau_{k})^2 \left(1 + 16\rho^2\eta^2\ell_1\ell_2 t_0  \right) \| \mathbf{y}^k - \mathbf{1} \bar{y}^k\|^2 +   2 \rho^2 \eta^2 \|\hat{\mathbf{s}}^{k} - \mathbf{1} \bar{\hat{s}}^{k}\|^2 \\
     & + {6\rho^2 \eta^2 \ell_1 \ell_2} \| \mathbf{x}^{k} - \mathbf{1}\bar{x}^{k}\|^2   +  24\rho^2 \eta^2 \ell_1 \ell_2 m \tau_k^2 \|\bar{q}^k - x^*\|^2 + 
    \frac{48\rho^2 \eta^2 \ell_1 \ell_2 m (1 - \tau_k)^2}{\sigma}\left(F(\bar{y}^k) - F(x^*)\right) \\
    & \left. + {24\rho^2 \eta^2 \ell_1 \ell_2 m \tau_{k-1}^2} \|\bar{q}^{k-1} - x^*\|^2      +  \frac{48\rho^2 \eta^2 \ell_1 \ell_2 m (1 - \tau_{k-1})^2}{\sigma}\left(F(\bar{y}^{k-1}) - F(x^*)\right) + {16\ell_1\ell_2m\rho^2\eta^4 t_0^2 (2 t_0 - 1)}\|\bar{\mathcal{G}}^{k+1}\|^2\right]\\
   \leq & \frac{3 \rho^2}{\left(\frac{1}{\alpha_{k+1}} + \frac{\sigma }{4}\right)^2 } \left(\frac{1}{\alpha_{k+1}^2} + \frac{2 \rho^2 \tau_k^2}{t_0^2 \eta^2} + 2 \rho^2 \tau_k^2 \left(1 + 16\rho^2\eta^2\ell_1\ell_2 t_0  \right)\left(\frac{1}{t_0\eta} + \frac{\sigma}{4}\right)^2 \right)\|\mathbf{q}^k - \mathbf{1} \bar{q}^k\|^2 \\
   & + \frac{6 \rho^4 (1 - \tau_{k})^2}{\left(\frac{1}{\alpha_{k+1}} + \frac{\sigma }{4}\right)^2 }\left(\frac{1}{t_0^2 \eta^2} +  \left(1 + 16\rho^2\eta^2\ell_1\ell_2 t_0  \right)\left(\frac{1}{t_0\eta} + \frac{\sigma}{4}\right)^2 \right) \| \mathbf{y}^k - \mathbf{1} \bar{y}^k\|^2\\
     &  +    \frac{6\rho^4 \eta^2 \left(\frac{1}{t_0\eta} + \frac{\sigma}{4}\right)^2}{\left(\frac{1}{\alpha_{k+1}} + \frac{\sigma}{4}\right)^2}\|\hat{\mathbf{s}}^{k} - \mathbf{1} \bar{\hat{s}}^{k}\|^2  +  \frac{18 \rho^4 \eta^2 \ell_1 \ell_2 \left(\frac{1}{t_0\eta} + \frac{\sigma}{4}\right)^2}{\left(\frac{1}{\alpha_{k+1}} + \frac{\sigma}{4}\right)^2} \| \mathbf{x}^{k} - \mathbf{1}\bar{x}^{k}\|^2   \\
     & +   \frac{72 \rho^4  \eta^2 \ell_1 \ell_2 m \tau_k^2\left(\frac{1}{t_0\eta} + \frac{\sigma}{4}\right)^2}{\left(\frac{1}{\alpha_{k+1}} + \frac{\sigma}{4}\right)^2} \|\bar{q}^k - x^*\|^2  + 
    \frac{144 \rho^4 \eta^2 \ell_1 \ell_2 m (1 - \tau_k)^2 \left(\frac{1}{t_0\eta} + \frac{\sigma}{4}\right)^2}{\sigma \left(\frac{1}{\alpha_{k+1}} + \frac{\sigma}{4}\right)^2}\left(F(\bar{y}^k) - F(x^*)\right) \\
     &+   \frac{72 \rho^4  \eta^2 \ell_1 \ell_2 m \tau_{k-1}^2\left(\frac{1}{t_0\eta} + \frac{\sigma}{4}\right)^2}{\left(\frac{1}{\alpha_{k+1}} + \frac{\sigma}{4}\right)^2} \|\bar{q}^{k-1} - x^*\|^2 + 
    \frac{144 \rho^4 \eta^2 \ell_1 \ell_2 m (1 - \tau_{k-1})^2 \left(\frac{1}{t_0\eta} + \frac{\sigma}{4}\right)^2}{\sigma \left(\frac{1}{\alpha_{k+1}} + \frac{\sigma}{4}\right)^2}\left(F(\bar{y}^{k-1}) - F(x^*)\right) \\
     &  + \frac{48 \rho^2 \rho^2 \ell_1 \ell_2 m\eta^4 t_0^2 (2 t_0 - 1)\left(\frac{1}{t_0\eta} + \frac{\sigma}{4}\right)^2}{\left(\frac{1}{\alpha_{k+1}} + \frac{\sigma}{4}\right)^2}\|\bar{\mathcal{G}}^{k+1}\|^2.
\end{align*}
\end{proof}
Denote ${z}_{out}^{k} = \begin{pmatrix}
           \|\mathbf{x}^{k} - \mathbf{1}\bar{x}^{k}\|^2\\
           \|\mathbf{y}^{k} - \mathbf{1}\bar{y}^{k}\|^2\\
           \|\mathbf{q}^{k} - \mathbf{1}\bar{q}^{k}\|^2\\
           \eta^2 \|\hat{\mathbf{s}}^{k} - \mathbf{1}\bar{\hat{s}}^{k}\|^2
    \end{pmatrix}$, then we can construct a linear system of inequalities using the results from Lemma \ref{consensus_error_outer_loop}: 
\begin{align*}
    & {z}_{out}^{k+1} = \mathbf{C}_k {z}_{out}^{k} + {d}^{k} ,
    \end{align*}
where    
    \begin{align*}
    &  \mathbf{C}_k = \rho^2
    \begin{pmatrix}
    \mathbf{C}_{k, 1} &
    \mathbf{C}_{k, 2} &
    \mathbf{C}_{k, 3} &
    \mathbf{C}_{k, 4} 
    \end{pmatrix}^{\top} \\
    &  \mathbf{C}_{k, 1} = \begin{pmatrix}
          0 \\
          2 (1 - \tau_{k})^2 \\  
          2  \tau_k^2 \\ 
          0 
    \end{pmatrix}, \\
   & \mathbf{C}_{k, 2} = \begin{pmatrix}
           {6 \eta^2 \ell_1 \ell_2} \\ 
           2 (1 - \tau_{k})^2 \left(1 + 16\rho^2\eta^2\ell_1\ell_2 t_0  \right) \\
           2  \tau_k^2 \left(1 + 16\rho^2\eta^2\ell_1\ell_2 t_0  \right) \\
           2 
    \end{pmatrix}, \\
   & \mathbf{C}_{k, 3} = \begin{pmatrix}
           \frac{18 \rho^2 \eta^2 \ell_1 \ell_2 \left(\frac{1}{t_0\eta} + \frac{\sigma}{4}\right)^2}{\left(\frac{1}{\alpha_{k+1}} + \frac{\sigma}{4}\right)^2} \\
           \frac{6 \rho^2 (1 - \tau_{k})^2}{\left(\frac{1}{\alpha_{k+1}} + \frac{\sigma }{4}\right)^2 }\left(\frac{1}{t_0^2 \eta^2} +  \left(1 + 16\rho^2\eta^2\ell_1\ell_2 t_0 \right)\left(\frac{1}{t_0\eta} + \frac{\sigma}{4}\right)^2 \right) \\
           \frac{3 }{\left(\frac{1}{\alpha_{k+1}} + \frac{\sigma }{4}\right)^2 } \left(\frac{1}{\alpha_{k+1}^2} + \frac{2 \rho^2 \tau_k^2}{t_0^2 \eta^2} + 2 \rho^2 \tau_k^2 \left(1 + 16\rho^2\eta^2\ell_1\ell_2 t_0 \right)\left(\frac{1}{t_0\eta} + \frac{\sigma}{4}\right)^2 \right) \\
           \frac{6\rho^2  \left(\frac{1}{t_0\eta} + \frac{\sigma}{4}\right)^2}{\left(\frac{1}{\alpha_{k+1}} + \frac{\sigma}{4}\right)^2} 
    \end{pmatrix}, \\
   & \mathbf{C}_{k, 4} = \begin{pmatrix}
           6 \eta^2 \ell_1 \ell_2 \\
           {12\rho^2 \eta^2 \ell_1 \ell_2 (1-\tau_k)^2} \\
           {12\rho^2 \eta^2 \ell_1 \ell_2 \tau_k^2} \\
           2
    \end{pmatrix},
    \end{align*}
    and
    \begin{align*}
       {d}^k = &  \begin{pmatrix}
    {d}_1^k & {d}_2^k & {d}_3^k & {d}_4^k
    \end{pmatrix}^{\top},
    \\
     {d}_1^k =& 0, \\
     {d}_2^k =& 24\rho^2 \eta^2 \ell_1 \ell_2 m \tau_k^2 \|\bar{q}^k - x^*\|^2 + 
    \frac{48\rho^2 \eta^2 \ell_1 \ell_2 m (1 - \tau_k)^2}{\sigma}\left(F(\bar{y}^k) - F(x^*)\right) \\
     & + {24\rho^2 \eta^2 \ell_1 \ell_2 m \tau_{k-1}^2} \|\bar{q}^{k-1} - x^*\|^2      +  \frac{48\rho^2 \eta^2 \ell_1 \ell_2 m (1 - \tau_{k-1})^2}{\sigma}\left(F(\bar{y}^{k-1}) - F(x^*)\right) \\
     & + {16\ell_1\ell_2m\rho^2\eta^4 t_0^2 (2 t_0 - 1)}\|\bar{\mathcal{G}}^{k+1}\|^2, \\
      {d}_3^k =& \frac{72 \rho^4  \eta^2 \ell_1 \ell_2 m \tau_k^2\left(\frac{1}{t_0\eta} + \frac{\sigma}{4}\right)^2}{\left(\frac{1}{\alpha_{k+1}} + \frac{\sigma}{4}\right)^2} \|\bar{q}^k - x^*\|^2   + \frac{144 \rho^4 \eta^2 \ell_1 \ell_2 m (1 - \tau_k)^2 \left(\frac{1}{t_0\eta} + \frac{\sigma}{4}\right)^2}{\sigma \left(\frac{1}{\alpha_{k+1}} + \frac{\sigma}{4}\right)^2}\left(F(\bar{y}^k) - F(x^*)\right) \\
     & +   \frac{72 \rho^4  \eta^2 \ell_1 \ell_2 m \tau_{k-1}^2\left(\frac{1}{t_0\eta} + \frac{\sigma}{4}\right)^2}{\left(\frac{1}{\alpha_{k+1}} + \frac{\sigma}{4}\right)^2} \|\bar{q}^{k-1} - x^*\|^2  + 
    \frac{144 \rho^4 \eta^2 \ell_1 \ell_2 m (1 - \tau_{k-1})^2 \left(\frac{1}{t_0\eta} + \frac{\sigma}{4}\right)^2}{\sigma \left(\frac{1}{\alpha_{k+1}} + \frac{\sigma}{4}\right)^2}\left(F(\bar{y}^{k-1}) - F(x^*)\right) \\
    & + \frac{48 \rho^2 \rho^2 \ell_1 \ell_2 m\eta^4 t_0^2 (2 t_0 - 1)\left(\frac{1}{t_0\eta} + \frac{\sigma}{4}\right)^2}{\left(\frac{1}{\alpha_{k+1}} + \frac{\sigma}{4}\right)^2}\|\bar{\mathcal{G}}^{k+1}\|^2, \\
     {d}_4^k =& 24\rho^2 \eta^2 \ell_1 \ell_2 m \tau_k^2 \|\bar{q}^k - x^*\|^2 + 
    \frac{48\rho^4 \ell_1 \ell_2 m (1 - \tau_k)^2}{\sigma}\left(F(\bar{y}^k) - F(x^*)\right) \\
    & + {24\rho^2  \eta^2 \ell_1 \ell_2 m \tau_{k-1}^2} \|\bar{q}^{k-1} - x^*\|^2   +  \frac{48\rho^2 \eta^2 \ell_1 \ell_2 m (1 - \tau_{k-1})^2}{\sigma}\left(F(\bar{y}^{k-1}) - F(x^*)\right). 
\end{align*}

Since the function $f(\cdot)$ is strongly convex, we set $\tau_i = \tau$ and $\alpha_i = \alpha$ as constant for $\forall i \in \mathcal{N}$, so we can set $\mathbf{C}_i = \mathbf{C}$ because the matrix values do not change at each iteration. If we unroll the above equation, we can get the following rule:
\begin{equation}
    {z}_{out}^{k+1} =  \mathbf{C}^{k+1}{z}_{out}^0 + \sum_{i=0}^k \mathbf{C}^{k-i}{d}^i =  \sum_{i=0}^k \mathbf{C}^{k-i}{d}^i.
    \label{z_out_recur}
\end{equation}
The last equality is due to every random variable are initialized with the same value. We can further decompose ${d}^i$ into the following form
\begin{align*}
    {d}^i &=  24\rho^2 \eta^2 \ell_1 \ell_2 m \tau_i^2  \begin{pmatrix}
       0 \\
       1 \\
        \frac{3 \rho^2  \left(\frac{1}{t_0\eta} + \frac{\sigma}{4}\right)^2}{\left(\frac{1}{\alpha_{i+1}} + \frac{\sigma}{4}\right)^2}\\
        1
    \end{pmatrix}\|\bar{q}^i - x^*\|^2 + 
   {24\rho^2 \eta^2 \ell_1 \ell_2 m \tau_{i-1}^2} \begin{pmatrix}
       0 \\
        1 \\
        \frac{3 \rho^2   \left(\frac{1}{t_0\eta} + \frac{\sigma}{4}\right)^2}{\left(\frac{1}{\alpha_{i+1}} + \frac{\sigma}{4}\right)^2}\\
        1
    \end{pmatrix}\|\bar{q}^{i-1} - x^*\|^2 \\
     & +  \frac{48\rho^2 \eta^2 \ell_1 \ell_2 m (1 - \tau_i)^2}{\sigma} \begin{pmatrix}
       0 \\
        1 \\
        \frac{3 \rho^2   \left(\frac{1}{t_0\eta} + \frac{\sigma}{4}\right)^2}{ \left(\frac{1}{\alpha_{i+1}} + \frac{\sigma}{4}\right)^2}\\
        1
    \end{pmatrix}\left(F(\bar{y}^i) - F(x^*)\right) \\ 
    & + \frac{48\rho^2 \eta^2 \ell_1 \ell_2 m (1 - \tau_{i-1})^2}{ \sigma}\begin{pmatrix}
       0 \\
        1 \\
        \frac{3 \rho^2   \left(\frac{1}{t_0\eta} + \frac{\sigma}{4}\right)^2}{ \left(\frac{1}{\alpha_{i+1}} + \frac{\sigma}{4}\right)^2}\\
        1
    \end{pmatrix}\left(F(\bar{y}^{i-1}) - F(x^*)\right)\\
    & + 16\rho^2\eta^4 \ell_1\ell_2m t_0^2 (2 t_0 - 1) \begin{pmatrix}
       0\\
       1 \\
       \frac{3 \rho^2   \left(\frac{1}{t_0\eta} + \frac{\sigma}{4}\right)^2}{\left(\frac{1}{\alpha_{i+1}} + \frac{\sigma}{4}\right)^2} \\
       0
    \end{pmatrix}\|\bar{\mathcal{G}}^{i+1}\|^2.
\end{align*}
If we take the norm on ${z}_{out}^k$ and denote 
$\xi = \norm{
\begin{pmatrix}
      0\\
       1\\
       \frac{3 \rho^2   \left(\frac{1}{t_0\eta} + \frac{\sigma}{4}\right)^2}{\left(\frac{1}{\alpha} + \frac{\sigma}{4}\right)^2} \\
       1
    \end{pmatrix}
}$, we can have the following bound: 
\begin{align*}
   \| {z}_{out}^{k+1} \| \leq & \sum_{i=0}^k \|\mathbf{C}\|_2^{k-i} \|d^i\| \leq \sum_{i=0}^k \|\mathbf{C}\|_2^{k-i} \xi  \left({24\rho^2 \eta^2 \ell_1 \ell_2 m \tau^2} \|\bar{q}^i - x^*\|^2  \right.\\ 
   & + {24\rho^2 \eta^2 \ell_1 \ell_2 m \tau^2} \|\bar{q}^{i-1} - x^*\|^2 + \frac{48\rho^2 \eta^2 \ell_1 \ell_2 m (1 - \tau)^2}{\sigma} \left(F(\bar{y}^i) - F(x^*)\right) \\ 
   & \left. + \frac{48\rho^2 \eta^2 \ell_1 \ell_2 m (1 - \tau)^2}{\sigma} \left(F(\bar{y}^{i-1}) - F(x^*)\right) + 16\ell_1\ell_2m\rho^2\eta^4 t_0^2 (2 t_0 - 1) \|\bar{\mathcal{G}}^{i+1}\|^2 \right).
\end{align*}
To determine $\|\mathbf{C}\|_2$, note that $(1 - \tau)^2 + \tau^2 \leq 1$ for $\forall \tau \in [0, 1]$, one has
\begin{align*}
    \|\mathbf{C}\|_2 \leq & \|\mathbf{C}\|_F \\
    \leq & \rho^2\left( 6 + 2 \left(1 + 16\rho^2\eta^2\ell_1\ell_2 t_0   \right) + 12 \eta^2 \ell_1 \ell_2 + \frac{6 \rho^2  (3 \eta^2 \ell_1 \ell_2 + 1) \left(\frac{1}{t_0\eta} + \frac{\sigma}{4}\right)^2}{\left(\frac{1}{\alpha} + \frac{\sigma}{4}\right)^2} +  12\rho^2 \eta^2 \ell_1 \ell_2 \right.\\
    & \left.+ \frac{6 \rho^2}{\left(\frac{1}{\alpha} + \frac{\sigma }{4}\right)^2 }\left(\frac{1}{t_0^2 \eta^2} +  \left(1 + 16\rho^2\eta^2\ell_1\ell_2 t_0  \right)\left(\frac{1}{t_0\eta} + \frac{\sigma}{4}\right)^2 \right) + \frac{3}{\left(1 + \frac{\alpha \sigma}{4}\right)^2} \right)\\
     = & \rho^2\left( 8 + 32\rho^2\eta^2\ell_1\ell_2 (t_0 - 1) +  12 (1 + \rho^2)  \eta^2 \ell_1 \ell_2  \right.\\
     & \left. + \frac{6 \rho^2}{\left(\frac{1}{\alpha} + \frac{\sigma }{4}\right)^2 }\left(\frac{1}{t_0^2 \eta^2} +  \left(2 + 3 \eta^2 \ell_1\ell_2  + 16\rho^2\eta^2\ell_1\ell_2 t_0   \right)\left(\frac{1}{t_0\eta} + \frac{\sigma}{4}\right)^2 \right)  + \frac{3}{\left(1 + \frac{\alpha \sigma}{4}\right)^2}\right) \\
     \leq & \rho^2\left( 11 + 32\rho^2\eta^2\ell_1\ell_2 t_0 +  12 (1 + \rho^2)  \eta^2 \ell_1 \ell_2  +
      \frac{6 \rho^2}{\left(\frac{1}{\alpha} + \frac{\sigma }{4}\right)^2 }  \left(3 + 3 \eta^2 \ell_1\ell_2  + 16\rho^2\eta^2\ell_1\ell_2 t_0 \right)\left(\frac{1}{t_0\eta} + \frac{\sigma}{4}\right)^2  \right) .
\end{align*}
\subsection{Proof of Theorem \ref{theorem_strong_convex}}
We restate the main theorem \ref{theorem_strong_convex} with detail of hyperparameters as below:
\begin{theorem}
    Running Algo. \ref{strongconvex_algo_2} with the hyperparameters $ \eta \leq \min\left(\frac{1}{2L}, \frac{1}{8}\sqrt{\frac{b}{\ell_1 \ell_2 t_0}}\right)$ and $\tau = min\{\frac{1}{2}, \frac{\sqrt{t_0 \eta \sigma}}{2}\}$.
Let $\rho \leq  \min\left( \sqrt{\frac{1}{64(1 + {b} + \frac{1}{\tau^4})}}, \sqrt{\frac{1}{64 m J b (\frac{k}{\sigma} + \frac{8 k L}{\sigma^2} + \frac{16 k L}{\sigma^2 \tau^4} + \eta t_0)}} \right)$, we have
\begin{align*}
     \mathbb{E}\left[F(\bar{y}^{k}) - F(x^*) \right] &\leq \frac{3}{\left( 1 + \frac{\tau}{2} \right)^{k}} \left( F(\bar{y}^0) - F(x^*) \right), \\
     \mathbb{E}\left[ \norm{\bar{x}^k - x^*}^2 \right] &\leq  \frac{128L}{\sigma\left( 1 + \frac{\tau}{2} \right)^{k}}\norm{\bar{x}^0 - x^*}^2, \\
     \mathbb{E}\left[ \norm{\bar{q}^k - x^*}^2 \right] & \leq
     \frac{1750L}{\sigma \tau^4\left( 1 + \frac{\tau}{2} \right)^{k}}\norm{\bar{x}^0 - x^*}^2 . 
\end{align*}
\end{theorem}
We prove the theorem by induction.

{\bf Induction Hypothesis}: Assume that for any $i \leq k$, we have $F(\bar{y}^i) - F(x^*) \leq \frac{C_f}{\left(1+\frac{\tau}{2}\right)^i} \left(F(\bar{x}^0) - F(x^*)\right)$, $\|\bar{q}^i - x^*\|^2 \leq \frac{C_q}{\left(1+\frac{\tau}{2}\right)^i} \|\bar{x}^0 - x^*\|^2$ and $\|\bar{x}^i - x^*\|^2 \leq \frac{C_x}{\left(1+\frac{\tau}{2}\right)^i} \|\bar{x}^0 - x^*\|^2$, then it follows:

By induction, one has:
\begin{align*}
   \| {z}_{out}^{k} \|  \leq &  \sum_{i=0}^{k-1} \|\mathbf{C}\|_2^{k-i-1} \xi  \left({24\rho^2 \eta^2 \ell_1 \ell_2 m \tau^2} \|\bar{q}^i - x^*\|^2  + {24\rho^2 \eta^2 \ell_1 \ell_2 m \tau^2} \|\bar{q}^{i-1} - x^*\|^2 + \frac{48\rho^2 \eta^2 \ell_1 \ell_2 m (1 - \tau)^2}{\sigma} \left(F(\bar{y}^i) - F(x^*)\right)  \right.\\ 
   & \left. + \frac{48\rho^2 \eta^2 \ell_1 \ell_2 m (1 - \tau)^2}{\sigma} \left(F(\bar{y}^{i-1}) - F(x^*)\right) + 32 \ell_1\ell_2m\rho^2\eta^4 t_0^3  \|\bar{\mathcal{G}}^{i+1}\|^2 \right) \\
    \leq & \sum_{i=0}^{k-1} \|\mathbf{C}\|_2^{k-i -1} \xi \left({24\rho^2 \eta^2 \ell_1 \ell_2 m \tau^2}\frac{C_q}{\left(1+\frac{\tau}{2}\right)^i} \|\bar{x}^0 - x^*\|^2 + {24\rho^2 \eta^2 \ell_1 \ell_2 m \tau^2} \frac{C_q}{\left(1+\frac{\tau}{2}\right)^{i-1}} \|\bar{x}^0 - x^*\|^2 \right. \\
   & + \frac{48\rho^2 \eta^2 \ell_1 \ell_2 m (1 - \tau)^2}{\sigma}\frac{C_f}{\left(1+\frac{\tau}{2}\right)^i} \left(F(\bar{x}^0) - F(x^*)\right) + \frac{48\rho^2 \eta^2 \ell_1 \ell_2 m (1 - \tau)^2}{\sigma}\frac{C_f}{(1+\frac{\tau}{2})^{i-1}} \left(F(\bar{x}^0) - F(x^*)\right) \\
   & \left. + 32 \ell_1\ell_2m\rho^2\eta^2 t_0  \|\bar{x}^{i + 1} - \bar{y}^{i + 1}\|^2 \right)\\
      \leq & \sum_{i=0}^{k-1} \|\mathbf{C}\|_2^{k-i -1} \xi \left({24\rho^2 \eta^2 \ell_1 \ell_2 m \tau^2}\frac{C_q}{\left(1+\frac{\tau}{2}\right)^i} \|\bar{x}^0 - x^*\|^2 + {24\rho^2 \eta^2 \ell_1 \ell_2 m \tau^2} \frac{C_q}{\left(1+\frac{\tau}{2}\right)^{i-1}} \|\bar{x}^0 - x^*\|^2 \right. \\
   &+ \frac{48\rho^2 \eta^2 \ell_1 \ell_2 m (1 - \tau)^2}{\sigma}\frac{C_f}{\left(1+\frac{\tau}{2}\right)^i} \left(F(\bar{x}^0) - F(x^*)\right) + \frac{48\rho^2 \eta^2 \ell_1 \ell_2 m (1 - \tau)^2}{\sigma}\frac{C_f}{\left(1+\frac{\tau}{2}\right)^{i-1}} \left(F(\bar{x}^0) - F(x^*)\right) \\
   & \left. + 64 \ell_1\ell_2m\rho^2\eta^2 t_0 \frac{C_x}{\left(1+\frac{\tau}{2}\right)^{i+1}} \|\bar{x}^{0} - x ^ *\|^2 + \frac{128 \ell_1\ell_2m\rho^2\eta^2 t_0}{ \sigma} \frac{C_f}{\left(1+\frac{\tau}{2}\right)^{i+1}}\left( F(\bar{y}^{0}) - F(x ^ *)\right) \right)\\
   \leq  & {48\rho^2 \eta^2 \ell_1 \ell_2 m \tau^2}C_q \xi \frac{1}{\left[1 - \|\mathbf{C}\|_2 \left(1 + \frac{\tau}{2}\right)\right]\left(1+\frac{\tau}{2}\right)^{k-2}} \|\bar{x}^0 - x^*\|^2 \\
   & + \frac{96\rho^2 \eta^2 \ell_1 \ell_2 m C_f \xi (1 - \tau)^2}{\sigma}\frac{1}{\left[1 - \|\mathbf{C}\|_2 (1 + \frac{\tau}{2})\right]\left(1+\frac{\tau}{2}\right)^{k-2}}  \left(F(\bar{x}^0) - F(x^*)\right) \\
   & + {64 \ell_1\ell_2m\rho^2\eta^2 t_0   }C_x \xi \frac{1}{\left[1 - \|\mathbf{C}\|_2 \left(1 + \frac{\tau}{2}\right)\right]\left(1+\frac{\tau}{2}\right)^{k}} \|\bar{x}^0 - x^*\|^2 \\
   & +\frac{128 \ell_1\ell_2m\rho^2\eta^2 t_0 C_f \xi}{\sigma} \frac{1}{\left[1 - \|\mathbf{C}\|_2 (1 + \frac{\tau}{2})\right]\left(1+\frac{\tau}{2}\right)^{k}}  \left(F(\bar{x}^0) - F(x^*)\right).  
\end{align*}
We can upper bound the suboptimality of the function $F(\cdot)$ at the current iterate with consensus errors and suboptimality from the previous iteration:
\begin{lemma}
For any $\alpha_{k+1} > 0$, it follows that:
\begin{equation}
\begin{split}
   \mathbb{E} &\left[\frac{\alpha_{k+1}}{\tau_{k}}\left(F(\bar{y}^{k+1}) - F(x^*)\right)\right]  \\
     \leq  \mathbb{E} & \left[ \frac{\alpha_{k+1} (1 - \tau_{k})}{\tau_{k}} \left(1  +  \frac{48\rho^2 \ell_1 \ell_2 m J \eta^2 (1 - \tau_k)}{\sigma}\right) \left(F(\bar{y}^k) - F(x^*)\right) \right.\\
    &  + \frac{\alpha_{k+1}}{\tau_k}\frac{48\rho^2 \ell_1 \ell_2 m J \eta^2 (1 - \tau_{k-1})^2}{\sigma}\left(F(\bar{y}^{k-1}) - F(x^*)\right)  + \left (  \frac{\alpha_{k+1}^2}{2} -\frac{\alpha_{k+1}\eta t_0}{4\tau_k} \right) \|\bar{\mathcal{G}}^{k+1}\|^2 \\
    &  + \left(\frac{1}{2}   +  24 \alpha_{k+1}\rho^2 \ell_1 \ell_2 m \tau_k J \eta^2 \right)\|x^* - \bar{q}^k\|^2 - \frac{\left(1+ \frac{\sigma \alpha_{k+1}}{4}\right)}{2}\|x^* - \bar{q}^{k+1}\|^2 \\
    &  + \frac{\alpha_{k+1}}{\tau_k} {24\rho^2 \ell_1 \ell_2 m \tau_{k-1}^2 J \eta^2} \|\bar{q}^{k-1} - x^*\|^2  + \frac{\alpha_{k+1}}{\tau_k} 2 \rho^2 \tau_k^2 \left( J + \frac{1}{4 \eta m  } + 6 \rho^2 \ell_1 \ell_2 \eta^2 J\right) \|\mathbf{q}^k - \mathbf{1} \bar{q}^k\|^2 \\
    &  + \frac{\alpha_{k+1}}{\tau_k} 2\rho^2 (1 - \tau_{k})^2 \left( J + \frac{1}{4 \eta m  } + 6 \rho^2 \ell_1 \ell_2 \eta^2 J\right) \| \mathbf{y}^k - \mathbf{1} \bar{y}^k\|^2     + \frac{\alpha_{k+1}}{\tau_k} {6\rho^2 \ell_1 \ell_2 J \eta^2} \| \mathbf{x}^{k} - \mathbf{1}\bar{x}^{k}\|^2 \\
    & \left. + \frac{\alpha_{k+1}}{\tau_k} 2\rho^2  J \eta^2 \|\hat{\mathbf{s}}^{k} - \mathbf{1} \bar{\hat{s}}^{k}\|^2 \right] .
    \end{split}
    \label{out_recursive_form}
\end{equation}
\end{lemma}
\begin{proof}
 Recall that $\bar{\mathcal{G}}^{k + 1} = \frac{\bar{w}^0 - \bar{w}^{T+1} }{\eta t_0}$, from Lemma \ref{lemma_inner} we have: 
\begin{align*}
\small\begin{split}
        \mathbb{E} & \left[\alpha_{k+1}\left(F(\bar{y}^{k+1}) - F(x^*)\right)\right]  \\
     \leq \mathbb{E} &\left[-\frac{\alpha_{k+1}\eta t_0\|\bar{\mathcal{G}}^{k+1} \|^2}{4} + \alpha_{k+1}\langle \bar{x}^{k+1} - x^*, \bar{\mathcal{G}}^{k+1}\rangle  - \frac{ \alpha_{k+1}\sigma}{8}\| x^* - \bar{y}^{k+1}\|^2 \right.\\
    &  \left.   + \left(\alpha_{k+1}  J + \frac{ \alpha_{k+1}}{4 \eta m  }\right)\| \mathbf{x}^{k+1} - \mathbf{1}\bar{x}^{k+1}  \|^2 +  \alpha_{k+1} J \eta^2 \| \hat{\mathbf{s}}^{k+1} - \mathbf{1}\bar{\hat{s}}^{k+1}  \|^2 \right]\\
     \leq \mathbb{E} &\left[ \alpha_{k+1}\langle \bar{x}^{k+1} - \bar{q}^k, \bar{\mathcal{G}}^{k+1}\rangle -\frac{\alpha_{k+1}\eta t_0\|\bar{\mathcal{G}}^{k+1} \|^2}{4} + \alpha_{k+1}(\langle \bar{q}^k - x^*, \bar{\mathcal{G}}^{k+1}\rangle \right.\\
    & \left.  -  \frac{\sigma}{8}\| x^* - \bar{y}^{k+1}\|^2) + \left(\alpha_{k+1}  J + \frac{ \alpha_{k+1}}{4 \eta m  }\right)\| \mathbf{x}^{k+1} - \mathbf{1}\bar{x}^{k+1}  \|^2 +  \alpha_{k+1} J \eta^2 \| \hat{\mathbf{s}}^{k+1} - \mathbf{1}\bar{\hat{s}}^{k+1}  \|^2 \right]\\
     \leq \mathbb{E}&\left[ \frac{\alpha_{k+1} (1 - \tau_{k})}{\tau_{k}}\langle \bar{y}^k - \bar{x}^{k+1}, \bar{\mathcal{G}}^{k+1}\rangle -\frac{\alpha_{k+1}\eta t_0\|\bar{\mathcal{G}}^{k+1} \|^2}{4} + \alpha_{k+1}\left(\langle \bar{q}^k - x^*, \bar{\mathcal{G}}^{k+1}\rangle -  \frac{\sigma}{8}\| x^* - \bar{y}^{k+1}\|^2\right) \right.\\
    &  \left.  + \left(\alpha_{k+1}  J + \frac{ \alpha_{k+1}}{4 \eta m  }\right)\| \mathbf{x}^{k+1} - \mathbf{1}\bar{x}^{k+1}  \|^2 +  \alpha_{k+1} J \eta^2 \| \hat{\mathbf{s}}^{k+1} - \mathbf{1}\bar{\hat{s}}^{k+1}  \|^2 \right]\\
      \leq \mathbb{E}&\left[ \frac{\alpha_{k+1} (1 - \tau_{k})}{\tau_{k}}\langle \bar{y}^k - \bar{x}^{k+1}, \bar{\mathcal{G}}^{k+1}\rangle -\frac{\alpha_{k+1}\eta t_0\|\bar{\mathcal{G}}^{k+1} \|^2}{4} +  \frac{(\alpha_{k+1})^2}{2} \|\bar{\mathcal{G}}^{k+1}\|^2 \right.\\
    & \left. + \frac{\|x^* - \bar{q}^k\|^2}{2} - \frac{\left(1+ \frac{\sigma \alpha_{k+1}}{4}\right)\|x^* - \bar{q}^{k+1}\|^2}{2} + \left(\alpha_{k+1}  J + \frac{ \alpha_{k+1}}{ \eta m  }\right)\| \mathbf{x}^{k+1} - \mathbf{1}\bar{x}^{k+1}  \|^2  +  \alpha_{k+1} J \eta^2 \| \hat{\mathbf{s}}^{k+1} - \mathbf{1}\bar{\hat{s}}^{k+1}  \|^2 \right],
\end{split}
\end{align*}
where the last inequality is due to Lemma \ref{mirror_desc} with $\hat{\eta}=\alpha_{k+1}$, $\hat{w} = \bar{q}^{k+1}$, $\tilde{w} = \bar{q}^k$ and $\tilde{\psi}(u) = \frac{\sigma}{8}\norm{u - \bar{y}^{k+1}}^2$ with $u = x^*$. 

By Lemma \ref{lemma_inner} with $u = \bar{y}^{k}$ and recall that $\bar{w}^0 = \bar{x}^{k+1}$ and $\bar{w}^T = \bar{y}^{k+1}$, we have:
\begin{align*}
        \mathbb{E}&\left[\alpha_{k+1}\left(F(\bar{y}^{k+1}) - F(x^*)\right)\right]  \\
      \leq \mathbb{E}&\left[ \frac{\alpha_{k+1} (1 - \tau_{k})}{\tau_{k}}\left(F(\bar{y}^k) - F(\bar{y}^{k+1}) - \frac{t_0 \eta}{4} \|\bar{\mathcal{G}}^{k+1}\|^2 + \left( J + \frac{1}{4 \eta m   }\right)\| \mathbf{x}^{k+1} - \mathbf{1}\bar{x}^{k+1} \|^2  \right. \right.\\
    &  \left. +  J \eta^2 \| \hat{\mathbf{s}}^{k+1} - \mathbf{1}\bar{\hat{s}}^{k+1}  \|^2 \right) -\frac{\alpha_{k+1}\eta t_0\|\bar{\mathcal{G}}^{k+1} \|^2}{4} +  \frac{\alpha_{k+1}^2}{2} \|\bar{\mathcal{G}}^{k+1}\|^2\\
    &  + \frac{\|x^* - \bar{q}^k\|^2}{2} - \frac{\left(1+ \frac{\sigma \alpha_{k+1}}{4}\right)\|x^* - \bar{q}^{k+1}\|^2}{2} \\
    &  \left. + \left(\alpha_{k+1}  J + \frac{ \alpha_{k+1}}{4 \eta m }\right)\| \mathbf{x}^{k+1} - \mathbf{1}\bar{x}^{k+1}  \|^2 +  \alpha_{k+1} J \eta^2 \| \hat{\mathbf{s}}^{k+1} - \mathbf{1}\bar{\hat{s}}^{k+1}  \|^2 \right]\\
     \leq \mathbb{E}&\left[ \frac{\alpha_{k+1} (1 - \tau_{k})}{\tau_{k}}\left(F(\bar{y}^k) - F(\bar{y}^{k+1}) \right) -\frac{\alpha_{k+1}\eta t_0\|\bar{\mathcal{G}}^{k+1} \|^2}{4\tau_k} +  \frac{\alpha_{k+1}^2}{2} \|\bar{\mathcal{G}}^{k+1}\|^2 \right.\\
    &  + \frac{\|x^* - \bar{q}^k\|^2}{2} - \frac{\left(1+ \frac{\sigma \alpha_{k+1}}{4}\right)\|x^* - \bar{q}^{k+1}\|^2}{2} \\
    &  \left. + \frac{\alpha_{k+1}}{\tau_k}\left( J + \frac{1}{4 \eta m  }\right)\| \mathbf{x}^{k+1} - \mathbf{1}\bar{x}^{k+1}  \|^2 +  \frac{\alpha_{k+1}}{\tau_k} J \eta^2 \| \hat{\mathbf{s}}^{k+1} - \mathbf{1}\bar{\hat{s}}^{k+1}  \|^2 \right] . 
\end{align*}
By expanding $\| \mathbf{x}^{k+1} - \mathbf{1}\bar{x}^{k+1}  \|^2$ and $\| \hat{\mathbf{s}}^{k+1} - \mathbf{1}\bar{\hat{s}}^{k+1}  \|^2$, we have
\begin{align*}
     \mathbb{E}&\left[\alpha_{k+1}\left(F(\bar{y}^{k+1}) - F(x^*)\right)\right]  \\
      \leq \mathbb{E}&\left[ \frac{\alpha_{k+1} (1 - \tau_{k})}{\tau_{k}}\left(F(\bar{y}^k) - F(\bar{y}^{k+1}) \right) -\frac{\alpha_{k+1}\eta t_0\|\bar{\mathcal{G}}^{k+1} \|^2}{4\tau_k} +  \frac{\alpha_{k+1}^2}{2} \|\bar{\mathcal{G}}^{k+1}\|^2 \right. \\
    &  + \frac{\|x^* - \bar{q}^k\|^2}{2} - \frac{\left(1+ \frac{\sigma \alpha_{k+1}}{4}\right)\|x^* - \bar{q}^{k+1}\|^2}{2} \\
    &  + \frac{\alpha_{k+1}}{\tau_k}\left( J + \frac{1}{4 \eta m   }\right)\left(2 \rho^2 \tau_k^2 \|\mathbf{q}^k - \mathbf{1} \bar{q}^k\|^2 + 2\rho^2 (1 - \tau_{k})^2 \| \mathbf{y}^k - \mathbf{1} \bar{y}^k\|^2\right) \\
    &  +  \frac{\alpha_{k+1}}{\tau_k} J \eta^2 \left( 2\rho^2 \|\hat{\mathbf{s}}^{k} - \mathbf{1} \bar{\hat{s}}^{k}\|^2 +   {12\rho^4 \ell_1 \ell_2 \tau_k^2} \|\mathbf{q}^k - \mathbf{1} \bar{q}^k\|^2 + {12\rho^4 \ell_1 \ell_2 (1-\tau_k)^2} \| \mathbf{y}^k - \mathbf{1} \bar{y}^k\|^2 \right. \\
     &  + {6\rho^2 \ell_1 \ell_2} \| \mathbf{x}^{k} - \mathbf{1}\bar{x}^{k}\|^2   +  24\rho^2 \ell_1 \ell_2 m \tau_k^2 \|\bar{q}^k - x^*\|^2 + 
    \frac{48\rho^2 \ell_1 \ell_2 m (1 - \tau_k)^2}{\sigma}\left(F(\bar{y}^k) - F(x^*)\right) \\
    &  \left. \left. + {24\rho^2 \ell_1 \ell_2 m \tau_{k-1}^2} \|\bar{q}^{k-1} - x^*\|^2      +  \frac{48\rho^2 \ell_1 \ell_2 m (1 - \tau_{k-1})^2}{\sigma}\left(F(\bar{y}^{k-1}) - F(x^*)\right) \right) \right]\\
     = \mathbb{E}&\left[ \frac{\alpha_{k+1} (1 - \tau_{k})}{\tau_{k}}\left(F(\bar{y}^k) - F(\bar{y}^{k+1}) \right) -\frac{\alpha_{k+1}\eta t_0\|\bar{\mathcal{G}}^{k+1} \|^2}{4\tau_k} +  \frac{\alpha_{k+1}^2}{2} \|\bar{\mathcal{G}}^{k+1}\|^2 \right. \\
    &  + \frac{\|x^* - \bar{q}^k\|^2}{2} - \frac{\left(1+ \frac{\sigma \alpha_{k+1}}{4}\right)\|x^* - \bar{q}^{k+1}\|^2}{2} 
     + \frac{\alpha_{k+1}}{\tau_k} 2 \rho^2 \tau_k^2 \left( J + \frac{1}{4 \eta m } + 6 \rho^2 \ell_1 \ell_2 \eta^2 J\right) \|\mathbf{q}^k - \mathbf{1} \bar{q}^k\|^2 \\
    &  + \frac{\alpha_{k+1}}{\tau_k} 2\rho^2 (1 - \tau_{k})^2 \left( J + \frac{1}{4 \eta m  } + 6 \rho^2 \ell_1 \ell_2 \eta^2 J \right) \| \mathbf{y}^k - \mathbf{1} \bar{y}^k\|^2   \\
    &  + \frac{\alpha_{k+1}}{\tau_k} {6\rho^2 \ell_1 \ell_2 J \eta^2} \| \mathbf{x}^{k} - \mathbf{1}\bar{x}^{k}\|^2 + \frac{\alpha_{k+1}}{\tau_k} 2\rho^2  J \eta^2 \|\hat{\mathbf{s}}^{k} - \mathbf{1} \bar{\hat{s}}^{k}\|^2 \\ 
    &  + \frac{\alpha_{k+1}}{\tau_k} 24\rho^2 \ell_1 \ell_2 m \tau_k^2 J \eta^2 \|\bar{q}^k - x^*\|^2 + \frac{\alpha_{k+1}}{\tau_k} \frac{48\rho^2 \ell_1 \ell_2 m J \eta^2 (1 - \tau_k)^2}{\sigma}\left(F(\bar{y}^k) - F(x^*)\right) \\
    &  \left.\left. + \frac{\alpha_{k+1}}{\tau_k} {24\rho^2 \ell_1 \ell_2 m \tau_{k-1}^2 J \eta^2} \|\bar{q}^{k-1} - x^*\|^2      +  \frac{\alpha_{k+1}}{\tau_k}\frac{48\rho^2 \ell_1 \ell_2 m J \eta^2 (1 - \tau_{k-1})^2}{\sigma}\left(F(\bar{y}^{k-1}) - F(x^*)\right) \right) \right] .
\end{align*}
After rearranging the terms, we get the following form:
\begin{align*}
      \mathbb{E} &\left[\frac{\alpha_{k+1}}{\tau_{k}}\left(F(\bar{y}^{k+1}) - F(x^*)\right)\right]  \\
     \leq  \mathbb{E} & \left[ \frac{\alpha_{k+1} (1 - \tau_{k})}{\tau_{k}} \left(1  +  \frac{48\rho^2 \ell_1 \ell_2 m J \eta^2 (1 - \tau_k)}{\sigma}\right) \left(F(\bar{y}^k) - F(x^*)\right) \right.\\
    &  + \frac{\alpha_{k+1}}{\tau_k}\frac{48\rho^2 \ell_1 \ell_2 m J \eta^2 (1 - \tau_{k-1})^2}{\sigma}\left(F(\bar{y}^{k-1}) - F(x^*)\right)  + \left (  \frac{\alpha_{k+1}^2}{2} -\frac{\alpha_{k+1}\eta t_0}{4\tau_k} \right) \|\bar{\mathcal{G}}^{k+1}\|^2 \\
    &  + \left(\frac{1}{2}   +  24 \alpha_{k+1}\rho^2 \ell_1 \ell_2 m \tau_k J \eta^2 \right)\|x^* - \bar{q}^k\|^2 - \frac{\left(1+ \frac{\sigma \alpha_{k+1}}{4}\right)}{2}\|x^* - \bar{q}^{k+1}\|^2 \\
    &  + \frac{\alpha_{k+1}}{\tau_k} {24\rho^2 \ell_1 \ell_2 m \tau_{k-1}^2 J \eta^2} \|\bar{q}^{k-1} - x^*\|^2  + \frac{\alpha_{k+1}}{\tau_k} 2 \rho^2 \tau_k^2 \left( J + \frac{1}{4 \eta m  } + 6 \rho^2 \ell_1 \ell_2 \eta^2 J\right) \|\mathbf{q}^k - \mathbf{1} \bar{q}^k\|^2 \\
    &  + \frac{\alpha_{k+1}}{\tau_k} 2\rho^2 (1 - \tau_{k})^2 \left( J + \frac{1}{4 \eta m  } + 6 \rho^2 \ell_1 \ell_2 \eta^2 J\right) \| \mathbf{y}^k - \mathbf{1} \bar{y}^k\|^2     + \frac{\alpha_{k+1}}{\tau_k} {6\rho^2 \ell_1 \ell_2 J \eta^2} \| \mathbf{x}^{k} - \mathbf{1}\bar{x}^{k}\|^2 \\
    & \left.+ \frac{\alpha_{k+1}}{\tau_k} 2\rho^2  J \eta^2 \|\hat{\mathbf{s}}^{k} - \mathbf{1} \bar{\hat{s}}^{k}\|^2 \right] .
\end{align*}
\end{proof}
Now we show that the consensus errors on the right-hand side of inequality (\ref{out_recursive_form}) can be simplified with the variable $z_{out}^{k}$. 
\begin{lemma} Recall the definition of $z_{out}^{k+1}$ as 
\begin{equation*}
{z}_{out}^{k+1} = \begin{pmatrix}
           \|\mathbf{x}^{k+1} - \mathbf{1}\bar{x}^{k+1}\|^2\\
           \|\mathbf{y}^{k+1} - \mathbf{1}\bar{y}^{k+1}\|^2\\
           \|\mathbf{q}^{k+1} - \mathbf{1}\bar{q}^{k+1}\|^2\\
           \eta^2 \|\hat{\mathbf{s}}^{k+1} - \mathbf{1}\bar{\hat{s}}^{k+1}\|^2
    \end{pmatrix},
\end{equation*}
we can rewrite (\ref{out_recursive_form}) in term of $z_{out}^k$ as: 
\begin{equation}
\begin{split}
       \mathbb{E} &\left[\frac{\alpha}{\tau}\left(F(\bar{y}^{k+1}) - F(u)\right)\right]  \\
     \leq \mathbb{E}&\left[ \frac{\alpha (1 - \tau)}{\tau} \left(1  +  \frac{3\rho^2  m J  b (1 - \tau)}{4 \sigma t_0}\right)\left(F(\bar{y}^k) - F(x^*)\right)  
     + \frac{\alpha}{\tau}\frac{3 \rho^2   m J b (1 - \tau)^2}{4 \sigma t_0}\left(F(\bar{y}^{k-1}) - F(x^*)\right) \right. \\
     & + \left (  \frac{\alpha^2}{2} -\frac{\alpha\eta t_0}{4\tau} \right) \|\bar{\mathcal{G}}^{k+1}\|^2 
     + \left(\frac{1}{2}   +  \frac{3 \alpha\rho^2  m b \tau J  }{ 8 t_0}\right)\|x^* - \bar{q}^k\|^2 - \frac{\left(1+ \frac{\sigma \alpha}{4}\right)}{2}\|x^* - \bar{q}^{k+1}\|^2 \\
     & \left. +  \frac{3 \alpha \rho^2  m  J b \tau}{8 t_0} \|\bar{q}^{k-1} - x^*\|^2  + \frac{ \alpha \rho^2 J b}{\tau}   \|\mathbf{z}_{out}^k\|^2 \right]. 
\end{split}
 \label{outer_recur}
\end{equation}
\end{lemma}
\begin{proof}
Recall that in the strongly convex case, we fix $\tau_k = \tau$ and $\alpha_k = \alpha$ for $\forall k$. Then  we have 
\begin{align*}
     \mathbb{E} &\left[\frac{\alpha}{\tau}\left(F(\bar{y}^{k+1}) - F(x^*)\right)\right]  \\
      \leq \mathbb{E}&\left[ \frac{\alpha (1 - \tau)}{\tau} \left(1  +  \frac{48\rho^2 \ell_1 \ell_2 m J \eta^2 (1 - \tau)}{\sigma}\right)\left(F(\bar{y}^k) - F(x^*)\right) \right. \\
     & + \frac{\alpha}{\tau}\frac{48\rho^2 \ell_1 \ell_2 m J \eta^2 (1 - \tau)^2}{\sigma}\left(F(\bar{y}^{k-1}) - F(x^*)\right)  + \left (  \frac{\alpha^2}{2} -\frac{\alpha\eta t_0}{4\tau} \right) \|\bar{\mathcal{G}}^{k+1}\|^2 \\
    &  + \left(\frac{1}{2}   +  24 \alpha\rho^2 \ell_1 \ell_2 m \tau J \eta^2\right)\|x^* - \bar{q}^k\|^2 - \frac{\left(1+ \frac{\sigma \alpha}{4}\right)}{2}\|x^* - \bar{q}^{k+1}\|^2 \\
    &  + \frac{\alpha}{\tau} {24\rho^2 \ell_1 \ell_2 m \tau^2 J \eta^2} \|\bar{q}^{k-1} - x^*\|^2 \\
    &  + \frac{\alpha}{\tau} 2 \rho^2 \tau^2 \left( J + \frac{1}{4 \eta m  } + 6 \rho^2 \ell_1 \ell_2 \eta^2 J\right) \|\mathbf{q}^k - \mathbf{1} \bar{q}^k\|^2 \\
    &  + \frac{\alpha}{\tau} 2\rho^2 (1 - \tau)^2 \left( J + \frac{1}{4 \eta m  } + 6 \rho^2 \ell_1 \ell_2 \eta^2 J\right) \| \mathbf{y}^k - \mathbf{1} \bar{y}^k\|^2   \\
    &  \left. + \frac{\alpha}{\tau} {6\rho^2 \ell_1 \ell_2 J \eta^2} \| \mathbf{x}^{k} - \mathbf{1}\bar{x}^{k}\|^2 + \frac{\alpha}{\tau} 2\rho^2  J \eta^2 \|\hat{\mathbf{s}}^{k} - \mathbf{1} \bar{\hat{s}}^{k}\|^2 \right].
\end{align*}
One can show that:
\begin{align*}
    &   2  \tau^2 \left( J + \frac{1}{4 \eta m  } + 6 \rho^2 \ell_1 \ell_2 \eta^2 J\right) \|\mathbf{q}^k - \mathbf{1} \bar{q}^k\|^2 \\
    &  +  2\rho^2 (1 - \tau)^2 \left( J + \frac{1}{4 \eta m  } + 6 \rho^2 \ell_1 \ell_2 \eta^2 J\right) \| \mathbf{y}^k - \mathbf{1} \bar{y}^k\|^2   \\
    &  \left. +  6 \ell_1 \ell_2 J \eta^2 \| \mathbf{x}^{k} - \mathbf{1}\bar{x}^{k}\|^2 +  2\rho^2  J \eta^2 \|\hat{\mathbf{s}}^{k} - \mathbf{1} \bar{\hat{s}}^{k}\|^2 \right]\\
    \leq &  2   \left( 2J + \frac{1}{4 \eta m  } + 6  \ell_1 \ell_2 \eta^2 J + 3 \eta^2 \ell_1 \ell_2 J \right) \|\mathbf{z}_{out}^k\|^2\\
    \leq & 2    \left( 2J + \frac{1}{4 \eta m  }  + \frac{ J b}{16 t_0}\left(1 + 2\rho^2\right) \right)\|\mathbf{z}_{out}^k\|^2 \\
    \leq & 2      \frac{ J b}{16 t_0}(1 + 2\rho^2 + 48 t_0) \|\mathbf{z}_{out}^k\|^2 \\
    \leq & {   J b}   \|\mathbf{z}_{out}^k\|^2.
\end{align*}
Note that the third inequality follows from $\frac{1}{4\eta m }  \leq J$ and $b \geq 1$ is the batch size.
Combining the above two inequalities, we have:
\begin{align*}
  \mathbb{E} &\left[\frac{\alpha}{\tau}\left(F(\bar{y}^{k+1}) - F(x^*)\right)\right]  \\
       \leq \mathbb{E}&\left[ \frac{\alpha (1 - \tau)}{\tau} \left(1  +  \frac{3\rho^2  m J  b (1 - \tau)}{4 \sigma t_0}\right)\left(F(\bar{y}^k) - F(x^*)\right) \right.\\
    &  + \frac{\alpha}{\tau}\frac{3 \rho^2   m J b (1 - \tau)^2}{4 \sigma t_0}\left(F(\bar{y}^{k-1}) - F(x^*)\right)  + \left (  \frac{\alpha^2}{2} -\frac{\alpha\eta t_0}{4\tau} \right) \|\bar{\mathcal{G}}^{k+1}\|^2 \\
    &  + \left(\frac{1}{2}   +  \frac{3 \alpha\rho^2  m b \tau J  }{ 8 t_0}\right)\|x^* - \bar{q}^k\|^2 - \frac{1+ \frac{\sigma \alpha}{4}}{2}\|x^* - \bar{q}^{k+1}\|^2 \\
    &  \left. +  \frac{3 \alpha \rho^2  m  J b \tau}{8 t_0} \|\bar{q}^{k-1} - x^*\|^2  + \frac{ \alpha \rho^2 J b}{\tau}   \|\mathbf{z}_{out}^k\|^2 \right]. 
\end{align*}
\end{proof} 
Now we show the main convergence results in Theorem \ref{theorem_strong_convex}: 
\begin{proof} 
Recall that we use induction for the proof of our main theorem. Assume for any $i \leq k$, we have
\begin{align*}
        & F(\bar{y}^i) - F(x^*) \leq \frac{C_f}{(1+\frac{\tau}{2})^i} \left(F(\bar{x}^0) - F(x^*)\right) ,\\
    & \|\bar{q}^i - x^*\|^2 \leq \frac{C_q}{(1+\frac{\tau}{2})^i} \|\bar{x}^0 - x^*\|^2 ,\\
    & \|\bar{x}^i - x^*\|^2 \leq \frac{C_x}{(1+\frac{\tau}{2})^i} \|\bar{x}^0 - x^*\|^2 . 
\end{align*}
Let {$\alpha = \frac{t_0 \eta}{2\tau} = \frac{2\tau}{\sigma}$},                     
recall the definition of  $\xi = \norm{
\begin{pmatrix}
      0\\
       1\\
       \frac{3 \rho^2   (\frac{1}{t_0\eta} + \frac{\sigma}{4})^2}{(\frac{1}{\alpha} + \frac{\sigma}{4})^2} \\
       1\\
    \end{pmatrix}
}$ we have
\begin{align*}
    \xi & \leq 2 + \frac{3 \rho^2   \left(\frac{1}{t_0\eta} + \frac{\sigma}{4}\right)^2}{\left(\frac{1}{\alpha} + \frac{\sigma}{4}\right)^2}  \leq 2 + \frac{3 \rho^2   \left(\frac{1}{t_0\eta} + \frac{\sigma}{4}\right)^2}{(\frac{\sigma}{4})^2} \\
    & \leq  2 + 3 \rho^2   \left(\frac{4}{t_0\eta \sigma} + 1\right)^2 \leq  2 + 3 \rho^2   \left(\frac{1}{\tau^2} + 1\right)^2 \\
    &\leq  2 + 3 \rho^2   \left(\frac{2}{\tau^2}\right)^2 \leq 2 + 12 \frac{\rho^2}{\tau^4}  \leq 3.
\end{align*}

Recall the recursive formula of $\mathbf{z}_{out}$ in Eq. (\ref{z_out_recur}), we can bound the norm of $C$ by:
\begin{align*}
       \|\mathbf{C}\|_2\leq & \rho^2\left[ 11 + 32\rho^2\eta^2\ell_1\ell_2 t_0 +  12 (1 + \rho^2)  \eta^2 \ell_1 \ell_2 \right. \\
     & \left. + \frac{6 \rho^2}{\left(\frac{1}{\alpha} + \frac{\sigma }{4}\right)^2 }  \left(3 + 3 \eta^2 \ell_1\ell_2  + 16\rho^2\eta^2\ell_1\ell_2 t_0  \right)\left(\frac{1}{t_0\eta} + \frac{\sigma}{4}\right)^2  \right] \\
     \leq & \rho^2\left( 11 + 32\rho^2\eta^2\ell_1\ell_2 t_0 +  12 (1 + \rho^2)  \eta^2 \ell_1 \ell_2   + 6 + 6 \eta^2 \ell_1\ell_2  + 32\rho^2\eta^2\ell_1\ell_2 t_0 \right) \\
     \leq & \rho^2(17 + 64\rho^2\eta^2\ell_1\ell_2 t_0 + 18 \eta^2 \ell_1\ell_2 + 12 \rho^2 \eta^2 \ell_1\ell_2) \\
     \leq & \rho^2\left(17 + \rho^2 b + \frac{b}{2t_0} +   \frac{\rho^2 b}{4t_0}\right) \leq \rho^2\left(18 + \frac{b}{64} + \frac{b}{2 t_0}\right) \leq \rho^2(18 + b) \leq \frac{1}{2}.
\end{align*}
Recall that we can bound $\|\mathbf{z}_{out}^k\|^2$ by:
\begin{align*}
    \|\mathbf{z}_{out}^k\|^2  \leq&  {48\rho^2 \eta^2 \ell_1 \ell_2 m \tau^2}C_q \xi\frac{1}{\left[1 - \|C\|_2 \left(1 + \frac{\tau}{2}\right)\right]\left(1+\frac{\tau}{2}\right)^{k-2}} \|\bar{x}^0 - x^*\|^2 \\
    & + \frac{96\rho^2 \eta^2 \ell_1 \ell_2 m \xi C_f (1 - \tau)^2}{\sigma}  \frac{1}{\left[1 - \|C\|_2 \left(1 + \frac{\tau}{2}\right)\right]\left(1+\frac{\tau}{2}\right)^{k-2}} \left(F(\bar{x}^0) - F(x^*)\right) \\
   & + {64 \ell_1\ell_2m\rho^2\eta^2 t_0   }C_x \xi\frac{1}{\left[1 - \|C\|_2 \left(1 + \frac{\tau}{2}\right)\right]\left(1+\frac{\tau}{2}\right)^{k}} \|\bar{x}^0 - x^*\|^2 \\
   & +\frac{128 \ell_1\ell_2m\rho^2\eta^2 t_0 \xi C_f}{\sigma}\frac{1}{\left[1 - \|C\|_2 \left(1 + \frac{\tau}{2}\right)\right]\left(1+\frac{\tau}{2}\right)^{k}} \left(F(\bar{x}^0) - F(x^*)\right)  \\
    \leq &  \frac{9\rho^2 b m \tau^2 C_q}{4 t_0} \frac{1}{\left(\frac{1}{2} - \frac{\tau}{4}\right)\left(1+\frac{\tau}{2}\right)^{k-2}} \|\bar{x}^0 - x^*\|^2  + \frac{9\rho^2 b m C_f (1 - \tau)^2}{2 t_0 \sigma}\frac{1}{\left(\frac{1}{2} - \frac{\tau}{4}\right)\left(1+\frac{\tau}{2}\right)^{k-2}} \left(F(\bar{x}^0) - F(x^*)\right) \\
   & + {3 \rho^2 m b   C_x } \frac{1}{\left(\frac{1}{2} - \frac{\tau}{4}\right) \left(1+\frac{\tau}{2}\right)^{k}} \|\bar{x}^0 - x^*\|^2 +\frac{6 \rho^2 m  b C_f }{ \sigma}\frac{1}{\left(\frac{1}{2} - \frac{\tau}{4}\right)\left(1+\frac{\tau}{2}\right)^{k}} \left(F(\bar{x}^0) - F(x^*)\right)  \\
     \leq & \frac{27\rho^2 b m \tau^2 C_q}{4 t_0} \frac{1}{\left(1+\frac{\tau}{2}\right)^{k-2}} \|\bar{x}^0 - x^*\|^2  + \frac{27\rho^2 b m C_f (1 - \tau)^2}{2 t_0 \sigma}\frac{1}{\left(1+\frac{\tau}{2}\right)^{k-2}} \left(F(\bar{x}^0) - F(x^*)\right) \\
   & + {9 \rho^2 m b   C_x } \frac{1}{\left(1+\frac{\tau}{2}\right)^{k}} \|\bar{x}^0 - x^*\|^2 +\frac{18 \rho^2 m  b C_f }{ \sigma}\frac{1}{\left(1+\frac{\tau}{2}\right)^{k}}\left(F(\bar{x}^0) - F(x^*)\right)  \\
   \leq & 9\rho^2 m b\left( C_x +  C_q\right) \frac{1}{\left(1+\frac{\tau}{2}\right)^{k-2}} \|\bar{x}^0 - x^*\|^2 + \frac{36 \rho^2 m b C_f}{\sigma}\frac{1}{\left(1+\frac{\tau}{2}\right)^{k}} \left(F(\bar{x}^0) - F(x^*)\right),
\end{align*}
where the second inequality is followed from $\eta^2 \ell_1 \ell_2 \leq \frac{b}{64 t_0}$ and $\norm{C}_2 \leq \frac{1}{2}$. We use $\frac{1}{ \frac{1}{2} - \frac{\tau}{4}} \leq 3$ in the third inequality and the last inequality simply follows from $0 \leq \tau, 1-\tau \leq 1$. 

{\bf Induction Step}: By combining Eq. (\ref{outer_recur}) with $u = x^*$ and induction rules $F(\bar{y}^i) - F(x^*) \leq \frac{C_f}{\left(1+\frac{\tau}{2}\right)^i} \left(F(\bar{x}^0) - F(x^*)\right)$, $\|\bar{q}^i - x^*\|^2 \leq \frac{C_q}{\left(1+\frac{\tau}{2}\right)^i} \|\bar{x}^0 - x^*\|^2$ and $\|\bar{x}^i - x^*\|^2 \leq \frac{C_x}{\left(1+\frac{\tau}{2}\right)^i} \|\bar{x}^0 - x^*\|^2$, it follows
\begin{align*}
     \mathbb{E}&\left[\frac{2}{\sigma}\left(F(\bar{y}^{k+1}) - F(x^*)\right)\right]  \\
      \leq \mathbb{E}&\left[ \frac{ 2(1 - \frac{\tau}{2} - \frac{\tau}{2})}{\sigma} \left(1  +  \frac{3\rho^2  m J  b (1 - \tau)}{4 \sigma t_0}\right)\left(F(\bar{y}^{k}) - F(x^*)\right) \right. \\
    &  + \frac{2}{\sigma}\frac{3 \rho^2   m J b (1 - \tau)^2}{4 \sigma t_0}\frac{1}{\left(1 + \frac{\tau}{2}\right)^{k-1}}C_f \left(F(\bar{y}^{0}) - F(x^*)\right)  + \left (  \frac{\alpha^2}{2} -\frac{\alpha\eta t_0}{4\tau} \right) \|\bar{\mathcal{G}}^{k+1}\|^2 \\
    &  + \left(\frac{1}{2}   +  \frac{3 \rho^2  m b \tau^2 J  }{ 4 \sigma t_0}\right)\|x^* - \bar{q}^k\|^2 - \frac{1+ \frac{\sigma \alpha}{4}}{2}\|x^* - \bar{q}^{k+1}\|^2 \\
    &  +  \frac{3  \rho^2  m  J b \tau^2}{4 \sigma t_0}\frac{1}{\left(1 + \frac{\tau}{2}\right)^{k-1}} \|\bar{q}^{0} - x^*\|^2  +    \frac{ \rho^2 J m b\left( C_x +  C_q\right)}{\sigma} \frac{1}{\left(1 + \frac{\tau}{2}\right)^{k-2}} \|\bar{x}^0 - x^*\|^2 \\
    & \left.  + \frac{ \rho^2 m J b}{\sigma^2}\frac{1}{\left(1 + \frac{\tau}{2}\right)^{k}} C_f \left(F(\bar{x}^0) - F(x^*)\right)\right] \\
      \leq \mathbb{E}&\left[ \frac{ 2\left(1 - \frac{\tau}{2}\right)}{\sigma}\left(F(\bar{y}^{k}) - F(x^*)\right) + \frac{2 \left(1 - \frac{\tau}{2}\right)}{\sigma}\frac{3\rho^2  m J  b (1 - \tau)}{4 \sigma t_0}\frac{1}{\left(1 + \frac{\tau}{2}\right)^{k}} C_f\left(F(\bar{y}^{0}) - F(x^*)\right) \right.\\
    &  + \frac{2}{\sigma}\frac{3 \rho^2   m J b (1 - \tau)^2}{4 \sigma t_0}\frac{1}{\left(1 + \frac{\tau}{2}\right)^{k-1}} C_f\left(F(\bar{y}^{0}) - F(x^*)\right)  \\
    &  + \frac{1}{2} \|x^* - \bar{q}^k\|^2 - \frac{1+ \frac{\tau}{2}}{2}\|x^* - \bar{q}^{k+1}\|^2 + \frac{3 \rho^2  m b \tau^2 J  }{ 4 \sigma t_0} \frac{1}{\left(1 + \frac{\tau}{2}\right)^{k}} C_q \|x^* - \bar{q}^0\|^2 \\
    &  +  \frac{3  \rho^2  m  J b \tau^2 }{4 \sigma t_0}\frac{1}{\left(1 + \frac{\tau}{2}\right)^{k-1}} C_q \|\bar{q}^{0} - x^*\|^2  +    \frac{ \rho^2 J m b( C_x +  C_q)}{\sigma}  \frac{1}{\left(1 + \frac{\tau}{2}\right)^{k-2}} \|\bar{x}^0 - x^*\|^2 \\
    &  \left. + \frac{ \rho^2 m J b}{\sigma^2}\frac{1}{\left(1 + \frac{\tau}{2}\right)^{k}} C_f \left(F(\bar{x}^0) - F(x^*)\right)\right] \\
      \leq \mathbb{E}&\left[ \frac{ 2(1 - \frac{\tau}{2})}{\sigma}\left(F(\bar{y}^{k}) - F(x^*)\right) + \frac{3\rho^2  m J  b C_f }{2 \sigma^2 t_0}\frac{1}{\left(1 + \frac{\tau}{2}\right)^{k+1}} \left(F(\bar{y}^{0}) - F(x^*)\right) \right. \\
     &  + \frac{3 \rho^2   m J b C_f }{2 \sigma^2 t_0}\frac{1}{\left(1 + \frac{\tau}{2}\right)^{k+1}} \left(F(\bar{y}^{0}) - F(x^*)\right)  \\
     &  + \frac{1}{2} \|x^* - \bar{q}^k\|^2 - \frac{1+ \frac{\tau}{2}}{2}\|x^* - \bar{q}^{k+1}\|^2 + \frac{3 \rho^2  m b  J C_q }{ 2 \sigma^2 t_0} \frac{1}{\left(1 + \frac{\tau}{2}\right)^{k}}  \left(F(\bar{q}^0) - F(x^*)\right) \\
      &  +  \frac{3  \rho^2  m  J b C_q }{2 \sigma^2 t_0}\frac{1}{\left(1 + \frac{\tau}{2}\right)^{k-1}} \left( f(\bar{q}^{0}) - f(x^*)\right)  +    \frac{ 2\rho^2 J m b( C_x +  C_q)}{\sigma^2}  \frac{1}{\left(1 + \frac{\tau}{2}\right)^{k-2}} \left(f(\bar{x}^0) - f(x^*)\right) \\
      &  \left. + \frac{ \rho^2 m J b C_f}{\sigma^2}\frac{1}{\left(1 + \frac{\tau}{2}\right)^{k}}  \left(F(\bar{x}^0) - F(x^*)\right) \right]\\
        \leq  \mathbb{E}&\left[ \frac{ 2\left(1 - \frac{\tau}{2}\right)}{\sigma}\left(F(\bar{y}^{k}) - F(x^*)\right) + \frac{5\rho^2  m J  b C_f }{ \sigma^2 }\frac{1}{\left(1 + \frac{\tau}{2}\right)^{k+1}} \left(F(\bar{y}^{0}) - F(x^*)\right) \right. \\
        &  \left. + \frac{1}{2} \|x^* - \bar{q}^k\|^2 - \frac{(1+ \frac{\tau}{2})}{2}\|x^* - \bar{q}^{k+1}\|^2 + \frac{ \rho^2  m b  J (5 C_x + 10 C_q) }{  \sigma^2 } \frac{1}{\left(1 + \frac{\tau}{2}\right)^{k+1}}  \left(F(\bar{q}^0) - F(x^*)\right) \right],
\end{align*}
where the second inequality uses the $\alpha = \frac{t_0 \eta}{2 \tau} = \frac{2 \tau}{\sigma}$; the third inequality uses the fact that $1- \tau \leq 1 - \frac{\tau }{2} \leq \frac{1}{ 1 + \frac{\tau}{2}}$ and $\norm{\hat{x} - x^*}^2 \leq \frac{2}{\sigma} \left( F(\hat{x}) - F(x^*)\right)$ by the strong convexity of $F$.

Recall that $x_i^0$, $y_i^0$ and $q_i^0$ are all initialized to the same value, so $\bar{x}^0$ , $\bar{y}^0$ and $\bar{q}^0$ are the same values. Use the fact $1 - \frac{\tau}{2} \leq \frac{1}{1 + \frac{\tau}{2}}$, one immediately has
\begin{align*}
\small\begin{split}
     &\mathbb{E}\left[F(\bar{y}^{k+1}) - F(x^*)\right]  \\
        \leq &\mathbb{E}\left[ \frac{1}{\left(1 + \frac{\tau}{2}\right)^{k+1}}\left(F(\bar{y}^{0}) - F(x^*) + \frac{\sigma(1 + \frac{\tau}{2})}{4}\|x^* - \bar{q}^{0}\|^2\right) \right.  
         \left. + \frac{ \rho^2  m b k J (5 C_f + 5 C_x + 10 C_q) }{  2 \sigma } \frac{1}{\left(1 + \frac{\tau}{2}\right)^{k+1}}  \left(F(\bar{q}^0) - F(x^*)\right) \right]\\
             \leq & \frac{3}{\left(1 + \frac{\tau}{2}\right)^{k+1}}\left(F(\bar{y}^{0}) - F(x^*)\right) .
\end{split}
\end{align*}
Therefore, we have $\mathbb{E}\left[F(\bar{y}^{k+1}) - F(x^*)\right] \leq \frac{3}{\left(1 + \frac{\tau}{2}\right)^{k+1}}\left(F(\bar{y}^{0}) - F(x^*)\right)$.

Recall that we choose  $\alpha = \frac{t_0 \eta}{2\tau} = \frac{2\tau}{\sigma}$ for the strongly convex function.
By combining $\bar{x}^{k+1} = \tau_{k} \bar{q}^k + (1 - \tau_{k}) \bar{y}^k$ and $\bar{q}^{k+1} = \frac{1}{\frac{1}{\alpha_{k+1}} + \frac{\sigma}{4}}\left(\frac{\bar{q}_k}{\alpha_{k+1}} + \left(\frac{1}{t_0\eta} + \frac{\sigma}{4}\right)\bar{y}_{k+1} - \frac{1}{t_0\eta}\bar{x}_{k+1}\right)$, we can see that
\begin{align*}
    & \frac{\bar{x}^{k+1} - (1 - \tau_{k})\bar{y}^{k}}{\tau_{k}} = \frac{1}{\frac{1}{\alpha_{k+1}} + \frac{\sigma}{4}}\frac{\bar{x}^{k} - (1 - \tau_{k})\bar{y}^{k-1}}{\alpha_{k+1} \tau_{k}} + \frac{1}{\frac{1}{\alpha_{k+1}} + \frac{\sigma}{4}} \left(\frac{1}{t_0\eta} + \frac{\sigma}{4}\right)\bar{y}^{k} - \frac{1}{t_0\eta \left(\frac{1}{\alpha_{k+1}} + \frac{\sigma}{4}\right)}\bar{x}^{k} .
\end{align*}
  After rearranging the terms, we have 
  \begin{align*} 
   \bar{x}^{k+1} &= (1 - \tau_{k})\bar{y}^{k} + \frac{1}{1 + \frac{\alpha_{k+1} \sigma}{4}}\left(\bar{x}^{k} - (1 - \tau_{k})\bar{y}^{k-1}\right) + \frac{\tau_{k}}{\frac{1}{\alpha_{k+1}} + \frac{\sigma}{4}} \left(\frac{1}{t_0\eta} + \frac{\sigma}{4}\right)\bar{y}^{k} - \frac{\tau_{k}}{t_0\eta \left(\frac{1}{\alpha_{k+1}} + \frac{\sigma}{4}\right)}\bar{x}^{k} \\
  &= \left(\frac{1}{1 + \frac{\alpha_{k+1} \sigma}{4}}- \frac{\tau_{k}}{t_0 \eta \left(\frac{1}{\alpha_{k+1}} + \frac{\sigma}{4}\right)}\right)\bar{x}^{k} + \left(1 - \tau_{k} + \frac{\tau_{k}}{\frac{1}{\alpha_{k+1}} + \frac{\sigma}{4}} \left(\frac{1}{t_0\eta} + \frac{\sigma}{4}\right)\right) \bar{y}^k - \frac{1 - \tau_{k}}{1 + \frac{\alpha_{k+1} \sigma}{4}}\bar{y}^{k-1}.
\end{align*}
Assuming $\|\bar{x}^{k} - x^*\|^2 \leq C_x \frac{1}{\left(1 + \frac{\tau}{2}\right)^{k}} \|\bar{x}^{0} - x^*\|^2$ with $C_x = \frac{128L}{\sigma}$, we can deduce the rate of convergence $\|\bar{x}^{k+1} - x^*\|$ by:
\begin{align*}
    &\|\bar{x}^{k+1} - x^*\|^2 \\
    \leq & 2\left(\frac{1}{1 + \frac{\alpha \sigma}{4}}- \frac{\tau}{t_0\eta \left(\frac{1}{\alpha} + \frac{\sigma}{4}\right)}\right)^2 \|\bar{x}^k - x^*\|^2 +  4\left(1 - \tau + \frac{\tau}{\frac{1}{\alpha} + \frac{\sigma}{4}} \left(\frac{1}{t_0\eta} + \frac{\sigma}{4}\right)\right)^2 \| \bar{y}^k - x^*\|^2 \\
    &  + 4\frac{(1 - \tau)^2}{\left(1 + \frac{\alpha \sigma}{4}\right)^2}\|\bar{y}^{k-1} - x^*\|^2 \\
    \leq & 2\left(\frac{1}{1 + \frac{\alpha \sigma}{4}}- \frac{1}{2 \alpha (\frac{1}{\alpha} + \frac{\sigma}{4})}\right)^2 C_x \frac{1}{\left(1 + \frac{\tau}{2}\right)^k}\|\bar{x}^0 - x^*\|^2 +  \frac{8}{\sigma}\left(1  + \frac{\tau}{\frac{\sigma}{2 \tau} + \frac{\sigma}{4}} \left(\frac{\sigma}{4 \tau^2} + \frac{\sigma}{4}\right)\right)^2 \left(F(\bar{y}^k) - F(x^*)\right) \\
    &  + \frac{8}{\sigma}\frac{(1 - \tau)^2}{\left(1 + \frac{\alpha \sigma}{4}\right)^2}\left(F(\bar{y}^{k-1}) - F(x^*)\right) \\
     \leq & 2\left(\frac{1}{1 + \frac{\alpha \sigma}{4}}- \frac{1}{2 \left({1} + \frac{\alpha\sigma}{4}\right)}\right)^2 C_x \frac{1}{\left(1 + \frac{\tau}{2}\right)^k}\|\bar{x}^0 - x^*\|^2 +  \frac{8}{\sigma}\left(1  + \frac{\tau}{\frac{1}{2 \tau} + \frac{1}{4}} \left(\frac{1}{4 \tau^2} + \frac{1}{4}\right)\right)^2C_f\frac{1}{\left(1 + \frac{\tau}{2}\right)^k} \left(F(\bar{y}^0) - F(x^*)\right) \\
    &  + \frac{8}{\sigma}\frac{(1 - \tau)^2}{\left(1 + \frac{\alpha \sigma}{4}\right)^2}C_f\frac{1}{\left(1 + \frac{\tau}{2}\right)^{k-1}}\left(F(\bar{y}^{0}) - F(x^*)\right) \\
     \leq & \frac{1}{2\left(1 + \frac{\tau}{2}\right)^2}  C_x \frac{1}{\left(1 + \frac{\tau}{2}\right)^k}\|\bar{x}^0 - x^*\|^2 +  \frac{8}{\sigma}\left(1 + \frac{\tau}{\frac{2+ \tau}{4 \tau} } \left(\frac{1+ \tau^2}{4 \tau^2} \right)\right)^2 C_f \frac{1}{(1 + \frac{\tau}{2})^k} \frac{L}{2}\|\bar{y}^0 - x^*\|^2 \\
    &  + \frac{8}{\sigma}\frac{(1 - \tau)^2}{\left(1 + \frac{\alpha \sigma}{4}\right)^2}C_f\frac{1}{\left(1 + \frac{\tau}{2}\right)^{k-1}}\frac{L}{2}\|\bar{y}^0 - x^*\|^2 \\
    \leq & \frac{1}{2} C_x \frac{1}{\left(1 + \frac{\tau}{2}\right)^{k+1}}\|\bar{x}^0 - x^*\|^2 +  \frac{4L}{\sigma}\left(1  + \frac{1+ \tau^2}{2+\tau}
    \right)^2 C_f \frac{1}{\left(1 + \frac{\tau}{2}\right)^{k}}\|\bar{x}^0 - x^*\|^2 \\
    &  + \frac{5L}{\sigma}\frac{(1 - \tau)^2}{\left(1 + \frac{\alpha \sigma}{4}\right)^2}C_f\frac{1}{(1 + \frac{\tau}{2})^{k}}\|\bar{x}^0 - x^*\|^2 \\
     \leq & \frac{1}{2} C_x \frac{1}{\left(1 + \frac{\tau}{2}\right)^{k+1}}\|\bar{x}^0 - x^*\|^2 + \frac{L}{\sigma}(9 + 5 )C_f\frac{1}{\left(1 + \frac{\tau}{2}\right)^{k}}\|\bar{x}^0 - x^*\|^2  \\
     \leq & C_x \frac{1}{\left(1 + \frac{\tau}{2}\right)^{k+1}} \|\bar{x}^{0} - x^*\|^2,
\end{align*}
where the first inequality uses $\norm{a + b}^2 \leq 2\norm{a}^2 + 2\norm{b}^2$;
the second inequality follows from  $\alpha = \frac{t_0 \eta}{2 \tau} = \frac{2 \tau}{\sigma}$ and the initialization of $\bar{x}^0 = \bar{y}^0$.

Recall that $\bar{q}^{k+1}= \frac{1}{\frac{1}{\alpha} + \frac{\sigma}{4}}\left(\frac{\bar{q}^k}{\alpha} + \left(\frac{1}{t_0\eta} + \frac{\sigma}{4}\right)\bar{y}^{k+1} - \frac{1}{t_0\eta}\bar{x}^{k+1}\right)$
Assuming that $\|\bar{q}^{k} - x^*\|^2 \leq C_q \frac{1}{\left(1 + \frac{\tau}{2}\right)^{k}} \|\bar{q}^{0} - x^*\|^2$ with $C_q =\frac{1750L}{\sigma \tau^4}$, we have:
\begin{align*}
   & \|\bar{q}^{k+1} - x^*\|^2 \\
    = & \frac{1}{\left(\frac{1}{\alpha} + \frac{\sigma}{4}\right)^2}\left(\frac{1}{\alpha^2}\|\bar{q}^k - x^*\|^2 + \left\|\left(\frac{1}{t_0 \eta} + \frac{\sigma}{4}\right)\left(\bar{y}^{k+1} - x^*\right) - \frac{1}{t_0 \eta}(\bar{x}^{k+1} - x^*)\right\|^2 \right.\\ 
   & \left. + 2 \Bigl\langle \frac{1}{\alpha} (\bar{q}^k - x^*), \left(\frac{1}{t_0 \eta} + \frac{\sigma}{4}\right)(\bar{y}^{k+1} - x^*) - \frac{1}{t_0 \eta}(\bar{x}^{k+1} - x^*) \Bigr\rangle \right) \\
   \leq & \frac{1}{\left(\frac{1}{\alpha} + \frac{\sigma}{4}\right)^2}\left(\frac{1}{\alpha^2}\|\bar{q}^k - x^*\|^2 + \norm{\left(\frac{1}{t_0 \eta} + \frac{\sigma}{4}\right)(\bar{y}^{k+1} - x^*) - \frac{1}{t_0 \eta}(\bar{x}^{k+1} - x^*)}^2 \right. \\ 
   & \left. + \frac{\tau}{4\alpha^2}\|\bar{q}^k - x^*\|^2 + \frac{4}{\tau}\norm{\left(\frac{1}{t_0 \eta} + \frac{\sigma}{4}\right)(\bar{y}^{k+1} - x^*) - \frac{1}{t_0 \eta}(\bar{x}^{k+1} - x^*)}^2\right) \\
   \leq & \frac{1 + \frac{\tau}{4}}{\left(1 + \frac{\alpha \sigma}{4}\right)^2} \|\bar{q}^k - x^*\|^2 + 2\frac{1 + \frac{4}{\tau}}{\left(\frac{1}{\alpha} + \frac{\sigma}{4}\right)^2}\left(\frac{1}{t_0 \eta} + \frac{\sigma}{4}\right)^2\|\bar{y}^{k+1} - x^*\|^2 + 2\frac{1 + \frac{4}{\tau}}{\left(\frac{1}{\alpha} + \frac{\sigma}{4}\right)^2}\frac{1}{t_0^2 \eta^2}\|\bar{x}^{k+1} - x^*\|^2\\
    \leq & \frac{1 + \frac{\tau}{4} }{\left(1 + \frac{\tau}{2}\right)^2} \|\bar{q}^k - x^*\|^2 + 2\frac{1 + \frac{4}{\tau}}{\left(\frac{1}{2\tau} + \frac{1}{4}\right)^2}\left(\frac{1}{4 \tau^2} + \frac{1}{4}\right)^2\|\bar{y}^{k+1} - x^*\|^2 + 2\frac{1 + \frac{4}{\tau}}{\left(\frac{1}{2\tau} + \frac{1}{4}\right)^2}\frac{1}{16 \tau^4}\|\bar{x}^{k+1} - x^*\|^2\\
    \leq & \frac{1 + \frac{\tau}{4} }{\left(1 + \frac{\tau}{2}\right)^2} \|\bar{q}^k - x^*\|^2 + 2\frac{16\tau(\tau+4)}{(\tau+2)^2}\frac{1}{4 \tau^4}\|\bar{y}^{k+1} - x^*\|^2 + 2\frac{16\tau(\tau+4)}{(\tau+2)^2}\frac{1}{16 \tau^4}\|\bar{x}^{k+1} - x^*\|^2 \\
     \leq & \frac{1 + \frac{\tau}{4} }{\left(1 + \frac{\tau}{2}\right)^2} \|\bar{q}^k - x^*\|^2 + \frac{2}{\sigma}\frac{8(\tau+4)}{(\tau+2)^2}\frac{1}{ \tau^3}\left(F(\bar{y}^{k+1}) - F(x^*)\right) + \frac{2(\tau+4)}{(\tau+2)^2}\frac{1}{ \tau^3} \frac{C_x}{\left(1 + \frac{\tau}{2}\right)^{k+1}}\|\bar{x}^{0} - x^*\|^2 \\
     \leq & \frac{1 + \frac{\tau}{4} }{\left(1 + \frac{\tau}{2}\right)^2} \|\bar{q}^k - x^*\|^2 + \frac{4}{\sigma}{(\tau+4)}\frac{1}{ \tau^3}\frac{C_f}{\left(1 + \frac{\tau}{2}\right)^{k+1}}(F(\bar{y}^{0}) - F(x^*)) + \frac{(\tau+4)}{2}\frac{1}{ \tau^3}\frac{C_x}{\left(1 + \frac{\tau}{2}\right)^{k+1}}\|\bar{x}^{0} - x^*\|^2 \\
     \leq &  C_q \frac{1 + \frac{\tau}{4} }{\left(1 + \frac{\tau}{2}\right)^{k+2}}\|\bar{q}^0 - x^*\|^2 + \frac{2L}{\sigma}{(\tau+4)}\frac{1}{ \tau^3}\frac{C_f}{\left(1 + \frac{\tau}{2}\right)^{k+1}}\|\bar{y}^{0} - x^*\|^2 + \frac{(\tau+4)}{2}\frac{1}{ \tau^3}\frac{C_x}{\left(1 + \frac{\tau}{2}\right)^{k+1}}\|\bar{x}^{0} - x^*\|^2 \\
     \leq & C_q \frac{1 + \frac{\tau}{4} }{\left(1 + \frac{\tau}{ 2}\right)^{k+2}}\|\bar{q}^0 - x^*\|^2 + \frac{350L}{\sigma}\frac{1}{ \tau^3}\frac{1}{\left(1 + \frac{\tau}{2}\right)^{k+1}}\|\bar{y}^{0} - x^*\|^2 \\
     \leq & C_q \frac{1}{\left(1 + \frac{\tau}{2}\right)^{k+1}} \|\bar{q}^{0} - x^*\|^2,
\end{align*}
where the first inequality uses Young inequality and the third inequality follows from  $\alpha = \frac{t_0 \eta}{2 \tau} = \frac{2 \tau}{\sigma}$.
\end{proof}
\begin{remark}
The condition $\rho \leq  \min\left( \sqrt{\frac{1}{64(1 + {b} + \frac{1}{\tau^4})}}, \sqrt{\frac{1}{64 m J b (\frac{k}{\sigma} + \frac{8 k L}{\sigma^2} + \frac{16 k L}{\sigma^2 \tau^4} + \eta t_0)}} \right)$ can be further reduced to $\rho \leq  \min\left( \sqrt{\frac{1}{64(1 + {b} + \frac{1}{\tau^4})}}, \sqrt{\frac{1}{64 m J b (\frac{k}{\sigma} + \frac{8 k L}{\sigma^2} + \frac{16 k L}{\sigma^2 \tau^4} )}} \right)$ if $\min(\sqrt{\ell_1 \ell_2} / \sigma, L / \sigma) > \sqrt{n}$, i.e., the conditional number of the problem is larger than the instance size.
\end{remark}
\section{Centralized SVRG for the Sum-of-Nonconvex Problem}\label{appendix:svrg_analysis}
In this section, we present the theoretical result of SVRG on the sum-of-nonconvex problem.
Although \citet{allen2018katyusha} provided a similar convergence result, our proof is based on a different quantity as pointed out by the remark below.
\begin{lemma}[\cite{allen2018katyusha}]\label{lem:SVRG}
If the gradient estimator $\tilde{\nabla}_t$ satisfies with $\tilde{\nabla}_t=\nabla f (w_t)$ and $\mathbb{E}\left[\norm{\tilde{\nabla}_t -\nabla f (w_t)}^2\right] \leq Q \norm{w_0 - w_t}^2$ for every $t$ and some universal constant $Q$, then as long as $\eta \leq \min\{\frac{1}{2L}, \frac{1}{2\sqrt{Q m}}\}$, it satisfies 
\begin{equation}
    \mathbb{E}[F(w^+) - F(u)] \leq \mathbb{E}\left[-\frac{1}{4 m \eta} \norm{w^+ - w^0}^2 + \frac{\langle w^0 - w^+, w^0 - u  \rangle}{m \eta} - \frac{\sigma_f + \sigma_{\psi}}{4}\norm{w^+ - u}^2\right].
\end{equation}
\end{lemma}
\begin{theorem}
Let $y^k$ be the $k$-th iterate in the outer loop of the centralized SVRG on the sum-of-nonconvex problem, we have: 
\begin{equation}
    \mathbb{E}[F(y^k) - F(y^0)] \leq \mathbb{E}\left[ \frac{1+\frac{1}{m \eta \sigma}}{(1 + \frac{m \eta \sigma}{2})^{k+1}} \left(  f(y^0) - f(x^{*})  \right) \right].
\end{equation}
\end{theorem}
\begin{proof}
Using Lemma \ref{lem:SVRG} with $u=y^k$, $w^+=y^{k+1}$ and $w^0=y^k$, we have
\begin{align*}
     \mathbb{E}\left[\left(\frac{1}{4 m \eta} + \frac{\sigma}{4}\right) \norm{y^{k+1} - y^k}^2 \right]  \leq \mathbb{E} \left[  F(y^k) - F(y^{k+1})\right].
\end{align*}
Using Lemma \ref{lem:SVRG} with with $u=x^*$, $w^+=y^{k+1}$ and $w^0=y^k$, we have
\begin{align*}
      \mathbb{E}[F(y^{k+1}) - F(x^*)] \leq \mathbb{E}&\left[-\frac{1}{4 m \eta} \norm{y^{k+1} - y^k}^2 + \frac{\langle y^k - y^{k+1}, y^k - x^*  \rangle}{m \eta} - \frac{\sigma}{4}\norm{y^{k+1} - x^*}^2 \right] \\
      \leq  \mathbb{E}&\left[-\frac{1}{4 m \eta} \norm{y^{k+1} - y^k}^2 + \frac{\norm{y^k - y^{k+1}}^2 + \norm{y^k - x^*}^2 - \norm{y^{k+1} - x^*}^2}{2 m \eta} \right. \\
      &  \left. - \frac{\sigma}{4}\norm{y^{k+1} - x^*}^2\right] \\
       \leq \mathbb{E}&\left[\frac{1}{4 m \eta} \norm{y^{k+1} - y^k}^2 + \frac{  \norm{y^k - x^*}^2 }{2 m \eta} 
      - \left(\frac{\sigma}{4} + \frac{1}{2 m \eta}\right)\norm{y^{k+1} - x^*}^2\right] \\
       \leq  \mathbb{E}&\left[\frac{\frac{1}{4 m \eta}}{\frac{1}{4 m \eta} + \frac{\sigma}{4}} \left(  f(y^k) - f(y^{k+1}) \right) + \frac{  \norm{y^k - x^*}^2 }{2 m \eta} 
      - \left(\frac{\sigma}{4} + \frac{1}{2 m \eta}\right)\norm{y^{k+1} - x^*}^2\right] \\
       \leq \mathbb{E}&\left[\frac{1}{1 + m \eta \sigma} \left(  f(y^k) - f(x^{*}) \right) + \frac{  \norm{y^k - x^*}^2 }{2 m \eta} 
      - \left(\frac{\sigma}{4} + \frac{1}{2 m \eta}\right)\norm{y^{k+1} - x^*}^2\right] \\
      \leq \mathbb{E}&\left[\frac{1}{1 + \frac{m \eta \sigma}{2} } \left(  f(y^k) - f(x^{*}) \right) + \frac{  \norm{y^k - x^*}^2 }{2 m \eta} 
      - \left(\frac{\sigma}{4} + \frac{1}{2 m \eta}\right)\norm{y^{k+1} - x^*}^2\right] \\
      \leq \mathbb{E}&\left[ \frac{1}{(1 + \frac{m \eta \sigma}{2})^{k+1}} \left(  f(y^0) - f(x^{*})  + \frac{\norm{y^0 - x^*}^2}{2 m \eta}\right) \right] \\
      \leq \mathbb{E}&\left[ \frac{1+\frac{1}{m \eta \sigma}}{(1 + \frac{m \eta \sigma}{2})^{k+1}} \left(  f(y^0) - f(x^{*})  \right) \right] . 
\end{align*}
\end{proof}
\begin{remark}
Note that the convergence rate of SVRG on the strongly convex problem was proved by \cite{allen2018katyusha}. The difference between our convergence  bound and their results is that our convergence result is based on $\mathbb{E}[F(y^{k+1}) - F(x^*)]$ while theirs give the bound on $F(\bar{y}) - F(x^*)$ where $\bar{y}$ is the weighted average of all the past iterates. In the next section, we will extend SVRG on the sum-of-nonconvex problem in the decentralized setting by extending the proof used in this section.
\end{remark}
\section{Theoretical Analysis of PMGT-SVRG}\label{appendix:pmgt_svrg_analysis}
\label{sec:pmgt_svrg}
In this section, we present the theoretical result of PMGT-SVRG on the sum-of-nonconvex problem.
We first show the PMGT-SVRG algorithm below. 
Note that it is different from the PMGT-SVRG proposed by \citet{ye2020pmgt} in the following two aspects:
(1) Our algorithm samples a minibatch of component functions at each inner iteration;
(2) We maintain the inner-outer loop structure of the vanilla SVRG algorithm \cite{johnson2013accelerating} while \citet{ye2020pmgt} extended the loopless SVRG \cite{kovalev2020don} to the decentralized setting.
\begin{algorithm}[!ht]
\textbf{Inputs}: $x_i^0$, $y_i^0$, $y_i^{-1}$ and $q_i^0$. \\
\textbf{Parameters}:{FastMix parameter $M$, Functions $f_i, \psi$ distributed on each agent, starting vector $x_0$, minibatch size $b \in [n]$, learning rate
$\eta \geq 0$, momentum Parameter $\tau_{k} \in (0, 1]$, $K$ number of epochs}\\
\textbf{Outputs}: $x_i^T$,  $y_i^{T}$ and $q_i^T$.
\begin{algorithmic}[1]
\STATE In parallel, for each agent $i$,
Initialize $x_i^0 = x_0, y_i^0 = x_0, y_i^{-1} = x_0, q_i^{0} = x_0$
\FOR{$k = 0, \ldots, K-1$}
    \STATE Update the local variable $\nu_i^{k+1} = \hat{s}_i^{k} + \nabla f_i (y_i^{k}) - \nabla f_i (y_i^{k-1})$   
    \STATE Update $\hat{\mathbf{s}}^{k+1} = \texttt{FastMix}(\mathbf{\nu}^{k+1}, M)$   
    \STATE Update $\vw^0 = \texttt{FastMix}(\mathbf{y}^k, M)$  
    \STATE Set $s_i^{-1} = v_i^{-1} = \hat{s}_i^{k+1}, \quad \mu_i = \hat{s}_i^{k+1}$, \quad $w_i^0 = y_i^{k}$ 
   \STATE Take $t_0 = \ceil{\frac{n}{b}}$, and sample $T \sim \texttt{Geom}\left(\frac{1}{t_0}\right)$    
    \FOR{$t = 0, \ldots, T$}
        \STATE Let $B_i^t$ be a batch of $b$ indices sampled uniformly from $[n]$ with replacement       
        \STATE update the local variable $v_i^t = \mu_i + \frac{1}{b}\sum_{j_i \in B_i^t}(\nabla f_{i,j_i}(w_i^t) - \nabla f_{i, j_i}(w_i^0))$       
        \STATE update the local gradient tracker $s_i^{t} = FastMix(\mathbf{s}^{t-1} + \mathbf{v}^{t} - \mathbf{v}^{t-1}, M)_i$     
        \STATE $w_i^{t+1} = \texttt{FastMix}(prox_{m \eta , \Psi}(w^t - \eta s^{t}), M)_i$
    \ENDFOR  
    \STATE Set $y_i^{k+1} = w_i^{T+1}$   
    \STATE Compute $q_i^{k+1} = \texttt{FastMix}\left(\frac{\mathbf{q}^k +  \frac{\tau}{2} \mathbf{y}^{k+1} - \frac{1}{2 \tau} (\mathbf{x}^{k+1} - \mathbf{y}^{k+1})}{1 + \frac{\tau}{2}}, M\right)_i$
\ENDFOR
\end{algorithmic}
\caption{PMGT-SVRG}
\label{svrg_algo_2}
\end{algorithm}

We first bound the consensus error in the outer loop by the following lemma.
\begin{lemma}
Let $\mathbf{S}^k$ and $\mathbf{y}^k$ be the $k$-th iterate of local variables in the Algorithm \ref{svrg_algo_2}, then we have:
\begin{equation}
    \|\hat{\mathbf{s}}^{k+1} - \mathbf{1} \bar{\hat{s}}^{k+1}\|^2  \leq  4\rho^2 \|\hat{\mathbf{s}}^{k} - \mathbf{1} \bar{\hat{s}}^{k}\|^2 +   {36 \rho^2 \ell_1 \ell_2}   \|\mathbf{y}^k - \mathbf{1} \bar{y}^k\|^2  + {14 \rho^2  m \ell_1 \ell_2} \norm{\bar{y}^{k+1} - \bar{y}^k}^2.
\end{equation}
For $\mathbf{y}^{k+1}$, it follows that:
\begin{equation}
     \mathbb{E}\left[ \norm{ \mathbf{y}^{k+1} - \mathbf{1} \bar{y}^{k+1}}^2 \right] 
    \leq  \mathbb{E}_T\left[ 2\rho^2 \left(1+ \frac{b}{2 t_0}\right)  \|\mathbf{y}^k - \mathbf{1} \bar{y}^k\|^2 +   4\rho^2 \eta^2 \|\hat{\mathbf{s}}^{k} - \mathbf{1} \bar{\hat{s}}^{k}\|^2   + \frac{3 \rho^2  m b }{4}\|\bar{y}^{k+1} - \bar{y}^k\|^2\right] .
\end{equation}
\end{lemma}
\begin{proof}
Recall that $\hat{\mathbf{s}}^{k+1} = \texttt{FastMix}(\mathbf{\nu}^{k+1}, M)$ in the Algorithm \ref{svrg_algo_2}, we have:  
\begin{align*}
    & \quad \|\hat{\mathbf{s}}^{k+1} - \mathbf{1} \bar{\hat{s}}^{k+1}\|^2 \\
    \leq & \rho^2 \|\mathbf{\nu}^{k+1} - \mathbf{1} \bar{\nu}^{k+1}\|^2 \\
    \leq & 2\rho^2 \|\hat{\mathbf{s}}^{k} - \mathbf{1} \bar{\hat{s}}^{k}\|^2 +   2\rho^2 \sum_{i=1}^m\|\nabla f_i (y_i^{k + 1}) - \nabla f_i (y_i^{k})\|^2 \\
     \leq & 2\rho^2 \|\hat{\mathbf{s}}^{k} - \mathbf{1} \bar{\hat{s}}^{k}\|^2 +   {2\rho^2 \ell_1 \ell_2} \|\mathbf{y}^{k + 1} - \mathbf{y}^{k}\|^2 \\
      \leq & 2\rho^2 \|\hat{\mathbf{s}}^{k} - \mathbf{1} \bar{\hat{s}}^{k}\|^2 +   {6 \rho^2 \ell_1 \ell_2} \|\mathbf{y}^{k + 1} - \mathbf{1} \bar{y}^{k+1} \|^2  + {6 \rho^2 \ell_1 \ell_2} \|\mathbf{y}^{k} - \mathbf{1} \bar{y}^{k} \|^2 + {6 \rho^2  m \ell_1 \ell_2} \norm{\bar{y}^{k+1} - \bar{y}^k}^2.
\end{align*}
Note that $\mathbf{y}^{k+1} = \mathbf{w}^{T+1}$ in Algorithm \ref{svrg_algo_2}.
By Eq. (\ref{inner_comm})  we have the following bound :
\begin{align*}
     \mathbb{E}\left[ \norm{ \mathbf{y}^{k+1} - \mathbf{1} \bar{y}^{k+1}}^2 \right] \leq  \mathbb{E}_T&\left[\frac{p \rho^2}{1-(1-p)\|A\|_2}\|\mathbf{y}^k - \mathbf{1} \bar{y}^k\|^2 + \frac{p \eta^2}{1-(1-p)\|A\|_2} \norm{\hat{\mathbf{s}}^{k+1} - \mathbf{1}\bar{\hat{s}}^{k+1} }^2 \right. \\
     &      + \frac{1-p}{1-(1-p)\|A\|_2}{16\rho^4\eta^2\ell_1\ell_2}\|\mathbf{y}^k - \mathbf{1}\bar{y}^k\|^2 \\
     & \left. + \frac{2-p}{1-(1-p)\|\mathbf{A}\|_2}{16\ell_1\ell_2m\rho^2\eta^2}\|\bar{y}^{k+1} - \bar{y}^k\|^2\right] \\
     \leq &  \rho^2 (1 + 16\rho^2\eta^2\ell_1\ell_2 t_0) \|\mathbf{y}^k - \mathbf{1} \bar{y}^k\|^2 +  \eta^2 \norm{\hat{\mathbf{s}}^{k+1} - \mathbf{1}\bar{\hat{s}}^{k+1} }^2  \\
     &    + {32 \rho^2 \ell_1 \ell_2 m t_0 \eta^2}\|\bar{y}^{k+1} - \bar{y}^k\|^2 \\
     \leq & \rho^2 (1 + 16\rho^2\eta^2\ell_1\ell_2 t_0) \|\mathbf{y}^k - \mathbf{1} \bar{y}^k\|^2 +  \eta^2 \norm{\hat{\mathbf{s}}^{k+1} - \mathbf{1}\bar{\hat{s}}^{k+1} }^2  \\
     &    + \frac{ \rho^2  m b }{2}\|\bar{y}^{k+1} - \bar{y}^k\|^2 .
\end{align*}
Therefore, we have: 
\begin{align*}
    & \quad \|\hat{\mathbf{s}}^{k+1} - \mathbf{1} \bar{\hat{s}}^{k+1}\|^2 \\
      \leq & 2\rho^2 \|\hat{\mathbf{s}}^{k} - \mathbf{1} \bar{\hat{s}}^{k}\|^2 +   {6 \rho^2 \ell_1 \ell_2} \|\mathbf{y}^{k + 1} - \mathbf{1} \bar{y}^{k+1} \|^2  + {6 \rho^2 \ell_1 \ell_2} \|\mathbf{y}^{k} - \mathbf{1} \bar{y}^{k} \|^2 + {6 \rho^2  m \ell_1 \ell_2} \norm{\bar{y}^{k+1} - \bar{y}^k}^2\\
      \leq & 2\rho^2 \|\hat{\mathbf{s}}^{k} - \mathbf{1} \bar{\hat{s}}^{k}\|^2 +   {6 \rho^4 \ell_1 \ell_2}  (1 + 16\rho^2\eta^2\ell_1\ell_2 t_0) \|\mathbf{y}^k - \mathbf{1} \bar{y}^k\|^2 + 6 \rho^2 \ell_1 \ell_2 \eta^2 \norm{\hat{\mathbf{s}}^{k+1} - \mathbf{1}\bar{\hat{s}}^{k+1} }^2  \\
     &    + {3 \rho^4 \ell_1 \ell_2 m b}\|\bar{y}^{k+1} - \bar{y}^k\|^2  + {6 \rho^2 \ell_1 \ell_2} \|\mathbf{y}^{k} - \mathbf{1} \bar{y}^{k} \|^2 + {6 \rho^2  m \ell_1 \ell_2} \norm{\bar{y}^{k+1} - \bar{y}^k}^2\\
       \leq & 2\rho^2 \|\hat{\mathbf{s}}^{k} - \mathbf{1} \bar{\hat{s}}^{k}\|^2 +   {6 \rho^4 \ell_1 \ell_2}  (1 + 16\rho^2\eta^2\ell_1\ell_2 t_0) \|\mathbf{y}^k - \mathbf{1} \bar{y}^k\|^2 + 6 \rho^2 \ell_1 \ell_2 \eta^2 \norm{\hat{\mathbf{s}}^{k+1} - \mathbf{1}\bar{\hat{s}}^{k+1} }^2  \\
     &      + {6 \rho^2 \ell_1 \ell_2} \|\mathbf{y}^{k} - \mathbf{1} \bar{y}^{k} \|^2 + {7 \rho^2  m \ell_1 \ell_2} \norm{\bar{y}^{k+1} - \bar{y}^k}^2.
\end{align*}
After rearrangement, we have
\begin{equation}
    \begin{split}
          \|\hat{\mathbf{s}}^{k+1} - \mathbf{1} \bar{\hat{s}}^{k+1}\|^2  \leq & 4\rho^2 \|\hat{\mathbf{s}}^{k} - \mathbf{1} \bar{\hat{s}}^{k}\|^2 +   {12 \rho^4 \ell_1 \ell_2}  (1 + 16\rho^2\eta^2\ell_1\ell_2 t_0) \|\mathbf{y}^k - \mathbf{1} \bar{y}^k\|^2  \\
          &+ {12 \rho^2 \ell_1 \ell_2} \|\mathbf{y}^{k} - \mathbf{1} \bar{y}^{k} \|^2 + {14 \rho^2  m \ell_1 \ell_2} \norm{\bar{y}^{k+1} - \bar{y}^k}^2\\
          \leq &  4\rho^2 \|\hat{\mathbf{s}}^{k} - \mathbf{1} \bar{\hat{s}}^{k}\|^2 +   {36 \rho^2 \ell_1 \ell_2}   \|\mathbf{y}^k - \mathbf{1} \bar{y}^k\|^2  + {14 \rho^2  m \ell_1 \ell_2} \norm{\bar{y}^{k+1} - \bar{y}^k}^2.
    \end{split}
\end{equation}
So it follows:
\begin{align*}
 & \mathbb{E}[ \norm{ \mathbf{y}^{k+1} - \mathbf{1} \bar{y}^{k+1}}^2 ] \\
    \leq  \mathbb{E}_T&[ 2\rho^2  \|\mathbf{y}^k - \mathbf{1} \bar{y}^k\|^2 +   4\rho^2 \eta^2 \|\hat{\mathbf{s}}^{k} - \mathbf{1} \bar{\hat{s}}^{k}\|^2 +   {36 \rho^2 \eta^2 \ell_1 \ell_2}   \|\mathbf{y}^k - \mathbf{1} \bar{y}^k\|^2  \\
    & + {14 \rho^2  m \eta^2 \ell_1 \ell_2} \norm{\bar{y}^{k+1} - \bar{y}^k}^2    + \frac{ \rho^2  m b }{2}\|\bar{y}^{k+1} - \bar{y}^k\|^2] \\
    \leq  \mathbb{E}_T&[ 2\rho^2 (1+ \frac{b}{2 t_0})  \|\mathbf{y}^k - \mathbf{1} \bar{y}^k\|^2 +   4\rho^2 \eta^2 \|\hat{\mathbf{s}}^{k} - \mathbf{1} \bar{\hat{s}}^{k}\|^2   + \frac{3 \rho^2  m b }{4}\|\bar{y}^{k+1} - \bar{y}^k\|^2] .
\end{align*}
\end{proof}
Similar to the convergence analysis of PMGT-KatyushaX, we construct the system of linear inequalities for the consensus error.
Let $\mathbf{z}_{svrg}^{k+1} = \begin{pmatrix}
      \norm{\mathbf{y}^{k+1} - \mathbf{1} \bar{y}^{k+1}}^2 \\
      \eta^2 \norm{\hat{\mathbf{s}}^{k+1} - \mathbf{1}\bar{\hat{s}}^{k+1} }^2
\end{pmatrix}$, then it satisfies the following recurrence:
   \begin{equation}
        {z}_{svrg}^{t+1} = \mathbf{A}_{svrg}{z}_{svrg}^t + {b}_{svrg}^t,
    \end{equation}
    where 
    \begin{align*}
    \quad \mathbf{A}_{svrg} &= 2\rho^2 \begin{pmatrix}
    1 + \frac{b}{2 t_0} & 2\\
    2 & 18 \eta^2 \ell_1 \ell_2 \\
    \end{pmatrix} \\
    \end{align*}
    and 
    \begin{align*} {b}^t&=\begin{pmatrix}
       \frac{3 \rho^2 m b}{4} \norm{\bar{y}^{k+1} - \bar{y}^k}^2 \\
       {14 \rho^2 m \eta^2 \ell_1 \ell_2 } \norm{\bar{y}^{k+1} - \bar{y}^k}^2
    \end{pmatrix}.
    \end{align*}
Expand ${z}_{svrg}^{t+1} = \mathbf{A}_{svrg}^{t+1}{z}^0 + \sum_{i=0}^t \mathbf{A}_{svrg}^{t-i}b_{svrg}^i$. Due to the initialization, ${z}_{svrg}^{t+1} = \mathbf{0}$. We can bound the norm of $\mathbf{A}_{svrg}$ and $b_{svrg}^i$ by the following inequalities:
\begin{align*}
    \norm{\mA_{svrg}}_2 \leq \norm{\mA_{svrg}}_F \leq 2 \rho^2 (5 + \frac{b}{t_0}) \leq \frac{1}{2}
\end{align*}
and 
\begin{align*}
    \norm{b_{svrg}^i} \leq \rho^2 m b \norm{\bar{y}^{i+1} - \bar{y}^i}^2.
\end{align*}
Then one has
\begin{align*}
   \norm{ {z}_{svrg}^{t+1}} \leq \sum_{i=0}^t \frac{\rho^2 m b}{2^{t-i}} \norm{\bar{y}^{i+1} - \bar{y}^i}^2.
\end{align*}
Now we can prove the main convergence theorem for the PMGT-SVRG as below.
\begin{theorem}
Let $F(\cdot)$ defined in (\ref{strongconvex_algo_2}) be a $\sigma$-strongly convex function, function $f(\cdot)$ is $L$-smooth, and each component function $f_{i,j}$ is $(\ell_1, \ell_2)$-smooth.  As long as $\rho = \min\left({\sqrt{\frac{1}{64b}}, \sqrt{\frac{1}{\frac{ 9 m J b}{\sigma t_0} + 128 k m J b (\eta t_0 + \frac{1}{\sigma}) }}} \right)$, then the outputs of PMGT-SVRG satisfy
\begin{align*}
       \mathbb{E}[F(\bar{y}^{k+1}) - F(x^*)]  &\leq \mathbb{E}\left[\frac{2 + \frac{1}{t_0 \eta \sigma} }{\left(1 + \frac{t_0 \eta \sigma}{4}\right)^{k+1}}\left(F(\bar{y}^0) - F(x^*)\right)\right] .
\end{align*}
\label{theorem_strong}
\end{theorem}
\begin{proof}
Recall that Lemma \ref{lemma_inner}, we have the following inequality at the end of the loop:
\begin{align*}
    \mathbb{E}\left[F(\bar{w}^{T+1}) - F(u)\right] \leq  \mathbb{E}&\left[-\frac{\|\bar{w}^{T+1} - \bar{w}^0\|^2}{4\eta t_0} + \frac{\langle \bar{w}^0 - u, \bar{w}^0 - \bar{w}^{T+1}\rangle}{t_0\eta} - \frac{\sigma_f + \sigma_{\psi}}{8}\| u - \bar{w}^{T+1}\|^2 \right.\\
    &  \left.  + \left( J + \frac{1}{2 \eta m  t_0}\right)\| \vw^0 - \mathbf{1}\bar{w}^0  \|^2 +  J \eta^2 \| \mathbf{s}^0 - \mathbf{1}\bar{s}^0  \|^2 \right].
\end{align*}
After replacing $\bar{w}^{T+1}$ with $\bar{y}^{k+1}$ and $\bar{w}^0$ with $\bar{y}^k$, it follows that:
\begin{equation}
\begin{split}
   \mathbb{E}\left[F(\bar{y}^{k+1}) - F(u)\right] \leq  \mathbb{E}&\left[-\frac{\|\bar{y}^{k+1} - \bar{y}^k\|^2}{4\eta t_0} + \frac{\langle \bar{y}^k - u, \bar{y}^k - \bar{y}^{k+1}\rangle}{t_0\eta} - \frac{\sigma_f + \sigma_{\psi}}{8}\| u - \bar{y}^{k+1}\|^2 \right. \\
    &   \left. + \left( J + \frac{1}{2 \eta m  t_0}\right)\| \mathbf{y}^k - \mathbf{1}\bar{y}^k  \|^2 +  J \eta^2 \| \hat{\mathbf{s}}^{k+1} - \mathbf{1}\bar{\hat{s}}^{k+1}  \|^2 \right] .
    \end{split}
    \label{svrg_inner}
\end{equation}

Replacing $u$ with $\bar{y}^k$ in Eq. (\ref{svrg_inner}), we have
\begin{equation}
\begin{split}
     \mathbb{E}&\left[\left(\frac{1}{4\eta t_0} + \frac{\sigma}{8}\right) \norm{\bar{y}^{k+1} - \bar{y}^k}^2\right]  \\
    \leq  \mathbb{E}&\left[F(\bar{y}^k) - F(\bar{y}^{k+1}) + \left( J + \frac{1}{2 \eta m  t_0} \right)\| \mathbf{y}^k - \mathbf{1}\bar{y}^k  \|^2  +    J \eta^2 \|\hat{\mathbf{s}}^{k+1} - \mathbf{1} \bar{\hat{s}}^{k+1}\|^2 \right] \\
    \leq  \mathbb{E}&\left[F(\bar{y}^k) - F(\bar{y}^{k+1}) + \left( J + \frac{1}{2 \eta m  t_0} \right)\| \mathbf{y}^k - \mathbf{1}\bar{y}^k  \|^2  +     4\rho^2 J \eta^2 \|\hat{\mathbf{s}}^{k} - \mathbf{1} \bar{\hat{s}}^{k}\|^2 \right. \\
    & \left. +   {36 \rho^2 \ell_1 \ell_2 J \eta^2 }   \|\mathbf{y}^k - \mathbf{1} \bar{y}^k\|^2  + {14 \rho^2  m \ell_1 \ell_2 J \eta^2 } \norm{\bar{y}^{k+1} - \bar{y}^k}^2 \right] \\
    \leq  \mathbb{E}&\left[F(\bar{y}^k) - F(\bar{y}^{k+1}) + \left( J + \frac{1}{2 \eta m  t_0} \right)\| \mathbf{y}^k - \mathbf{1}\bar{y}^k  \|^2  +     4\rho^2 J \eta^2 \|\hat{\mathbf{s}}^{k} - \mathbf{1} \bar{\hat{s}}^{k}\|^2 \right.\\
    & \left. +   {36 \rho^2 \ell_1 \ell_2 J \eta^2 }   \|\mathbf{y}^k - \mathbf{1} \bar{y}^k\|^2  + \frac{ \rho^2  m  J b }{4 t_0} \norm{\bar{y}^{k+1} - \bar{y}^k}^2 \right] \\
    \leq  \mathbb{E}&\left[F(\bar{y}^k) - F(\bar{y}^{k+1}) +  3 J \| \mathbf{z}_{svrg}^k \|  + \frac{ \rho^2  m  J b }{4 t_0} \norm{\bar{y}^{k+1} - \bar{y}^k}^2 \right].  
    \end{split}
    \label{svrg_recur_1}
\end{equation}
Replacing $u$ with $x^*$ in Eq. (\ref{svrg_inner}), we have
\begin{align*}
         \mathbb{E}\left[F(\bar{y}^{k+1}) - F(x^*)\right] \leq  \mathbb{E}&\left[-\frac{\|\bar{y}^{k+1} - \bar{y}^k\|^2}{4\eta t_0} + \frac{\langle \bar{y}^k - x^*, \bar{y}^k - \bar{y}^{k+1}\rangle}{t_0\eta} - \frac{\sigma}{8}\| x^* - \bar{y}^{k+1}\|^2 \right.\\
    &  \left.  + \left( J + \frac{1}{2 \eta m  t_0}\right)\| \mathbf{y}^k - \mathbf{1}\bar{y}^k  \|^2 +  J \eta^2 \| \hat{\mathbf{s}}^{k+1} - \mathbf{1}\bar{\hat{s}}^{k+1}  \|^2 \right]\\
     \leq \mathbb{E}&\left[-\frac{\|\bar{y}^{k+1} - \bar{y}^k\|^2}{4\eta t_0} +
    \frac{\norm{\bar{y}^k - x^*}^2 + \norm{ \bar{y}^k - \bar{y}^{k+1}}^2 - \norm{\bar{y}^{k+1} - x^*}^2}{2 t_0 \eta} \right. \\
    & \left. - \frac{\sigma}{8}\| x^* - \bar{y}^{k+1}\|^2 + \left( J + \frac{1}{2 \eta m  t_0}\right)\| \mathbf{y}^k - \mathbf{1}\bar{y}^k  \|^2  +  J \eta^2 \| \hat{\mathbf{s}}^{k+1} - \mathbf{1}\bar{\hat{s}}^{k+1}  \|^2 \right]\\
      \leq \mathbb{E}&\left[\frac{\|\bar{y}^{k+1} - \bar{y}^k\|^2}{4\eta t_0} +
    \frac{\norm{\bar{y}^k - x^*}^2 }{2 t_0 \eta}  - \left(\frac{\sigma}{8} + \frac{1}{2 t_0 \eta}\right)\| x^* - \bar{y}^{k+1}\|^2 \right.\\
    &  + 2  J \| \mathbf{y}^k - \mathbf{1}\bar{y}^k  \|^2  +  4\rho^2 J \eta^2\|\hat{\mathbf{s}}^{k} - \mathbf{1} \bar{\hat{s}}^{k}\|^2 +   {36 \rho^2 \ell_1 \ell_2 J \eta^2}   \|\mathbf{y}^k - \mathbf{1} \bar{y}^k\|^2 \\
     & \left. + {14 \rho^2 m \ell_1 \ell_2 J \eta^2} \norm{\bar{y}^{k+1} - \bar{y}^k}^2 \right]\\
       \leq \mathbb{E}&\left[\frac{\|\bar{y}^{k+1} - \bar{y}^k\|^2}{4\eta t_0} +
    \frac{\norm{\bar{y}^k - x^*}^2 }{2 t_0 \eta}  - \left(\frac{\sigma}{8} + \frac{1}{2 t_0 \eta}\right)\| x^* - \bar{y}^{k+1}\|^2 \right. \\
    &  \left. + 3  J \norm{z_{svrg}^{k}}  + \frac{\rho^2 m  J b}{4 t_0} \norm{\bar{y}^{k+1} - \bar{y}^k}^2 \right].
\end{align*}
Assume that $\norm{\bar{y}^{i+1} - \bar{y}^i}^2 \leq \frac{C_y \left(F(\bar{y}^0) - F(x^*)\right)}{(1 + \frac{t_0 \eta \sigma}{4})^i}$ for $i \in [ k-1]$, then it follows that
\begin{align*}
      \mathbb{E}&\left[F(\bar{y}^{k+1}) - F(x^*)\right] \\
     \leq \mathbb{E}&\left[\left(\frac{1}{4\eta t_0} + \frac{\rho^2 m  J b}{4 t_0}\right)\|\bar{y}^{k+1} - \bar{y}^k\|^2 +
    \frac{\norm{\bar{y}^k - x^*}^2 }{2 t_0 \eta}  - \left(\frac{\sigma}{8} + \frac{1}{2 t_0 \eta}\right)\| x^* - \bar{y}^{k+1}\|^2 \right. \\
    & \left. + 3  J \sum_{i=0}^{k-1} \frac{\rho^2 m b}{2^{k-1-i}} \frac{C_y \left(F(\bar{y}^0) - F(x^*)\right)}{\left(1 + \frac{t_0 \eta \sigma}{4}\right)^i} \right]\\
    \leq \mathbb{E}&\left[\frac{\frac{1}{4\eta t_0} + \frac{\rho^2 m  J b}{4 t_0}}{\frac{1}{4\eta t_0} + \frac{\sigma}{8} -\frac{\rho^2 m  J b}{4 t_0}}\left(F(\bar{y}^k) - F(x^*) + 3J \norm{\mathbf{z}_{svrg}^k}\right) +
    \frac{\norm{\bar{y}^k - x^*}^2 }{2 t_0 \eta}  - \left(\frac{\sigma}{8} + \frac{1}{2 t_0 \eta}\right)\| x^* - \bar{y}^{k+1}\|^2 \right.\\
    & \left. + 3  J \sum_{i=0}^{k-1} \frac{\rho^2 m b}{2^{k-1-i}} \frac{C_y \left(F(\bar{y}^0) - F(x^*)\right)}{\left(1 + \frac{t_0 \eta \sigma}{4}\right)^i} \right]\\
    \leq \mathbb{E}&\left[\frac{\frac{1}{4\eta t_0} + \frac{\rho^2 m  J b}{4 t_0}}{\frac{1}{4\eta t_0} + \frac{\sigma}{8} -\frac{\rho^2 m  J b}{4 t_0}}\left(F(\bar{y}^k) - F(x^*) \right) +
    \frac{\norm{\bar{y}^k - x^*}^2 }{2 t_0 \eta}  - \left(\frac{\sigma}{8} + \frac{1}{2 t_0 \eta}\right)\| x^* - \bar{y}^{k+1}\|^2 \right.\\
    & \left. + 3  J  \frac{\frac{1}{2\eta t_0} + \frac{\sigma}{8}}{\frac{1}{4\eta t_0} + \frac{\sigma}{8} -\frac{\rho^2 m  J b}{4 t_0}} \sum_{i=0}^{k-1} \frac{\rho^2 m b}{2^{k-1-i}} \frac{C_y \left(F(\bar{y}^0) - F(x^*)\right)}{\left(1 + \frac{t_0 \eta \sigma}{4}\right)^i} \right]\\
     \leq \mathbb{E}&\left[\frac{1}{1 + \frac{t_0 \eta \sigma}{4}}\left(F(\bar{y}^k) - F(x^*) \right) +
    \frac{\norm{\bar{y}^k - x^*}^2 }{2 t_0 \eta}  - \left(\frac{\sigma}{8} + \frac{1}{2 t_0 \eta}\right)\| x^* - \bar{y}^{k+1}\|^2 \right. \\
    & \left. + 6    \sum_{i=0}^{k-1} \frac{\rho^2 m J b}{2^{k-1-i}} \frac{C_y (F(\bar{y}^0) - F(x^*))}{(1 + \frac{t_0 \eta \sigma}{4})^i} \right]\\
     \leq \mathbb{E}&\left[\frac{1}{(1 + \frac{t_0 \eta \sigma}{4})^{k+1}}\left(F(\bar{y}^0) - F(x^*)  + \frac{\norm{\bar{y}^0 - x^*}^2 }{2 t_0 \eta}\right) +   {6\rho^2 k m J b} \frac{C_y (F(\bar{y}^0) - F(x^*))}{\left(1 + \frac{t_0 \eta \sigma}{4}\right)^{k-1}} \right]\\
      \leq \mathbb{E}&\left[\frac{1}{(1 + \frac{t_0 \eta \sigma}{4})^{k+1}}\left(F(\bar{y}^0) - F(x^*)  + \frac{F(\bar{y}^0) - F(x^*) }{ t_0 \eta \sigma}\right) +   {10\rho^2 k m J b} \frac{C_y \left(F(\bar{y}^0) - F(x^*)\right)}{\left(1 + \frac{t_0 \eta \sigma}{4}\right)^{k+1}} \right]\\
       \leq \mathbb{E}&\left[\frac{1 + \frac{1}{t_0 \eta \sigma} + 10 \rho^2 k m J b C_y}{\left(1 + \frac{t_0 \eta \sigma}{4}\right)^{k+1}}\left(F(\bar{y}^0) - F(x^*)\right)\right].
\end{align*}
 Rearranging Eq. (\ref{svrg_recur_1}), if we choose  $C_y = 8\eta  t_0 + \frac{8}{\sigma}$,then we have
 \begin{align*}
      &  \mathbb{E}\left[\left(\frac{1}{4\eta t_0} + \frac{\sigma}{16}\right) \norm{\bar{y}^{k+1} - \bar{y}^k}^2\right]  \\
      \leq & \mathbb{E}\left[F(\bar{y}^k) - F(x^*) + 3J \norm{\mathbf{z}_{svrg}^k}\right] \\
       \leq & \mathbb{E}\left[\frac{1 + \frac{1}{t_0 \eta \sigma} + 10 \rho^2 k m J b C_y}{\left(1 + \frac{t_0 \eta \sigma}{4}\right)^{k}}\left(F(\bar{y}^0) - F(x^*)\right) + {3\rho^2 k m J b} \frac{C_y \left(F(\bar{y}^0) - F(x^*)\right)}{\left(1 + \frac{t_0 \eta \sigma}{4}\right)^{k-1}} \right] \\
        \leq& \mathbb{E}\left[\frac{1 + \frac{1}{t_0 \eta \sigma} + 14 \rho^2 k m J b C_y}{\left(1 + \frac{t_0 \eta \sigma}{4}\right)^{k}}\left(F(\bar{y}^0) - F(x^*)\right)\right] \\
        \leq & \frac{ \left(\frac{1}{4\eta t_0} + \frac{\sigma}{16}\right) C_y \left(F(\bar{y}^0) - F(x^*)\right)}{\left(1 + \frac{t_0 \eta \sigma}{4}\right)^{k+1}} \\
        = &\frac{ \frac{C_y}{4\eta t_0}  \left(F(\bar{y}^0) - F(x^*)\right)}{\left(1 + \frac{t_0 \eta \sigma}{4}\right)^{k}}.
 \end{align*}
\end{proof}
\section{Extension to General Convex Function}\label{appendix:general_convex_ext}
In this section, we adapt the theoretical result of PMGT-KatyushaX from the strongly convex case to the general convex setting. 
Since the non-strongly convex objective function means $\sigma = 0$, we cannot directly apply Theorem~\ref{theorem_strong_convex}. 
To circumvent this issue, we define the approximate function as follows: 
\begin{equation}
    F_{\epsilon}(x) \coloneqq f(x) + \psi(x) + \frac{\epsilon_f}{2}\norm{x}^2,
\end{equation}
where $\epsilon_f>0$. It is easy to verify that $F_{\epsilon}(\cdot)$ is  $\epsilon_f$-strongly convex.
Let $x_{\epsilon}^*$ be the minimizer of function $F_{\epsilon}(\cdot)$. 
From the above deduction, it follows that
\begin{theorem}
Let $F(\cdot)$ defined in (\ref{strongconvex_algo_2}) be a $L$-smooth and convex function, each component function $f_{i,j}$ is $(\ell_1, \ell_2)$-smooth. Additionally, we assume that the underlying network matrix $W$ is doubly stochastic, and it satisfies the properties defined in Definition (\ref{doubly_stochastic}).
If we choose $\epsilon_f = \min(\epsilon, \epsilon/\norm{x_{\epsilon}^*}^2)$ with the desired accuracy $\epsilon$,
then applying Algorithm \ref{strongconvex_algo_2} on $F_{\epsilon}(\cdot)$ achieves the gradient complexity of
\begin{align*}
    \mathcal{O}\left(\Big(n + \frac{L^{\frac{1}{2}} + (\ell_1 \ell_2)^\frac{1}{4}}{\sqrt{\epsilon}}n^{\frac{3}{4}}\Big) \log \frac{F(\bar{y}^0 - F(x_{\epsilon}^*))}{\epsilon}\right)
\end{align*}
 and communication complexity of  
\begin{align*}
\begin{split}
      \tilde{\mathcal{O}}\left( \frac{ \left(\sqrt{n} + \frac{L^{\frac{1}{2}} + (\ell_1 \ell_2)^\frac{1}{4}}{\sqrt{\epsilon}}  n^{\frac{1}{4}} \right)  }{\sqrt{1 - \lambda_2(W)}} \log^2 \frac{F(\bar{y}^0 - F(x^*))}{\epsilon} \right).
\end{split}
\end{align*}
\end{theorem}
\begin{proof}
Recall that we define the approximate function as follows: 
\begin{equation}
    F_{\epsilon}(x) \coloneqq f(x) + \psi(x) + \frac{\epsilon_f}{2}\norm{x}^2,
\end{equation}
where $\epsilon_f>0$. By Theorem~\ref{theorem_strong_convex}, we have 
\begin{align*}
    & \mathbb{E}\left[F_{\epsilon}(\bar{y}^{k}) - F_{\epsilon}(x_{\epsilon}^*) \right] \leq \frac{3}{\left( 1 + \frac{\tau}{2} \right)^{k}} \left( F_{\epsilon}(\bar{y}^0) - F_{\epsilon}(x_{\epsilon}^*) \right).
\end{align*}
Then one has 
\begin{align*}
 & \mathbb{E}\left[ F_{\epsilon}(\bar{y}^{k}) \right] \leq F_{\epsilon}(x_{\epsilon}^*) + \frac{3}{\left( 1 + \frac{\tau}{2} \right)^{k}} \left( F_{\epsilon}(\bar{y}^0) - F_{\epsilon}(x_{\epsilon}^*) \right).
 \end{align*}
It is equivalent to
\begin{align*}
      & \mathbb{E}\left[ F(\bar{y}^{k}) \right] \leq  \min_x\left(f(x) + \psi(x)\right) 
      + \frac{\epsilon_f}{2}\max(\norm{x_{\epsilon}^*}^2, 1) 
      + \frac{3}{\left( 1 + \frac{\tau}{2} \right)^{k}} \left( F_{\epsilon}(\bar{y}^0) - F_{\epsilon}(x_{\epsilon}^*) \right) .
\end{align*}
We can conclude that
\begin{align*}
     \mathbb{E}\left[ F(\bar{y}^{k}) - \min_x F(x) \right] \leq \frac{\epsilon_f}{2}\max(\norm{x_{\epsilon}^*}^2, 1)  
       + \frac{3}{\left( 1 + \frac{\tau}{2} \right)^{k}} \left( F_{\epsilon}(\bar{y}^0) - F_{\epsilon}(x_{\epsilon}^*) \right) .
\end{align*}
\end{proof}
\end{document}